\documentclass[11pt]{article}
\usepackage{latexsym,amssymb,amsmath,amsthm, amsfonts, float}
\usepackage[pdftex]{graphicx}
\usepackage{algorithmicx}
\usepackage{algpseudocode}
\usepackage{tikz}

\usepackage{bm}
\usepackage{graphicx}
\usepackage{stmaryrd}

\usepackage{enumerate}

\def\CC{{\mathbb C}}

\def\RR{{\mathbb R}}

\def\ZZ{{\mathbb Z}}
\def\cal{\mathcal}

\def\cA{{\cal A}}

\def\cE{{\cal E}}

\def\cG{{\cal G}}
\def\cH{{\cal H}}

\def\cN{{\cal N}}
\def\cO{{\cal O}}

\newcommand{\R}{\ensuremath{\mathbb{R}}}
\newtheorem{algorithm}{Algorithm}
\newtheorem{lemma}{Lemma}
\newtheorem{theorem}{Theorem}
\newtheorem{remark}{Remark}
\newtheorem{proposition}{Proposition}
\newtheorem{definition}{Definition}

\newcommand{\pd}{\partial}

\newcommand{\y}{\mathbf{y}}
\renewcommand{\v}{\mathbf{v}}
\newcommand{\I}{\mathbf{I}}
\renewcommand{\H}{\mathbf{H}}
\renewcommand{\u}{\mathbf{u}}
\newcommand{\f}{\mathbf{f}}

\newcommand{\x}{\mathbf{x}}

\newcommand{\beq}{\begin{equation}}
\newcommand{\eeq}{\end{equation}}
\renewcommand{\tilde}{\widetilde}

\newcommand{\bit}{\begin{itemize}}
\newcommand{\eit}{\end{itemize}}
\newcommand{\ben}{\begin{enumerate}}
\newcommand{\een}{\end{enumerate}}
\title{The method of polarized traces for the 2D Helmholtz equation}

\author{Leonardo Zepeda-N\'u\~nez and Laurent Demanet \\
Department of Mathematics and Earth Resources Laboratory, \\
Massachusetts Institute of Technology }

\date{October 2014}

\begin{document}

\maketitle

\begin{abstract}

We present a solver for the 2D high-frequency Helmholtz equation in heterogeneous acoustic media, with online parallel complexity that scales optimally as $\cO(\frac{N}{L})$, where $N$ is the number of volume unknowns, and $L$ is the number of processors, as long as $L$ grows at most like a small fractional power of $N$. The solver decomposes the domain into layers, and uses transmission conditions in boundary integral form to explicitly define ``polarized traces", i.e., up- and down-going waves sampled at interfaces. Local direct solvers are used in each layer to precompute traces of local Green's functions in an embarrassingly parallel way (the offline part), and incomplete  Green's formulas are used to propagate interface data in a sweeping fashion, as a preconditioner inside a GMRES loop (the online part). Adaptive low-rank partitioning of the integral kernels is used to speed up their application to interface data. The method uses second-order finite differences. The complexity scalings are empirical but motivated by an analysis of ranks of off-diagonal blocks of oscillatory integrals. They continue to hold in the context of standard geophysical community models such as BP and Marmousi 2, where convergence occurs in 5 to 10 GMRES iterations. While the parallelism in this paper stems from decomposing the domain, we do not explore the alternative of parallelizing the systems solves with distributed linear algebra routines.

\end{abstract}

\section{Introduction}\label{sec:introintro}

Many recent papers have shown that domain decomposition with accurate transmission boundary conditions is the right mix of ideas for simulating propagating high-frequency waves in a heterogeneous medium. To a great extent, the approach can be traced back to the AILU method of Gander and Nataf in 2001 \cite{GanderNataf:ailu_for_hemholtz_problems_a_new_preconditioner_based_on_an_analytic_factorization,Gander_Nataf:AILU_for_helmholtz_problems_a_new_preconditioner_based_on_the_analytic_parabolic_factorization}. The first linear complexity claim was perhaps made in the work of Engquist and Ying on sweeping preconditioners in 2011 \cite{EngquistYing:Sweeping_H, EngquistYing:Sweeping_PML} -- a special kind of domain decomposition into grid-spacing-thin layers. In 2013, Stolk \cite{CStolk_rapidily_converging_domain_decomposition} restored the flexibilty to consider coarser layerings, with a domain decomposition method that realizes interface transmission via an ingenious forcing term, resulting in linear complexity scalings very similar to those of the sweeping preconditioners.
Other authors have since then proposed related methods, including \cite{Chen_Xiang:a_source_transfer_ddm_for_helmholtz_equations_in_unbounded_domain,GeuzaineVion:double_sweep}, which we review in section \ref{sec:related}. Many of these references present isolated instances of what should eventually become a systematic understanding of how to couple absorption/transmission conditions with domain decomposition.

In a different direction, much progress has been made on making direct methods efficient for the Helmholtz equation. Such is the case of Wang et al.'s method \cite{Wang:H_multifrontal}, which couples multi-frontal elimination with $\cH$-matrices. Another example is the work of Gillman, Barnett and Martinsson on computing impedance-to-impedance maps in a multiscale fashion \cite{Gillman_Barnett_Martinsson:A_spectrally_accurate_solution_technique_for_frequency_domain_scattering_problems_with_variable_media}.  It is not yet clear whether offline linear complexity scalings can be achieved this way, though good direct methods are often faster in practice than the iterative methods mentioned above. The main issue with direct methods is the lack of scalability to very large-scale problems due to the memory requirements.

The method presented in this paper is a hybrid: it uses legacy direct solvers locally on large subdomains, and shows how to properly couple those subdomains with accurate transmission in the form of an incomplete Green's representation formula acting on polarized (i.e., one-way) waves. The novelty of this paper is twofold:
\bit
\item We show how to reduce the discrete Helmholtz equation to an integral system at interfaces in a way that does not involve Schur complements, but which allows for polarization into one-way components;
\item We show that when using fast algorithms for the application of the integral kernels, the online time complexity of the numerical method can become \emph{sublinear in $N$} in a parallel environment, i.e., $\cO(N/L)$ over $L \ll N$ nodes\footnote{The upper bound on $L$ is typically of the form $\cO(N^{1/8})$, with some caveats. This discussion is covered in sections \ref{sec:intro_complexity} and \ref{sec:complexity}.} using a simple parallelization model \footnote{It is possible to improve the parallelization using distributed linear solvers such as SuperluDIST \cite{Xia:multifrontal}, Pardiso \cite{Schenk:pardiso1}, MUMPS \cite{Amestoy:MUMPS} among others.}.
\eit

The proposed numerical method uses a layering of the computational domain, and also lets $L$ be the number of layers. It morally reduces to a sweeping preconditioner when there are as many layers as grid points in one direction ($L = n \sim N^{1/2}$); and it reduces to an efficient direct method when there is no layering ($L = 1$). In both those limits, the online complexity reaches $\cO(N)$ up to log factors.

But it is only when the number of layers $L$ obeys $1 \ll L \ll N$ that the method's online asymptotic complexity scaling $\cO(N/L)$ is strictly better than $\cO(N)$, even without relying on distributed linear algebra solvers.

Finally, the method was designed to be modular; it can seamlessly integrate new advances in direct solvers such as \cite{Gillman_Barnett_Martinsson:A_spectrally_accurate_solution_technique_for_frequency_domain_scattering_problems_with_variable_media,Rouet_Li_Ghysels:A_distributed-memory_package_for_dense_Hierarchically_Semi-Separable_matrix_computations_using_randomization,Wang_Li_Sia_Situ_Hoop:Efficient_Scalable_Algorithms_for_Solving_Dense_Linear_Systems_with_Hierarchically_Semiseparable_Structures}, and in fast summation techniques \cite{Bebendorf:2008,Li_Yang_Martin_Ho_Ying:Butterfly_Factorization}. Moreover, it can be easily modified for different discretizations; as explained in the the sequel, the algorithm hinges on a discrete Green's representation formula, which can be computed, in theory, for any sparse linear system issued from a discretization of a linear PDE, via summation by parts.

\subsection{Polarization} \label{section:polarization}
Let $\x = (x,z)$, the equation that we consider is in this paper is
\begin{equation}
    -\triangle u(\x) - \omega^2 m(\x) u(\x) = f(\x), \qquad \text{for } \x \in \RR^2,
\end{equation}
with $m(\x) = 1$ outside a bounded neighborhood $\Omega$ of the origin, with Sommerfeld radiation condition (SRC)
\begin{equation}
    \left | \frac{\partial u}{\partial r} - i \omega u \right | = \cO(r), \qquad \text{as } r \rightarrow \infty,
\end{equation}
where $r = |\x|$. Equivalently, we may consider $\x \in \Omega$, and replace the SRC by arbitrarily accurate absorbing boundary conditions in the form of a perfectly matched layer (PML) \cite{Berenger:PML,Johnson:PML} outside $\Omega$.\footnote{In practice, we do not restrict the PML to uniform media, i.e., $m(\x) \neq 1$ in $\RR^2 \backslash \Omega$.}

In the context of scalar waves in an unbounded domain, we say that a wave is \emph{polarized} at an interface when it is generated by sources supported only on one side of that interface.

Polarization is key to localizing propagating waves. If, for instance, the Helmholtz equation is posed in a large domain $\Omega$, but with sources supported in some small subdomain $\Omega_1 \subset \Omega$, then the exterior unknowns can be eliminated to yield a local problem in $\Omega_1$ by an adequate choice of ``polarizing" boundary condition. Such a boundary condition can take the form $\frac{\pd u}{\pd n} = D u$ in $\partial \Omega_1$, where $D$ is the Dirichlet-to-Neumann map exterior to $\Omega_1$. In a finite difference framework, $D$ can be represented via a Schur complement of the unknowns in the exterior region $\Omega \, \backslash \, \Omega_1$. See figure \ref{fig:polarization_sketch} for a sample geometry where $\Omega_1 = \Omega^{\scriptsize \mbox{down}}$, the bottom half-plane.

A polarizing boundary condition at $\pd \Omega_1$ may be absorbing in some special cases (such as a uniform medium), but is otherwise more accurately described as the \emph{image} of the absorbing boundary condition at $\pd \Omega$ by reduction as outlined above. Polarized waves at $\pd \Omega_1$ -- as we defined them -- are not in general strictly speaking outgoing, because of the presence of possible scatterers in the exterior zone $\Omega \, \backslash \, \Omega_1$. These scatterers will generate reflected/incoming waves that need to be accounted for in the polarization condition at $\pd \Omega_1$.

We say that a polarization condition is exact when the reduction of unknowns is done in the entire exterior domain,
without approximation. In a domain decomposition framework, the use of such exact polarizing conditions would enable solving the Helmholtz equation in two sweeps of the domain -- typically a top-down sweep followed by a bottom-up sweep \cite{Gander_Kwok:optimal_interface_conditiones_for_an_arbitrary_decomposition_into_subdomains}.

However, constructing exact polarizing conditions is a difficult task. No matter the ordering, elimination of the exterior unknowns by direct linear algebra is costly and difficult to parallelize. Probing from random vectors has been proposed to alleviate this problem \cite{RosalieLaurent:compressed_PML}, but requires significant user intervention. Note that elimination of the \emph{interior} unknowns by Schur complements, or equivalently computation of an interior Dirichlet-to-Neumann map, may be interesting \cite{Gillman_Barnett_Martinsson:A_spectrally_accurate_solution_technique_for_frequency_domain_scattering_problems_with_variable_media}, but does not result in a polarized boundary condition. Instead, this paper explores a local formulation of approximate polarized conditions in integral form, as incomplete Green's identities involving local Green's functions.


To see how to express a polarizing condition in boundary integral form, consider Fig. \ref{fig:polarization_sketch}. Let $\mathbf{x} = (x,z)$ with $z$ pointing down. Consider a linear interface $\Gamma$ partitioning $\R^2$ as $\Omega^{\scriptsize \mbox{down}} \cup \Omega^{\scriptsize \mbox{up}} \cup \Gamma$, with $f$ compactly supported in $\Omega^{\scriptsize \mbox{down}}$. Suppose that $m$ is heteregeneous, but constant outside a large compact neighborhod of the origin.

 Let $\x \in \Omega^{\scriptsize \mbox{up}}$, and consider a contour made up of the boundary $\partial D$ (the gray contour in Fig. \ref{fig:polarization_sketch}) of a semi-disk $D$ of radius $R$ in the upper half-plane $\Omega^{\scriptsize \mbox{up}}$. We can then apply the  Green's representation formula (GRF) \cite{Kress:Linear_integral_equations,McLean:Strongly_elliptic_systems_and_boundary_integral_equations} on $D$ to obtain
 \beq\label{eq:GRF}
u(\x) = \int_{\partial D} \left( \frac{\pd G}{\pd z_y}(\x,\y) u(\y) - G(\x,\y) \frac{\pd u}{\pd z_y}(\y)  \right) dS_{\y}, \qquad \x \in \mathring{D},
\eeq
and zero if $\x$ lies in the semi-circle of radius $R$. Moreover, Eq. \ref{eq:GRF} can decomposed in
\begin{align} \label{eq:GRF_decomposed}
u(\x) &= \int_{ \Gamma_R} \left( \frac{\pd G}{\pd z_y}(\x,\y) u(\y) - G(\x,\y) \frac{\pd u}{\pd z_y}(\y)  \right) dS_{\y} \\
      &+ \int_{ \partial D_R} \left( \frac{\pd G}{\pd z_y}(\x,\y) u(\y) - G(\x,\y) \frac{\pd u}{\pd z_y}(\y)  \right) dS_{\y}, \qquad \x \in \mathring{D} ,
\end{align}
where $\Gamma_R = \Gamma \cap \partial D$, and  $\partial D_R$ is the portion of $\partial D$ in the interior of $\Omega^{\text{up}}$.

Letting $R \rightarrow \infty$, the integral in $\partial D_R$ vanishes due to the SRC, which results in the \emph{incomplete Green's formula}
\beq\label{eq:incomplete-GRF}
u(\x) = \int_{\Gamma} \left( \frac{\pd G}{\pd z_y}(\x,\y) u(\y) - G(\x,\y) \frac{\pd u}{\pd z_y}(\y)  \right) dS_{\y}, \qquad \x \in \Omega^{\scriptsize \mbox{up}}.
\eeq
On the contrary, if $\x$ approaches $\Gamma$ from below, then we obtain the \emph{annihilation formula}
\beq \label{eq:annihilation}
0 = \int_{\Gamma} \left( \frac{\pd G}{\pd z_y}(\x,\y) u(\y) - G(\x,\y) \frac{\pd u}{\pd z_y}(\y) \right) dS_{\y} , \qquad \x \to \Gamma,  \,\, \x \text{ in } \Omega^{\text{down}}.
\eeq
Eqs. \ref{eq:incomplete-GRF} and \ref{eq:annihilation} are equivalent; either one can be used as the definition of a polarizing boundary condition on $\Gamma$. This boundary condition is exactly polarizing if $G$ is taken as the global Green's function for the Helmholtz equation in $\R^2$ with SRC.

\begin{figure}[H]
	\begin{center}
	 	\includegraphics[trim = 8mm 80mm 8mm 10mm, clip, width = 125mm]{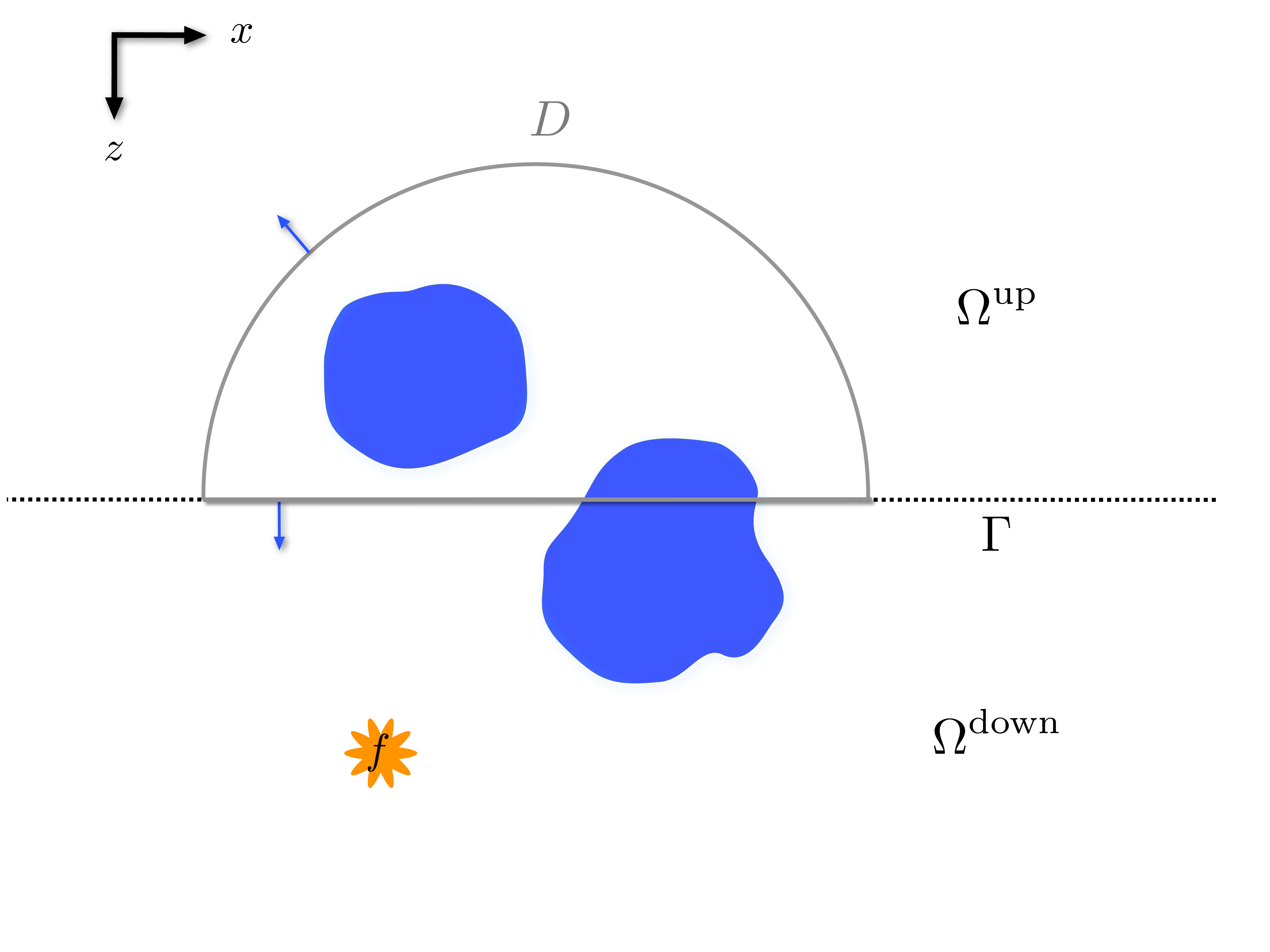}
	 	 \caption{Illustration of Eqs. \ref{eq:incomplete-GRF} and \ref{eq:annihilation}.}
	 	 \label{fig:polarization_sketch}
	 \end{center}
\end{figure}

Instead, this paper proposes to use these conditions with a \emph{local} Green's function $G$ that arises from a problem posed in a slab around $\Gamma$, with absorbing boundary conditions at the edges of the slab. With a local $G$, the conditions (Eq. \ref{eq:incomplete-GRF} or \ref{eq:annihilation}) are only approximately polarizing.

Other approximations of absorbing conditions, such as square-root operators, can be good in certain applications \cite{deHoop:generalization_of_the_phase_screen_approximation_for_the_scattering_of_acoustic_waves,Beylkin:oneway,Stoffa:split_step_fourier_migration, FishmandeHoop:exact_constructions_of_square_root_helmholtz_operator_symbols_the_focusing_quadratic_profile,StolkdeHoop:modeling_of_seismic_data_in_the_downward_continuation_approach}, but are often too coarse to be useful in the context of a fast domain decomposition solver. In particular, it should be noted that square-root operators do \emph{not} approximate polarizing conditions to arbitrary accuracy in media with heterogeneity normal to the boundary.

Finally, we point out that an analogous approach can be found for the homogeneous case under the name of Rayleigh integrals in Chapter 5 of \cite{Berkhout:Seismic_Migration_imaging_of_acoustic_energy_by_wavefield_extrapolation}, in which the polarizing condition is called causality condition; moreover, the approach presented above is equivalent to the upwards-propagating radiation
condition (UPRC) in \cite{Chandler-Wilde:Boundary_value_problems_for_the_Helmholtz_equation_in_a_half-plane}.

%
%
%

\subsection{Algorithm}
%
%
Let $\Omega$ be a rectangle in $\RR^2$,  and consider a layered partition of $\Omega$ into slabs, or layers $\{ \Omega^{\ell} \}_{\ell =1}^{L}$. The squared slowness $m(\x) = 1/c( \x)^2$, $\x = (x,z)$, is the only physical parameter we consider in this paper.
Define the global Helmholtz operator at frequency $\omega$ as
\begin{equation} \label{eq:Helmholtz}
\cH u  = \left( -\triangle  - m \omega^2  \right)  u \qquad  \text{in } \Omega,
\end{equation}
with an absorbing boundary condition on $\pd \Omega$, realized to good approximation by a perfectly matched layer surrounding $\Omega$. Let us define $f^{\ell}$ as the restriction of $f$ to $\Omega^{\ell}$, i.e., $f^{\ell} = f\chi_{\Omega^{\ell}}$. Define the local Helmholtz operators as
\begin{equation} \label{eq:local_Helmholtz}
\cH^{\ell} u  =  \left( -\triangle  - m \omega^2  \right)  u   \qquad  \text{in } \Omega^{\ell},
\end{equation}
with an absorbing boundary condition on $\pd \Omega^{\ell}$. Let $u$ be the solution to  $\cH u = f$. Using the local GRF, the solution can be written without approximation in each layer as
\begin{equation} \label{GRF}
u(\x) = G^{\ell}  f^{\ell}(\x) + \int_{\partial \Omega^{\ell}} \left( G^{\ell} (\x, \y) \partial_{\nu_y} u(\y) - \partial_{\nu_y}  G^{\ell} (\x, \y)  u(\y) \right) dS_{\y}
\end{equation}
for $\x \in \Omega^{\ell}$, where $G^{\ell}  f^{\ell}(\x) = \int_{\Omega^{\ell}} G^{\ell}(\x,\y) f^{\ell}(\y) d\y$ and $G^{\ell}(\x,\y) $ is the solution of  $\cH^{\ell} G^{\ell}(\x,\y) = \delta(\x-\y)$.

Denote $\Gamma_{\ell,\ell+1} = \partial \Omega^{\ell} \cap \partial \Omega^{\ell+1}$. Supposing that $\Omega^{\ell}$ are thin slabs either extending to infinity, or surrounded by a damping layer on the lateral sides, we can rewrite Eq. \ref{GRF} as
\begin{align}\notag \label{eq:GRF_slab}
	u(\x)  	&=  G^{\ell}  f^{\ell} (\x)  \\  \notag
 			&- \int_{\Gamma_{\ell-1,\ell}} G^{\ell}(\x, \x') \partial_z u(\x') dx'  + \int_{\Gamma_{\ell,\ell+1}} G^{\ell}(\x, \x') \partial_z u(\x') dx' \\
			&+  \int_{\Gamma_{\ell-1,\ell}}\partial_z G^{\ell}(\x, \x')  u(\x') dx'  - \int_{\Gamma_{\ell,\ell+1}}\partial_z G^{\ell}(\x, \x')  u(\x') dx'.
\end{align}

The knowledge of $u$ and $\pd_z u$ on the interfaces $\Gamma_{\ell, \ell+1}$ therefore suffices to recover the solution everywhere in $\Omega$.

We further split $u = u^{\uparrow} + u^{\downarrow}$ and $\pd_z u = \pd_z u^{\uparrow} + \pd_z u^{\downarrow}$ on $\Gamma_{\ell, \ell+1}$, by letting $( u^{\uparrow}, \pd_z u^{\uparrow})$ be polarized up in $\Omega^{\ell}$ (according to Eq. \ref{eq:annihilation}), and $( u^{\downarrow}, \pd_z u^{\downarrow})$ polarized down in $\Omega^{\ell+1}$. Together, the interface fields $u^{\uparrow}, u^{\downarrow}, \pd_z u^{\uparrow}, \pd_z u^{\downarrow}$ are the ``polarized traces" that serve as computational unknowns for the numerical method.

\begin{figure}[H]
	\begin{center}
	 	\includegraphics[trim = 8mm 10mm 8mm 15mm, clip, width = 65mm]{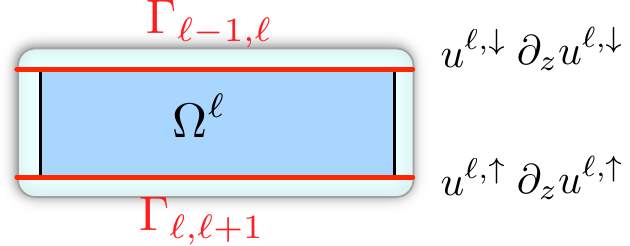}
	 	 \caption{Illustration of Eq. \ref{eq:GRF_slab}. The light-shaded layer around $\Omega^{\ell}$ represents the absorbing layer.}
	 	 \label{fig:PLR_compression}
	 \end{center}
\end{figure}

The discrete system is then set up from algebraic reformulations of the local GRF (Eq. \ref{eq:GRF_slab}) with the polarizing conditions (Eqs. \ref{eq:incomplete-GRF} and \ref{eq:annihilation}), in a manner that will be made explicit below. The reason for considering this sytem is twofold:
\begin{enumerate}
\item It has a 2-by-2 structure with block-triangular submatrices on the diagonal, and comparably small off-diagonal submatrices. A very good preconditioner consists in inverting the block-triangular submatrices by back- and forward-substitution. One application of this preconditioner can be seen as a sweep of the domain to compute transmitted (as well as locally reflected) waves using Eq. \ref{eq:incomplete-GRF}.
\item Each of these submatrices decomposes into blocks that are somewhat compressible in adaptive low-rank-partitioned format. This property stems from the polarization conditions, and would not hold for the Schur complements of the interior of the slab.
\end{enumerate}
Point 1 ensures that the number of GMRES iterations stays small and essentially bounded as a function of frequency and number of slabs. Point 2 enables the sublinear-time computation of each iteration.

\subsection{Complexity scalings}\label{sec:intro_complexity}


Time complexities that grow more slowly than the total number $N$ of volume unknowns are possible in a $L$-node cluster, in the following sense: the input consists of a right-hand-side $f$ distributed to $L$ nodes corresponding to different layers, while the problem is considered solved when the solution is locally known for each of the $L$ nodes/layers.

The Helmholtz equation is discretized as a linear system of the form $\H \u = \f$. There is an important distinction between
\bit
\item the offline stage, which consists of any precomputation involving $\H$, but not $\f$; and
\item the online stage, which involves solving $\H \u = \f$ for possibly many right-hand sides $\f$.
\eit
By online time complexity, we mean the runtime for solving the system once in the online stage. The distinction is important in situations like geophysical wave propagation, where offline precomputations are often amortized over the large number of system solves with the same matrix $\H$.


We let $n = N^{1/2}$ for the number of points per dimension (up to a constant). Since most of our complexity claims are empirical, our $O$ notation may include some log factors. All the numerical experiments in this paper assume the simplest second-order finite difference scheme for constant-density, heterogeneous-wave-speed, fixed-frequency acoustic waves in 2D -- a choice that we make for simplicity of the exposition. High frequency either means $\omega \sim n$ (constant number of points per wavelength), or $\omega \sim n^{1/2}$ (the scaling for which second-order FD are expected to be accurate), with different numerical experiments covering both cases.

The table below gathers the observed complexity scalings. We assume that the time it takes to transmit a message of length $m$ on the network is of the form $\alpha + \beta m$, with $\alpha$ the latency and $\beta$ the inverse bandwidth.

\begin{table}[H]
\begin{center}
\begin{tabular}{|c|c|c|c|c|}
\hline
Step 		&  Sequential  & Zero-comm parallel    & Number of processors & Communication cost \\
\hline
offline 	& - 		& $\cO(\left( N/L \right)^{3/2} )$  		& $ N^{1/2}L$  & $ \cO(\alpha L + \beta NL)$  \\
\hline
online			& $\cO(N^{\gamma} L) $  	& $\cO\left( N\log(N) /L \right)$ & $L$   &  $ \cO(\alpha L + \beta N^{1/2}L)$   \\
\hline
\end{tabular}
\caption{Time complexities for the different stages of the solver. The parameter $\gamma$ is striclty less than one; its value depends on the scaling of $\omega$ vs. $N$. A representative value is $\gamma = 3/4$ (more in the text).
}\label{table:complexity}
\end{center}
\end{table}

These totals are decomposed over the different tasks according to the following two tables.

\begin{table}[H]
\begin{center}
\begin{tabular}{|c|c|c|c|}
\hline
Step 		&  Sequential  & Zero-comm parallel    & Number of processors  \\
\hline
Factorization	& - 		& $\cO((N/L)^{3/2} )$	& $L$    \\
\hline
Computation of the Green's functions		& -   	& $\cO\left( N\log(N) /L \right)$ & $ L$  \\
\hline
Compression of the Green's functions        &  -    &  $ \cO(R_{\text{max}}N \log N)$ & $ N^{1/2}L$       \\
\hline
\end{tabular}
\caption{Complexity of the different steps within the offline computation. }\label{table:complexity_off_line}
\end{center}
\end{table}

\begin{table}[H]
\begin{center}
\begin{tabular}{|c|c|c|c|}
\hline
Step 		&  Sequential  & Zero-comm parallel    & Number of processors  \\
\hline
Backsubstitutions for the local solves	& - 		& $\cO\left( N\log(N) /L \right)$ 	& $ L$    \\
\hline
Solve for the polarized traces	&  $\cO(N^{\gamma} L)$ 	& - & $ 1$  \\
\hline
Reconstruction in the volume   &  -    &  $\cO\left( N\log(N) /L \right)$  & $L$       \\
\hline
\end{tabular}
\caption{Complexity of the different steps within the online computation. }\label{table:complexity_on_line}
\end{center}
\end{table}

The ``sweet spot" for $L$ is when $N^{\gamma} L \sim \frac{N}{L}$, i.e., when $L \sim N^{\frac{1-\gamma}{2}}$. The value of $\gamma$ depends on how the frequency $\omega$ scales with $n$:
\begin{itemize}
\item When $\omega \sim n^{1/2}$ ($\cO(n^{1/2})$ points per wavelength), the second-order finite difference scheme is expected to be accurate in a smooth medium. In that case, we observe $\gamma = \frac{5}{8}$ in our experiments, even when the medium is not smooth. Some theoretical arguments indicate, however, that this scaling may not continue to hold as $N \to \infty$, and would become $\gamma = \frac{3}{4}$ for values of $N$ that we were not able to access computationally. Hence we prefer to claim $\gamma = \frac{3}{4}$.
\item When $\omega \sim n$ ($\cO(1)$ points per wavelength), the second-order finite difference scheme is not pointwise accurate. Still, it is a preferred regime in exploration geophysics. In that case, we observe $\gamma = \frac{7}{8}$. This relatively high value of $\gamma$ is due to the poor quality of the discretization. For theoretical reasons, we anticipate that the scaling would again become $\gamma = \frac{3}{4}$ as $\omega \sim n$, were the quality of the discretization not an issue.
\end{itemize}

As the discussion above indicates, some of the scalings have heuristic justifications, but we have no mathematical proof of their validity in the case of heterogeneous media of limited smoothness. Counter-examples may exist. For instance, if the contrast increases between the smallest and the largest values taken on by $m$ -- the case of cavities and resonances -- it is clear that the constants in the $\cO$ notation degrade. See section \ref{sec:complexity} for a discussion of the heuristics mentioned above.

Further improvements in the complexity can, in principle, be achieved by using distributed linear algebra for each local solve using a number of nodes that grows with the number of degrees of freedom. Using the same hypothesis, the sweeping preconditioner \cite{EngquistYing:Sweeping_PML} can achieve sublinear runtime; however, it would require distributed linear algebra especially tailored to this particular problem such as \cite{Poulson_Engquist:a_parallel_sweeping_preconditioner_for_heteregeneous_3d_helmholtz}.

Although the present paper is concerned with the 2D Helmholtz equation, the algorithm introduced can be easily be extended for the 3D case. However, the complexity scalings would be completely different. The complexity for the local solves is known to be $\cO(N^{4/3}/L^2)$ (using a multifrontal method \cite{Duff_Reid:The_Multifrontal_Solution_of_Indefinite_Sparse_Symmetric_Linear,GeorgeNested_dissection} in a slab of dimensions $n\times n \times n/L$) and the same theoretical arguments in Section \ref{section:PLR} can be extended for the 3D case, which will provide a $\cO(N^{5/6})$ complexity for the application of the Green's integral compressed in partitioned low rank (PLR) form. A crude estimate would yield an overall online runtime $\cO(N^{4/3}/L^2 + LN^{5/6})$, which would result in a linear runtime algorithm provided that $L \sim N^{1/6}$. Notice however that the application of the Green's integrals using a butterfly algorithm \cite{Li_Yang_Martin_Ho_Ying:Butterfly_Factorization,Neil_Rokhlin:butterfly} can, in principle, be performed in $\cO(N^{2/3})$ time (up to logarithmic factors), which would, in theory, provide sublinear online runtimes.

\subsection{Additional related work} \label{sec:related}

In addition to the references mentioned on page 1, much work has recently appeared on domain decomposition and sweeping-style methods for the Helmholtz equation:
\begin{itemize}

\item The earliest suggestion to perform domain decomposition with transmission boundary conditions is perhaps to be found in the work of Despres \cite{ Despres:domain_decomposition_hemholtz}, which led, in joint work with Cessenat and Benamou, to the Ultra Weak Variational Formulation \cite{BenamouDespres:domain_decomposition,Cessenat:Application_dune_nouvelle_formulation_variationelle_aux_equations_dondes_harmonique,Cessenat_Despres:application_of_an_ultra_weak_variational_formulation_of_elliptic_pdes_to_the_2_d_helmholtz_problem,Cessenat_Despres:Using_Plane_Waves_as_Base_Functions_for_Solving_Time_Harmonic_Equations_with_the_Ultra_Weak_Variational_Formulation};

\item Their work spawned a series of new formulations within the framework of Discontinous Galerkin methods, such as the Trefftz formulation of Perugia et al. \cite{Perugia:trefft,Gittelson_Hipmair_Perugia:Trefftz}, and the discontinuous enrichment method of Farhat et al. \cite{Farhat:The_discontinuous_enrichment_method};

\item Simultaneously to the ultra weak formulation, the partition of unity method (PUM) by Babuska and Melenk \cite{babuska_melenk:partition_of_unity_method}, and the least-squares method by Monk and Wang \cite{Monk_Wang:A_least-squares_method_for_the_Helmholtz_equation:}, were developed.

\item Hiptmair et al. proposed the multi-trace formulation \cite{Hiptmair:multiple_traces_boundary_integral_formulation_for_hemholtz_transmission_problem}, which involves solving an integral equation posed on the interfaces of a decomposed domain, which is naturally well suited for operator preconditionning in the case of piece-wise constant medium;

\item Plessix and Mulder proposed a method similar to AILU in 2003 \cite{Plessix_Mulder:Separation_of_variable_preconditioner_for_iterativa_Helmholtz_solver};

\item More recently, Geuzaine and Vion explored approximate transmission boundary conditions \cite{GeuzaineVion:double_sweep,Vion_Rosalie_Demanet:A_DDM_double_sweep_preconditioner_for_the_Helmholtz_equation_with_matrix_probing_of_the_DtN_map} coupled with a multiplicative Schwartz iteration, to improve on traditional domain decomposition methods;

\item Chen and Xiang proposed another instance of efficient domain decomposition where the emphasis is on transferring sources from one subdomain to another \cite{Chen_Xiang:a_source_transfer_ddm_for_helmholtz_equations_in_unbounded_domain,Cheng_Xiang:A_Source_Transfer_Domain_Decomposition_Method_For_Helmholtz_Equations_in_Unbounded_Domain_Part_II_Extensions};


\item Luo et al. proposed a large-subdomain sweeping method, based on an approximate Green's function by geometric optics and the butterfly algorithm, which can handle transmitted waves in very favorable complexity \cite{Luo:Fast_Huygens_sweeping_methods_for_Helmholtz_equations_in_inhomogeneous_media_in_the_high_frequency_regime}; and a variant of this approach that can handle caustics in the geometric optics ansatz \cite{Qian_Luo_Burridge:Fast_Huygens_sweeping_methods_for_multiarrival_Greens_functions_of_Helmholtz_equations_in_the_high-frequency_regime};

\item Conen et al, developped an additive Schwarz iteration coupled with coarse grid preconditioner based on an eigenvalue decomposition of the DtN maps inside each subdomain \cite{Nataf:A_coarse_space_for_heterogeneous_Helmholtz_problems_based_on_the_Dirichlet-to-Neumann_operator};

\item Poulson et al. parallelized the sweeping preconditioners in 3D to deal with very large scale problems in geophysics \cite{Poulson_Engquist:a_parallel_sweeping_preconditioner_for_heteregeneous_3d_helmholtz}.  Tsuji et al. designed a spectrally accurate sweeping preconditioner for time-harmonic elastic waves \cite{Tsuji_Poulson:sweeping_preconditioners_for_elastic_wave_propagation}, and time-harmonic Maxwell equations \cite{Tsuji_engquist_Ying:A_sweeping_preconditioner_for_time-harmonic_Maxwells_equations_with_finite_elements,Tsuji_Ying:A_sweeping_preconditioner_for_Yees_finite_difference_approximation_of_time-harmonic_Maxwells_equations}.

\item Liu and Ying, while this paper was in review, developed a recursive version of the sweeping preconditioner in 3D that decreases the off-line cost to linear complexity \cite{Liu_Ying:Recursive_sweeping_preconditioner_for_the_3d_helmholtz_equation}, and another variant of the sweeping preconditioner based on domain decomposition coupled with an additive Schwarz preconditioner \cite{Liu_Ying:Additive_Sweeping_Preconditioner_for_the_Helmholtz_Equation};

\item Finally, an earlier version of the work presented in this paper can be found in \cite{Zepeda_Hewett_Demanet:Preconditioning_the_2D_Helmholtz_equation_with_polarized_traces}, in which the notion of polarized traces using Green's representation formula was first introduced, obtaining favorable scalings;

\end{itemize}


Some progress has also been achieved with multigrid methods, though the complexity scalings do not appear to be optimal in two and higher dimensions:

\begin{itemize}
\item Bhowmik and Stolk recently proposed new rules of optimized coarse grid corrections for a two level multigrid method \cite{Bhowmik_Stolk:A_multigrid_method_for_the_helmholtz_equation_with_optimized_coarse_grid_corrections};
\item Brandt and Livshits developed the wave-ray method \cite{Brandt_Livshits:multi_ray_multigrid_standing_wave_equations}, in which the oscillatory error components are eliminated by a ray-cycle, that exploits a geometric optics approximation of the Green's function;
\item Elman and Ernst use a relaxed multigrid method as a preconditioner for an outer Krylov iteration in \cite{Elman_Ernst:a_multigrid_method_discrete_helmholtz_equation};
\item Haber and McLachlan proposed an alternative formulation for the Hemholtz equation \cite{Haber:A_fast_method_for_the_solution_of_the_Helmholtz_equation}, which reduces the problem to solve an eikonal equation and an advection-diffusion-reaction equation, which can be solved efficiently by a Krylov method using a multigrid method as a preconditioner; and
\item Grote and Schenk \cite{Grote_Schenk:algebraic_multilever_preconditioner_Helmholtz_equation} proposed an algebraic multi-level preconditioner for the Helmholtz equation.
\end{itemize}

Erlangga et al. \cite{Erlangga:shifted_laplacian} showed how to implement a simple, although suboptimal, complex-shifted Laplace preconditioner with multigrid, which shines for its simplicity, and whose performance depends on the choice of the complex-shift. A suboptimal multigrid method with multilevel complex shifts was implemented in 2D and 3D by Cools, Reps and Vanroose \cite{Cools_Reps_Vanroose:A_new_level-dependent_coarse_grid_correction_scheme_for_indefinite_Helmholtz_problems}. Another variant of complex-shifted Laplacian method with deflation was studied by Sheikh et al.\cite{Sheikh_Lahaye_Vuik:On_the_convergence_of_shifted_Laplace_preconditioner_combined_with_multilevel_deflation}. The choice of the optimal complex-shift has been studied by Cools and Vanroose \cite{Cools_Vanroose:Local_Fourier_analysis_of_the_complex_shifted_Laplacian_preconditioner_for_Helmholtz_problems} and by Gander et al. \cite{Gander_Graham_Spence:largest_shift_for_complex_shitfted_Laplacian}. Finally, some extensions of the complex-shifted Laplacian preconditioner have been proposed recently by Cools and Vanroose \cite{Cools_Vanroose:Generalization_of_the_complex_shifted_Laplacian:_on_the_class_of_expansion_preconditioners_for_Helmholtz_problems}.

Chen et al. implemented a 3D solver using the complex-shifted Laplacian with optimal grids to minimize pollution effects in \cite{Chen:A_dispersion_minimizing_finite_difference_scheme_and_preconditioned_solver_for_the_3D_Helmholtz_equation}. Multigrid methods applied to large 3D examples have been implemented by Plessix \cite{Plessix:A_Helmholtz_iterative_solver_for_3D_seismic_imaging_problems}, Riyanti et al. \cite{Riyanti_Plessix:A_parallel_multigrid-based_preconditioner_for_the_3D_heterogeneous_high-frequency_Helmholtz_equation:} and, more recently, Calandra et al.  \cite{Calandra_Grattonn:an_improved_two_grid_preconditioner_for_the_solution_of_3d_Helmholtz}. A good review of iterative methods for the Helmholtz equation, including complex-shifted Laplace preconditioners, is in  \cite{Erlangga:Helmholtz}. Another review paper that discussed the difficulties generated by the high-frequency limit is \cite{Gander:why_is_difficult_to_solve_helmholtz_problems_with_classical_iterative_methods}.  Finally,  beautiful mathematical expositions of the Helmholtz equation and a good reviews are  \cite{Chandler_Graham_Langdon_Spence:Numerical_asymptotic_boundary_integral_methods_in_high-frequency_acoustic_scattering} and \cite{Moiola_Spence:Is_the_Helmholtz_Equation_Really_Sign_Indefinite}.

\subsection{Organization}

The discrete integral reformulation of the Helmholtz equation is presented in section 2. The polarization trick, which generates the triangular block structure at the expense of doubling the number of unknowns, is presented in section 3. The definition of the corresponding preconditioner is in section 4. Compression of the blocks by adaptive partitioned-low-rank (PLR) compression is in section 5. An analysis of the computational complexity of the resulting scheme is in section 6. Finally, section 7 presents various numerical experiments that validate the complexity scalings as $L$ and $\omega$ vary.

\section{Discrete Formulation}

In this section we show how to reformulate the discretized 2D Helmholtz equation, in an exact fashion, as a system for interface unknowns stemming from local Green's representation formulas.


\subsection{Discretization}

Let $ \Omega = (0,L_x) \times (0, L_z )$ be a rectangular domain. Throughout this paper, we pose Eq. \ref{eq:Helmholtz} with absorbing boundary conditions on $\pd \Omega$, realized as a perfectly matched layer (PML) \cite{Berenger:PML,Johnson:PML}. This approach consists of extending the computational domain with an absorbing layer, in which the differential operator is modified to efficiently damp the outgoing waves and reduce the reflections due to the truncation of the domain.

Let $\Omega^{ \text{ext}}  = (-\delta_{\text{pml}},L_x+\delta_{\text{pml}}) \times (-\delta_{\text{pml}}, L_z +\delta_{\text{pml}}) $ be the extended rectangular domain containing $\Omega$ and its absorbing layer. The Helmholtz operator in Eq. \ref{eq:Helmholtz} then takes the form
\begin{equation}
	\cH = -\partial_{xx} - \partial_{zz} - m\omega^2, \qquad  \text{ in } \Omega^{\text{ext}},
\end{equation}
in which the differential operators are redefined following
\begin{equation}
	\partial_x  \rightarrow \frac{1}{1+ i \frac{\sigma_x(\x)}{ \omega }} \partial_x, \qquad \partial_z  \rightarrow \frac{1}{1+ i \frac{\sigma_z(\x)}{ \omega } } \partial_z,
\end{equation}
and where $m$ is an extension\footnote{We assume that $m(x)$ is given to us in $\Omega^{\text{ext}}$. If it isn't, normal extension of the values on $\pd \Omega$ is a reasonable alternative.} of the squared slowness.
Moreover, $\sigma_x(\x)$ is defined as
\begin{equation}
	\sigma_x(\x) = 	\left \{\begin{array}{rl}
						\frac{C}{\delta_{\text{pml}}} \left (\frac{x } {\delta_{\text{pml}}} \right)^2, 		& \text{if } x \in (-\delta_{\text{pml}}, 0 ),\\
						0 				,																			& \text{if } x \in [ 0, L_x ],  \\
						\frac{C}{\delta_{\text{pml}}} \left (\frac{x - L_x }{\delta_{\text{pml}}}\right)^2, 	& \text{if } x \in ( L_x , L_x + \delta_{\text{pml}} ),\\
							\end{array} \right .
\end{equation}
and similarly for $\sigma_z(\x)$. We remark that $\delta_{\text{pml}}$ and $C$ can be seen as tuning parameters. In general, $\delta_{\text{pml}}$ goes from a couple of wavelengths in a uniform medium, to a large number independent of $\omega$ in a highly heterogeneous media; and $C$ is chosen to provide enough absorption.

With this notation we rewrite Eq. \ref{eq:Helmholtz} as
\begin{equation} \label{eq:Helmholtz_pml}
\cH u = f, \qquad \text{in } \Omega^{\text{ext}},
\end{equation}
with homogeneous Dirichlet boundary conditions ($f$ is the zero extended version of $f$ to $\Omega^{\text{ext}}$).

We discretize $\Omega$ as an equispaced regular grid of stepsize $h$, and of dimensions $n_x \times n_z$. For the extended domain $\Omega^{\text{ext}}$, we extend this grid by $ n_{\text{pml}}=\delta_{\text{pml}}/h$ points in each direction, obtaining a grid of size $ (2n_{\text{pml}}+n_x) \times (2n_{\text{pml}}+n_z$). Define $\x_{p,q} = (x_p, z_q) = (ph,qh)$.



We use the 5-point stencil Laplacian to discretize Eq. \ref{eq:Helmholtz_pml}. For the interior points $\x_{i,j} \in \Omega$, we have
\begin{equation} \label{eq:discrete_helmholtz_equation_by_point}
	\left( \H \u \right)_{p,q} = \frac{1}{h^2} \left ( -\u_{p-1,q}+2\u_{p,q} -\u_{p+1,q} \right) +  \frac{1}{h^2} \left ( -\u_{p,q-1}+2\u_{p,q} -\u_{p,q+1} \right) - \omega^2 m(\x_{p,q}).
\end{equation}
In the PML, we discretize
\begin{align}
	\alpha_x\partial_x (\alpha_x \partial_x u ) \qquad &\mbox{as} \qquad  \alpha_x(\x_{p,q})\frac{ \alpha(\x_{p+1/2,q})(\u_{p+1,q} - \u_{p,q}) - \alpha_x(\x_{p-1/2,q})( \u_{p,q} - \u_{p-1,q} )   }{h^2},\\
	\alpha_z\partial_z (\alpha_z \partial_z u ) \qquad &\mbox{as} \qquad  \alpha_z(\x_{p,q})\frac{ \alpha(\x_{p,q+1/2})(\u_{p,q+1} - \u_{p,q}) - \alpha_z(\x_{p,q-1/2})( \u_{p,q} - \u_{p,q-1} )   }{h^2}, \label{eq:pml_partial_z}
\end{align}
where
\begin{equation}
	\alpha_x(\x) =  \frac{1}{1+ i \frac{\sigma_x(\x)}{ \omega } }, \qquad \alpha_z(\x) =  \frac{1}{1+ i \frac{\sigma_z(\x)}{\omega } }.
\end{equation}

Finally, we solve
\begin{equation} \label{eq:discrete_Helmholtz_equation}
	\left( \H \u \right)_{p,q} = \mathbf{f}_{p,q}, \qquad (p,q) \in \llbracket -n_{\text{pml}}+1, n_x +  n_{\text{pml}} \rrbracket \times  \llbracket -n_{\text{pml}} + 1, n_z +  n_{\text{pml}} \rrbracket,
\end{equation}
for the global solution $\u$. We note that the matrix $\H$ in Eq. \ref{eq:discrete_Helmholtz_equation} is not symmetric, though there exists an equivalent symmetric (non-Hermitian) formulation of the PML \cite{EngquistYing:Sweeping_H,EngquistYing:Sweeping_PML}. While the symmetric formulation presents theoretical advantages\footnote{In Appendix \ref{appendix:quadratures} we use the symmetric formulation as a proof method, and provide the relationship between the Green's functions in the symmetric and unsymmetric cases. }, working with the non-symmetric form gives rise to less cumbersome discrete Green's formulas in the sequel.


\subsection{Domain Decomposition}

We partition $\Omega^{\text{ext}}$ in a layered fashion into $L$ subdomains $\{\Omega^{\ell} \}_{\ell=1}^L$, as follows. For convenience, we overload $\Omega^{\ell}$ for the physical layer, and for its corresponding grid.

We continue to let $\x_{p,q} = (x_p, z_q) = (ph,qh)$ in the global indices $p$ and $q$. The $\ell$-th layer is defined via a partition of the $z$ axis, by gathering all $q$ indices such that
\[
n_c^{\ell} < q \leq n_c^{\ell+1},
\]
for some increasing sequence $n_c^{\ell}$. The number of grid points in the $z$ direction in layer $\ell$ is
\[
n^{\ell} = n_c^{\ell+1} - n_c^{\ell}.
\]
(The subscript $c$ stands for cumulative.) Many of the expressions in this paper involve local indexing: we let $i = p$ in the $x$ direction, and $j$ such that
\[
q = n_c^{\ell} + j, \qquad 1 \leq j \leq n^{\ell},
\]
in the $z$ direction.

We hope that there is little risk of confusion in overloading $z_{j}$ (local indexing) for $z_{n_c^{\ell} + j}$ (global indexing). Similarly, a quantity like the squared slowness $m$ will be indexed locally in $\Omega^\ell$ via $m^{\ell}_{i,j} = m_{i,n^{\ell}_{\text{c}} + j}$. In the sequel, every instance of the superscript $\ell$ refers to restriction to the $\ell$-th layer.

\begin{figure}[H]
    \begin{center}
        \includegraphics[trim = 0mm 18mm 0mm 18mm, clip, width = 130mm]{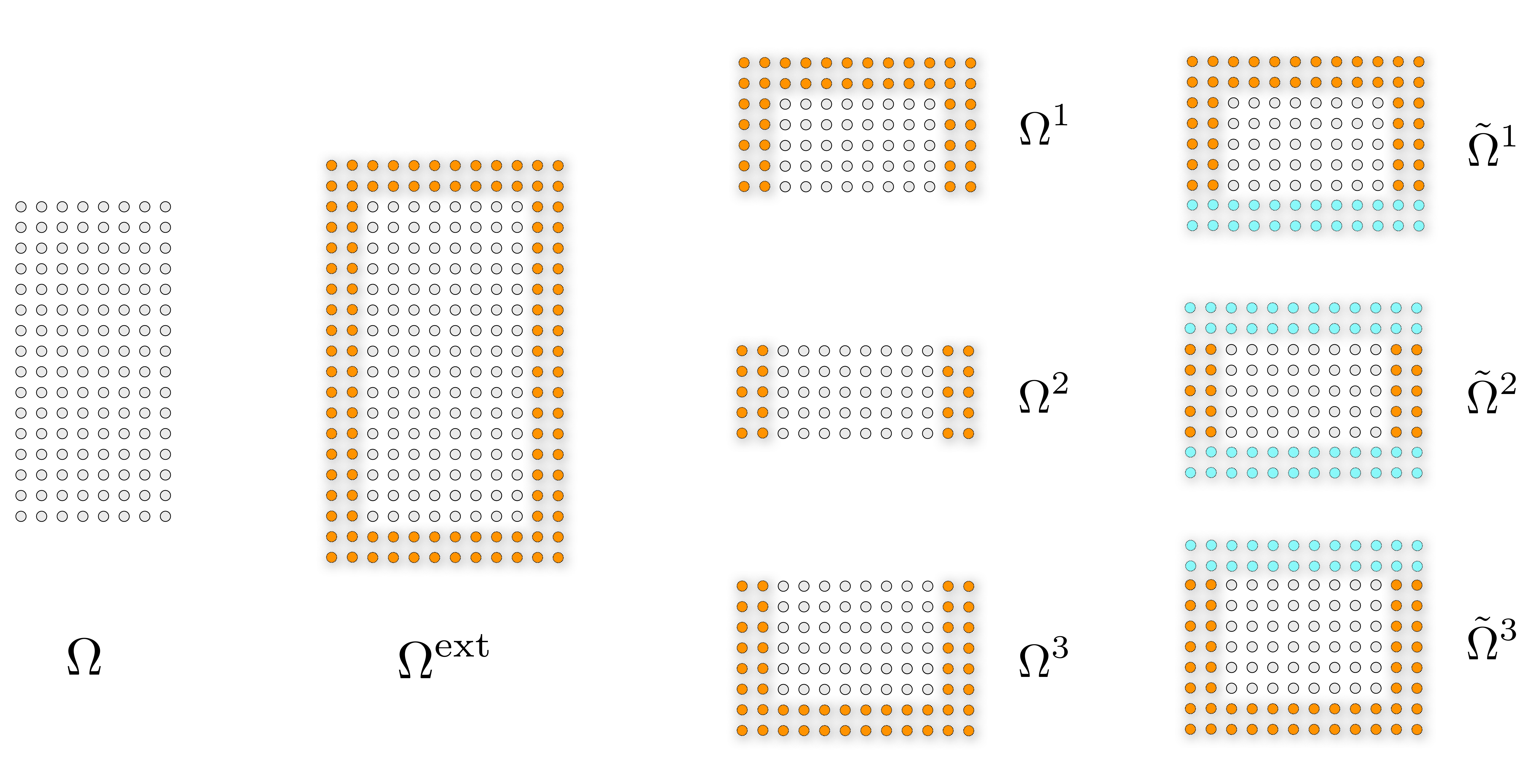}
         \caption{ The domain $\Omega$ is extended to $\Omega^{\text{ext}}$ by adding the PML nodes (orange). After decomposition into subdomains, the internal boundaries are padded with extra PML nodes (light blue). }
         \label{fig:DDM_sketch}
     \end{center}
\end{figure}

Absorbing boundary conditions in $z$ are then used to complete the definition of the Helmholtz subproblems in each layer $\Omega^{\ell}$. Define $\tilde{\Omega}^{\ell}$ as the extension of $\Omega^{\ell}$  in the $z$ direction by $n_{\text{pml}}$ points (see Fig. \ref{fig:DDM_sketch}), and define  $\tilde{m}^{\ell}$ as the normal extension of $m^{\ell}$ given by :
\[ \tilde{m}^{\ell}_{i,j} = \left \{ \begin{array}{rl}
							m^{\ell}_{i,1},   & j < 0, \\
							m^{\ell}_{i,j},   & 1 \leq j \leq n^{\ell},  \\
							m^{\ell}_{i,n^{\ell}}, & j>n^{\ell}.
						\end{array}
				\right. \]

The discretization inside $\Omega^\ell$ is otherwise inherited exactly from that of the global problem, via Eq. \ref{eq:discrete_helmholtz_equation_by_point}. Furthermore, on each PML the discrete differential operator is modified following Eq. \ref{eq:pml_partial_z}, using the extended local model $\tilde{m}$. The local problem is denoted as
\begin{equation} \label{eq:def_local_Green_function}
	\left( \H^{\ell} \u^{\ell} \right)_{i,j} = \mathbf{f}^{\ell}_{i,j}, \qquad (i,j) \in \llbracket -n_{\text{pml}}+1, n_x +  n_{\text{pml}} \rrbracket \times  \llbracket -n_{\text{pml}}+1, n^{\ell} +  n_{\text{pml}} \rrbracket,
\end{equation}
where $\mathbf{f}^{\ell}$ is extended by zero in the PML. The PML's profile is chosen so that the equations for the local and the global problem coincide exactly inside $ \llbracket -n_{\text{pml}}+1, n_x +  n_{\text{pml}} \rrbracket \times  \llbracket 1, n^{\ell} \rrbracket$.


\subsection{Discrete Green's Representation Formula}


In this section we show that solving the discretized PDE in the volume is equivalent to solving a discrete boundary integral equation from the Green's representation formula, with \emph{local} Green's functions that stem from the local Helmholtz equation. The observation is not surprising in view of very standard results in boundary integral equation theory.



We only gather the results, and refer the reader to Appendix  \ref{appendix:quadratures} for the proofs and details of the discrete Green's identities.

Define the numerical local Green's function in layer $\ell$ by
\begin{equation} \label{eq:definition_local_helmholtz_problem}
	 \H^{\ell} \mathbf{G}^{\ell}(\x_{i,j},\x_{i',j'}) =
	 \H^{\ell} \mathbf{G}^{\ell}_{i,j,i',j'}  = \delta(\x_{i,j}- \x_{i',j'} ), \qquad (i,j) \in \llbracket -n_{\text{pml}}+1, n_x +  n_{\text{pml}} \rrbracket \times  \llbracket -n_{\text{pml}}+1, n^{\ell} +  n_{\text{pml}} \rrbracket,
\end{equation}
where
\begin{equation}
\delta(\x_{i,j}- \x_{i',j'} ) =  \left \{ \begin{array}{rl}
											 \frac{1}{h^2}, &  \text{ if } \x_{i,j} = \x_{i',j'},\\
											  			0,  &  \text{ if } \x_{i,j} \neq \x_{i',j'},
										\end{array}
										\right .
\end{equation}
and where the operator $\mathbf{H}^\ell$ acts on the $(i,j)$ indices.

It is notationally more convenient to consider $\mathbf{G}^{\ell}$ as an operator acting on unknowns at horizontal interfaces, as follows. Again, notations are slighty overloaded.


\begin{definition} \label{def:layer_2_layer_Green_fucntion} We consider $\mathbf{G}^{\ell}(z_j, z_k)$ as the linear operator defined from $\llbracket -n_{\text{pml}}+1, n_x+n_{\text{pml}} \rrbracket \times \{ z_k \}$ to $\llbracket -n_{\text{pml}} + 1, n_x+n_{\text{pml}} \rrbracket \times \{z_j \}$ given by
	\begin{equation}
		\left ( \mathbf{G}^{\ell}(z_j, z_k) \v \right )_i = \, h \sum_{i'=-n_{\text{pml}+1} }^{n_x+n_{\text{pml}}} \mathbf{G}^{\ell}((x_i, z_j),(x_{i'}, z_k )) \v_{i'}  ,
	\end{equation}
	where $\v$ is a vector in $\CC^{n_x + 2n_{\text{pml}}}$, and $\mathbf{G}^{\ell}(z_j, z_k) $ are matrices in $\CC^{(n_x + 2n_{\text{pml}}) \times (n_x + 2n_{\text{pml}})}$.
\end{definition}

	Within this context we define the interface identity operator as
	\begin{equation}
		\left ( \mathbf{I} \v \right )_i =  h \v_i .
	\end{equation}

As in geophysics, we may refer to $z$ as depth. The layer to layer operator $\mathbf{G}^{\ell}(z_j, z_k)$ is indexed by two depths -- following the jargon from fast methods we call them source depth ($z_k$) and target depth ($z_j$).

\begin{definition} \label{def:incomplete_green} We consider $ \cG^{\uparrow, \ell}_j(\mathbf{v}_{n^{\ell}}, \mathbf{v}_{n^{\ell}+1} )$, the up-going local incomplete Green's integral; and $\cG^{\downarrow, \ell}_j(\mathbf{v}_{0}, \mathbf{v}_{1} )$, the down-going local incomplete Green's integral, as defined by:
	\begin{eqnarray}
		\cG^{\uparrow, \ell}_j(\mathbf{v}_{n^{\ell}}, \mathbf{v}_{n^{\ell}+1} ) &=& \mathbf{G}^{\ell}(z_j, z_{n^{\ell}+1})\left ( \frac{\mathbf{v}_{n^{\ell} + 1} -  \mathbf{v}_{n^{\ell}}}{h } \right)  - \left( \frac{\mathbf{G}^{\ell}(z_j, z_{n^{\ell}+1}) - \mathbf{G}^{\ell}(z_j, z_{n^{\ell}})}{h} \right) \mathbf{v}_{n^{\ell}+1}, \\
		\cG^{\downarrow, \ell}_j(\mathbf{v}_{0}, \mathbf{v}_{1} ) &=& -\mathbf{G}^{\ell}(z_j, z_{0})\left ( \frac{\mathbf{v}_{1} -  \mathbf{v}_{0}}{h } \right)  + \left( \frac{\mathbf{G}^{\ell}(z_j, z_{1}) - \mathbf{G}^{\ell}(z_j, z_{0})}{h} \right) \mathbf{v}_{0}.
	\end{eqnarray}
	In the sequel we use the shorthand notation $  \mathbf{G}^{\ell}(z_j, z_{k}) =  \mathbf{G}^{\ell}_{j,k}$ when explicitly building the matrix form of the integral systems.
\end{definition}

The incomplete Green's integrals in Def. \ref{def:incomplete_green} use the discrete counterparts of the single and double layer potentials. After some simplifications it is possible to express the incomplete Green's integrals in matrix form:
\begin{eqnarray}
	\cG^{\downarrow, \ell}_j(\mathbf{v}_{0}, \mathbf{v}_{1} )	 &=& \frac{1}{h} \left [  \begin{array}{cc}
		 \mathbf{G}^{\ell} (z_j,z_{1}) & - \mathbf{G}^{\ell}(z_j, z_{0})
	\end{array}
  \right]
  \left ( \begin{array}{c}
  			\v_{0} \\
  			\v_{1}
     	\end{array}
  \right), \\
\cG^{\uparrow, \ell}_j(\mathbf{v}_{n^{\ell}}, \mathbf{v}_{n^{\ell}+1} )  &=& \frac{1}{h}\left [  \begin{array}{cc}
				-\mathbf{G}^{\ell} (z_j,z_{n^{\ell}+1})  &   \mathbf{G}^{\ell}(z_j, z_{n^{\ell}})
			\end{array}
  \right]
  \left ( \begin{array}{c}
  			\v_{n^{\ell}} \\
  			\v_{n^{\ell}+1}
     	\end{array}
  \right).\label{eq:matrix_form_cG} \\
\end{eqnarray}

\begin{definition}  \label{def:Newton_potential}
	Consider the local Newton potential $\cN^{\ell}_k$ applied to a local source $\mathbf{f}^{\ell}$ as
	\begin{equation}
		\cN^{\ell}_k \mathbf{f}^{\ell} = \sum_{j=1}^{n^{\ell}} \mathbf{G}^{\ell}(z_k,z_j) \mathbf{f}^{\ell}_j.
	\end{equation}
	By construction $\cN^{\ell} \mathbf{f}^{\ell}$ satisfies the equation $\left ( \H^{\ell} \cN^{\ell} \mathbf{f}^{\ell}\right )_{i,j} = \mathbf{f}_{i,j}$ for  $-n_{\text{pml}} + 1 \leq i \leq n_x + n_{\text{pml}}$ and  $1 \leq j \leq n^{\ell}$.
\end{definition}

We can form local solutions to the discrete Helmholtz equation using the local discrete Green's representation formula, as
\begin{equation}  \label{eq:discrete_GRF}
	\v^{\ell}_{j} =	\cG^{\uparrow, \ell}_{j}(\v^{\ell}_{n^{\ell}}, \v^{\ell}_{n^{\ell}+1} )  + \cG^{\downarrow, \ell}_{j}(\v^{\ell}_{0}, \v^{\ell}_{1} ) + \cN^{\ell}_{j} \mathbf{f}^{\ell}, \qquad 1<j<n^{\ell}.
\end{equation}
This equation is, by construction, a solution to the local problem as long as $1<j<n^{\ell}$. It is easy to see that $\v^{\ell}_{j}$ is otherwise solution to
\begin{equation} \label{eq:PDE_version_green_formula}
	\mathbf{H}^{\ell} \v^{\ell} = \mathbf{f}^{\ell} + \delta(z_1-z) \v^{\ell}_0 - \delta(z_0-z)\v^{\ell}_1  - \delta(z_{n^{\ell}+1}-z) \v^{\ell}_{n^{\ell}} + \delta(z_{n^{\ell}}-z)\v^{\ell}_{n^{\ell}+1}.
\end{equation}
Notice that writing local solutions via Eq. \ref{eq:discrete_GRF} can be advantageous, since $\mathbf{G}^{\ell}$ only needs to be stored at interfaces. If a sparse LU factorization of $\mathbf{H}^{\ell}$ is available, and if the boundary data $\v^{\ell}_0, \v^{\ell}_1, \v^{\ell}_{n^{\ell}}, \v^{\ell}_{n^{\ell}+1}$ are available, then the computation of $\v^{\ell}$ is reduced to a local sparse solve with the appropriate forcing terms given in Eq. \ref{eq:PDE_version_green_formula}.

Eq. \ref{eq:discrete_GRF} is of particular interest when the data used to build the local solution are the traces of the global solution, as stated in the Lemma below. In other words, Eq. \ref{eq:discrete_GRF} continues to hold even when $j = 1$ and $j = n^{\ell}$.

\begin{lemma} \label{lemma:discrete_GRF} If  $\u^{\ell}_{0}, \u^{\ell}_{1},\u^{\ell}_{n^{\ell}},$ and $\u^{\ell}_{n^{\ell}+1}$  are the traces of the global solution at the interfaces, then
	\begin{equation}   \label{eq:discrete_GRF_inside}
		\u^{\ell}_{j} =	\cG^{\uparrow, \ell}_{j}(\u^{\ell}_{n^{\ell}}, \u^{\ell}_{n^{\ell}+1} )  + \cG^{\downarrow, \ell}_{j}(\u^{\ell}_{0}, \u^{\ell}_{1} ) + \cN^{\ell}_{j}  \mathbf{f}^{\ell},
	\end{equation}
	if $1\leq j \leq n^{\ell}$, where $\u^{\ell}$ is the global solution restricted to $\Omega^{\ell}$. Moreover,
	\begin{equation} \label{eq:discrete_GRF_outside}
		0 =	\cG^{\uparrow, \ell}_{j}(\u^{\ell}_{n^{\ell}}, \u^{\ell}_{n^{\ell}+1} )  + \cG^{\downarrow, \ell}_{j}(\u^{\ell}_{0}, \u^{\ell}_{1} ) + \cN^{\ell}_{j}  \mathbf{f}^{\ell},
	\end{equation}
	if $j\leq 0$ and $j\geq n^{\ell}+1$.
\end{lemma}
The proof of Lemma \ref{lemma:discrete_GRF} is given in Appendix \ref{appendix:green_representation}.
Thus, the problem of solving the global, discrete Helmholtz equation is reduced in an algebraically exact manner to the problem of finding the traces of the solution at the interfaces. Yet, the incomplete Green's operators are not the traditional Schur complements for $\Omega^{\ell}$.

What is perhaps most remarkable about those equations is the flexibility that we have in letting $\mathbf{G}^{\ell}$ be \emph{arbitrary} outside $\Omega^{\ell}$ -- flexiblity that we use by placing a PML.

\subsection{Discrete Integral Equation}



Evaluating the local GRF (Eq. \ref{eq:discrete_GRF_inside}) at the interfaces $j=1$ and $j=n^\ell$ of each layer $\Omega^\ell$ results in what we can call a \emph{self-consistency} relation for the solution. Together with a \emph{continuity} condition that the interface data must match at interfaces between layers, we obtain an algebraically equivalent reformulation of the discrete Helmholtz equation.

\begin{definition} \label{def:integral_formulation} Consider the discrete integral formulation for the Helmholtz equation as the system given by
	\begin{eqnarray}
		\cG^{\downarrow, \ell}_{1}(\u^{\ell}_{0}, \u^{\ell}_{1} ) 	+ \cG^{\uparrow, \ell}_{1}(\u^{\ell}_{n^{\ell}}, \u^{\ell}_{n^{\ell}+1} )  	+  \cN^{\ell}_{1}  \mathbf{f}^{\ell}  	&=& \u^{\ell}_{1} ,\\
		\cG^{\downarrow, \ell}_{n^{\ell}}(\u^{\ell}_{0}, \u^{\ell}_{1} ) 	+ \cG^{\uparrow, \ell}_{n^{\ell}}(\u^{\ell}_{n^{\ell}}, \u^{\ell}_{n^{\ell}+1})  	+  \cN^{\ell}_{n^{\ell}}  \mathbf{f}^{\ell}  	&=& \u^{\ell}_{n^{\ell}},\\
		\u^{\ell}_{n^{\ell}} 		&=& \u^{\ell+1}_{0}, \\
		\u^{\ell}_{n^{\ell}+1} 	&=& \u^{\ell+1}_{1},
	\end{eqnarray}
	if $1<\ell<L$, with
	\begin{eqnarray}
		\cG^{\uparrow, 1}_{n^1}(\u^{1}_{n^1}, \u^{1}_{n^{\ell}+1} ) + \cN^1_{n^1}  \mathbf{f}^1  &=& \u^{1}_{n^1},\\
		\u^{1}_{n^{1}} &=& \u^{2}_{0}, \\
		\u^{1}_{n^{1}+1} &=& \u^{2}_{1},
	\end{eqnarray}
	and
	\begin{eqnarray}
		\cG^{\downarrow, L}_{1}(\u^L_{0}, \u^L_{1} ) + \cN^L_{n^L}  \mathbf{f}^L  &=& \u^L_{1}, \\
		\u^{L-1}_{n^{L-1}} &=& \u^{L}_{0}, \\
		\u^{L-1}_{n^{L-1}+1} &=& \u^{L}_{1}.
	\end{eqnarray}
\end{definition}

	Following Def. \ref{def:incomplete_green} we can define
	\begin{equation}
		\scriptsize
		\mathbf{\underline{M}} = \frac{1}{h} \left [
									\begin{array}{ccccccc}
										-\mathbf{G}^1_{n,n+1} - \I	&  \mathbf{G}^1_{n,n} 		& 	0 							& 0 					& 0							& 0  \\
										\mathbf{G}^2_{1,1} 			& - \mathbf{G}^2_{1,0} - \I		& - \mathbf{G}^1_{1,n+1} 		& \mathbf{G}^1_{1,n}	& 0 					& 0  \\
										\mathbf{G}^2_{n,1} 			& - \mathbf{G}^2_{n,0}  		& - \mathbf{G}^2_{n,n+1} - \I	& \mathbf{G}^2_{n,n}	& 0 					& 0\\
										0 							&\ddots    						& \ddots				& \ddots 		 			 	& 0& 0 \\
										0 							& 0								& \mathbf{G}^{L-1}_{1,1} 	& - \mathbf{G}^{L-1}_{1,0}-\I	& - \mathbf{G}^{L-1}_{1,n+1} & \mathbf{G}^{L-1}_{1,n}	 	\\
										0 							& 0 							& \mathbf{G}^{L-1}_{n,1}  	& - \mathbf{G}^{L-1}_{n,0}  	& - \mathbf{G}^{L-1}_{n,n+1} - \I	& \mathbf{G}^{L-1}_{n,n}	\\
										0 							& 0 							&	0						&		0						& \mathbf{G}^{L}_{1,1} 	& - \mathbf{G}^{L}_{1,0} - \I
										\end{array}
								\right ],
	\end{equation}
	such that the discrete integral system can be written as
 	\begin{equation} \label{eq:integral_formulation}
 	\underline{\mathbf{M}} \; \underline{\u} = -	\left (
 														\begin{array}{c}
 															\cN^1_{n^1}\mathbf{f}^1 \\
 															\cN^2_1\mathbf{f}^2 \\
 															\cN^2_{n^2}\mathbf{f}^2 \\
 															\vdots \\
 															\cN^L_1\mathbf{f}^L
 														\end{array}
 													\right )  = - \underline{\mathbf{f}}.
 	\end{equation}


By $\underline{\u}$, we mean the collection of all the interface traces. As a lemma, we recover the required jump condition one grid point past each interface.

\begin{lemma}  \label{lemma:jump_condition}
	If $\u$ is the solution to the discrete integral system given by Eq. \ref{eq:integral_formulation} then
	\begin{equation} \label{eq:lemma_jump}
		0 = \cG^{\uparrow, \ell}_{j}(\u^{\ell}_{n^{\ell}}, \u^{\ell}_{n^{\ell}+1} )  + \cG^{\downarrow, \ell}_{j}( \u^{\ell}_{0}, \u^{\ell}_{1} ) + \cN^{\ell}_{j}  \mathbf{f}^{\ell} ,
	\end{equation}
	if $j=0$ or $j=n^{\ell}+1$.
\end{lemma}

All the algorithms hinge on the following result. We refer the reader to Appendix \ref{appendix:green_representation} for both proofs.

\begin{theorem} \label{thm:equivalence_integral_formulation}
The solution of the discrete integral system given by Eq. \ref{eq:integral_formulation} equals the restriction of the solution of the discretized Helmholtz equation given by Eq. \ref{eq:discrete_Helmholtz_equation} to the interfaces between subdomains.
\end{theorem}



The computational procedure suggested by the reduction to Eq. \ref{eq:integral_formulation} can be decomposed in an offline part given by Alg. \ref{alg:Integral_precomputation} and an online part given by Alg. \ref{alg:Integral_formulation}.

	\begin{algorithm} Offline computation for the discrete integral formulation \label{alg:Integral_precomputation}
	\begin{algorithmic}[1]
		\Function{ $H$ = Precomputation}{ $\mathbf{m}$, $\omega$ }
			\For{  $\ell = 1: L$ }
				\State $ \mathbf{m}^{\ell} = \mathbf{m}\chi_{\Omega^{\ell}} $   				\Comment{partition the model}
				\State $ \mathbf{H}^{\ell}  = - \triangle - \mathbf{m}^{\ell} \omega^2  $  	\Comment{set local problems}
				\State $[L^{\ell}, U^{\ell}] = \mathtt{lu} \left ( \mathbf{H}^{\ell} \right )  $  							\Comment{factorize local problems}
			\EndFor
			\For{  $\ell = 1: L$ } 																		\Comment{extract Green's functions}
				\State $\mathbf{G}^{\ell}(z_j,z_{j'}) = (U^{\ell})^{-1}(L^{\ell})^{-1} \delta(z_{j'})$    \Comment{ $z_{j'}$ and $z_{j}$ are in the interfaces }
			\EndFor
			\State Form  $\mathbf{\underline{M}}$ 														\Comment{set up the integral system}
		\EndFunction
	\end{algorithmic}
  \end{algorithm}
In the offline computation, the domain is decomposed in layers, the local LU factorizations are computed and stored, the local Green's functions are computed on the interfaces by backsubstitution, and the interface-to-interface operators are used to form the discrete integral system. Alg. \ref{alg:Integral_precomputation} sets up the data structure and machinery to solve Eq. \ref{eq:discrete_Helmholtz_equation} for different right-hand-sides using Alg. \ref{alg:Integral_formulation}.
	\begin{algorithm} Online computation for the discrete integral formulation \label{alg:Integral_formulation}
	\begin{algorithmic}[1]
		\Function{ $\u$ = Helmholtz solver}{ $\mathbf{f}$ }
			\For{  $\ell = 1: L$ }
				\State $ \mathbf{f}^{\ell} = \mathbf{f}\chi_{\Omega^{\ell}} $ 							\Comment{partition the source}
			\EndFor
			\For{  $\ell = 1: L$ }
				\State $\cN^{\ell} \mathbf{f}^{\ell} = (\mathbf{H}^{\ell})^{-1} \mathbf{f}^{\ell}  $ 				\Comment{solve local problems}
			\EndFor
			\State $ \underline{\mathbf{f}} =  \left ( \cN^1_{n^1}\mathbf{f}^1, \cN^2_1\mathbf{f}^2 ,\cN^2_{n^2} \mathbf{f}^2 ,\hdots ,\cN^L_1\mathbf{f}^L \right )^{t} $			\Comment{form r.h.s. for the integral system}
			\State $ \underline{\mathbf{u}} = \left( \underline{\mathbf{M}} \right)^{-1} ( - \mathbf{\underline{f}}) $ 																\Comment{solve for the traces (Eq. \ref{eq:integral_formulation})}
			\For{  $\ell = 1: L$ }
				\State $\u^{\ell}_{j} =	\cG^{\uparrow, \ell}_{j}(\u^{\ell}_{n^{\ell}}, \u^{\ell}_{n^{\ell}+1} )  + \cG^{\downarrow, \ell}_{j}(\u^{\ell}_{0}, \u^{\ell}_{1} ) + \cN^{\ell}_{j}  \mathbf{f}^{\ell}$  					\Comment{reconstruct local solutions (Eq. \ref{eq:discrete_GRF_inside})}
			\EndFor
			\State $\u = (\u^1 , \u^2, \hdots, \u^{L-1}, \u^{L})^t $ 																												\Comment{concatenate the local solutions}
		\EndFunction
	\end{algorithmic}
  \end{algorithm}

The matrix $\underline{\mathbf{M}}$ that results from the discrete integral equations is block sparse and tightly block banded, so the discrete integral system can in principle be factorized by a block LU algorithm. However, and even though the integral system (after pre-computation of the local Green's functions) represents a reduction of one dimension in the Helmholtz problem, solving the integral formulation directly can be prohibitively expensive for large systems, even using distributed linear algebra. The only feasible option in such cases is to use iterative methods; however, the observed condition number\footnote{We do not have a complete understanding of the numerically observed scaling of this condition number. We point out that the problem of solving the discretized Helmholtz equation remains the same, it has only been rewritten in a equivalent form, analogous to a Schur complement reduction. One then would expect to inherit the spectral properties of the original discretized problem. We observe that the discrete Helmholtz equation (Eq. \ref{eq:Helmholtz}) has a condition number that scales as $\cO(h^{-2})$ for fixed frequency (driven primarily by the discretized Laplacian), provided that the absorbing boundary conditions are accurate.

When the frequency increases with the number of degrees of freedom, we observe numerically that the condition number of the discretized Helmholtz equation exhibits the same scaling, provided that the absorbing boundary conditions are accurate. There is no theory that fully explains this behavior; to the authors' knowledge the best partial answer can be found in Thm 3.4 in \cite{Moiola_Spence:Is_the_Helmholtz_Equation_Really_Sign_Indefinite}, in which the authors find a variational formulation of a Helmholtz equation with impedance boundary conditions (an approximation of ABC) that has a coercivity constant with a lower bound independent of the frequency. This result implies that the condition number of the discrete Helmholtz equation would be dominated by the largest eigenvalue of the discretized operators which in the case of the Laplacian is $\cO(h^{-2})$.} of $\underline{\mathbf{M}}$ is $\cO(h^{-2})$, resulting in a high number of iterations for convergence.

The rest of this paper is devoted to the question of designing a good preconditioner for $\mathbf{\underline{M}}$ in regimes of propagating waves, which makes an iterative solver competitive for the discrete integral system.

\begin{remark} \label{remark:M_0}
The idea behind the discrete integral system in Def. \ref{def:integral_formulation} is to take the limit to the interfaces by approaching them from the interior of each layer. It is possible to write an equivalent system by taking the limit from outside. By Lemma \ref{lemma:jump_condition} the wavefield is zero outside the slab, hence by taking this limit we can write the equation for $\u$ as
\begin{equation} \label{eq:M_0_system}
	\mathbf{\underline{M}_0} \u = -	\left ( \begin{array}{c}
 												\cN^1_{n^1+1}\mathbf{f}^1 \\
 												\cN^2_0\mathbf{f}^2 \\
 												\cN^2_{n^2+1}\mathbf{f}^2 \\
 												\vdots \\
 												\cN^L_0\mathbf{f}^L
 											\end{array}
 									\right )  = - \mathbf{\underline{f}_0},
\end{equation}
where
\begin{equation} \label{eq:matrix_M_0}
		\scriptsize
		\mathbf{\underline{M}_0} = \frac{1}{h} \left [
									\begin{array}{ccccccc}
										-\mathbf{G}^1_{n+1,n+1} 	&  \mathbf{G}^1_{n+1,n} 		& 	0 							& 0 					& 0							& 0  \\
										\mathbf{G}^2_{0,1} 			& - \mathbf{G}^2_{0,0}    		& - \mathbf{G}^1_{0,n+1} 		& \mathbf{G}^1_{0,n}	& 0 					& 0  \\
										\mathbf{G}^2_{n+1,1} 		& - \mathbf{G}^2_{n+1,0}  		& - \mathbf{G}^2_{n+1,n+1} 		& \mathbf{G}^2_{n+1,n}	& 0 					& 0\\
										0 							&\ddots    						& \ddots						& \ddots 		 			 	& 0& 0 \\
										0 							& 0								& \mathbf{G}^{L-1}_{0,1} 		& - \mathbf{G}^{L-1}_{0,0} 		& - \mathbf{G}^{L-1}_{0,n+1} 	& \mathbf{G}^{L-1}_{0,n}	 	\\
										0 							& 0 							& \mathbf{G}^{L-1}_{n+1,1}  	& - \mathbf{G}^{L-1}_{n+1,0}  	& - \mathbf{G}^{L-1}_{n+1,n+1} 	& \mathbf{G}^{L-1}_{n+1,n}	\\
										0 							& 0 							&	0							&		0						& \mathbf{G}^{L}_{0,1} 			& - \mathbf{G}^{L}_{0,0}
										\end{array}
								\right ].
	\end{equation}
The systems given by Eq. \ref{eq:integral_formulation} and Eq. \ref{eq:M_0_system} are equivalent; however, the first one has more intuitive implications hence it is the one we deal with it in the immediate sequel.

\end{remark}

\begin{remark} \label{remark:different_partition}
The reduction to a boundary integral equation can be extended to more general partitions than slabs. In particular, it can be extended to layered partitions with irregular interfaces and partitions with different topologies such as grids. However, we need a layered partition to enforce the polarizing conditions, which are the cornerstone of the preconditioner.
\end{remark}

\section{Polarization}



Interesting structure is revealed when we reformulate the system given by Eq. \ref{eq:integral_formulation}, by splitting each interface field into two ``polarized traces". We start by describing the proper discrete expression of the concept seen in the introduction.

\subsection{Polarized Wavefields}

\begin{definition} \label{def:annihilator} (Polarization via annihilation relations.)
	A wavefield $\u$ is said to be up-going at the interface $\Gamma_{\ell,\ell+1}$ if it satisfies the annihilation relation given by
	\begin{equation}
			\cG^{\downarrow, \ell+1}_{1}(\u_0, \u_1 ) = 0.
	\end{equation}
	We denote it as $\u^{\uparrow}$.
	Analogously, a wavefield $\u$ is said to be down-going at the interface $\Gamma_{\ell,\ell+1}$ if it satisfies the annihilation relation given by
	\begin{equation}
			 \cG^{\uparrow, \ell}_{n^{\ell}}(\u_{n^{\ell}}, \u_{n^{\ell}+1} ) = 0.
	\end{equation}
	We denote it as $\u^{\downarrow}$.
\end{definition}

A pair $(\u^{\uparrow}_0, \u^{\uparrow}_1)$ satisfies the up-going annihilation relation at $\Gamma_{\ell,\ell+1}$ when it is a wavefield radiated from below $\Gamma_{\ell,\ell+1}$.

Up and down arrows are convenient notations, but it should be remembered that the polarized fields contain locally reflected waves in addition to transmitted waves, hence are not purely directional as in a uniform medium. The quality of polarization is directly inherited from the choice of local Green's function in the incomplete Green's operators $\cG^{\downarrow, \ell+1}$ and $\cG^{\uparrow, \ell}$. In the polarization context, we may refer to these operators as \emph{annihilators}.


The main feature of a polarized wave is that it can be extrapolated outside the domain using only the Dirichlet data at the boundary. In particular, we can extrapolate one grid point using the extrapolator in Def. \ref{def:extrapolator}.

\begin{definition}  \label{def:extrapolator}
	Let $\v_1$ be the trace of a wavefield at $j = 1$ in local coordinates. Then define the up-going one-sided extrapolator as
	\begin{equation}
		\cE^{\uparrow}_{\ell,\ell+1} \v_1 = \left ( \mathbf{G}^{\ell+1}(z_1,z_1) \right )^{-1} \mathbf{G}^{\ell+1}(z_1,z_0)  \v_1.
	\end{equation}
	Analogously, for a trace $\v_{n^\ell}$ at $j = n^\ell$, define the down-going one-sided extrapolator as
	\begin{equation}
		\cE^{\downarrow}_{\ell,\ell+1} \v_{n^\ell} = \left ( \mathbf{G}^{\ell}(z_{n^{\ell}},z_{n^{\ell}}) \right )^{-1} \mathbf{G}^{\ell}(z_{n^{\ell}},z_{n^{\ell} +1}) \v_{n^\ell}.
	\end{equation}

\end{definition}

Extrapolators reproduce the polarized waves.

\begin{lemma}\label{lemma:extrapolator}
	Let $\u^{\uparrow}$ be an up-going wavefield. Then  $\u^{\uparrow}_0$ is completely determined by $\u^{\uparrow}_1$, as
	\begin{equation}
		\u^{\uparrow}_0 = \cE^{\uparrow}_{\ell,\ell+1}\u^{\uparrow}_1.
	\end{equation}
	Analogously for a down-going wave $\u^{\downarrow}$, we have
	\begin{equation}
		\u^{\downarrow}_{n^{\ell}+1} = \cE^{\downarrow}_{\ell,\ell+1}\u^{\downarrow}_{n^\ell}.
	\end{equation}
\end{lemma}

	\begin{lemma} \label{lemma:extrapolator_def}
	The extrapolator satisfies the following properties:
	\begin{equation}
		\mathbf{G}^{\ell+1}(z_0,z_j) = \cE^{\uparrow}_{\ell,\ell+1} \mathbf{G}^{\ell+1}(z_1,z_j), \qquad \text{for } j \geq 1,
	\end{equation}
	and
	\begin{equation}
		\mathbf{G}^{\ell}(z_{n^{\ell}+1},z_j) = \cE^{\downarrow}_{\ell,\ell+1} \mathbf{G}^{\ell}(z_{n^{\ell}},z_j), \qquad \text{for }  j \leq n^{\ell}.
	\end{equation}

	Moreover, we have the jump conditions
	\begin{equation}
		\mathbf{G}^{\ell+1}(z_0,z_0) -  h \cE^{\uparrow}_{\ell,\ell+1}  =  \cE^{\uparrow}_{\ell,\ell+1}   \mathbf{G}^{\ell+1}(z_1,z_0), \label{eq:extrapolator_jump_0}
	\end{equation}
	and
	\begin{equation}
		\mathbf{G}^{\ell}(z_{n^{\ell}+1},z_{n^{\ell}+1}) - h \cE^{\downarrow}_{\ell,\ell+1}  = \cE^{\downarrow}_{\ell,\ell+1}  \mathbf{G}^{\ell}(z_{n^{\ell}},z_{n^{\ell}+1}).
	\end{equation}
\end{lemma}

In one dimension, the proof of Lemma \ref{lemma:extrapolator_def} (in Appendix \ref{appendix:green_representation}) is a direct application of the nullity theorem \cite{Strang:the_interplay_of_ranks_of_submatrices} and the cofactor formula, in which  $\cE^{\uparrow}_{j,j+1}$ is the ratio between two co-linear vectors. In two dimensions, the proof is slightly more complex but follows the same reasoning.

\begin{lemma} \label{lemma:annihilator_relations_volume}
If $\u^{\uparrow}$ is an up-going wave-field, then the annihilation relation holds inside the layer, i.e.
	 \begin{equation}
	 	 \cG^{\downarrow, \ell+1}_{j}(\u^{\uparrow}_0, \u^{\uparrow}_1 ) = 0 , \qquad \text{for } j \geq 1.
	  \end{equation}
Analogously, if $\u^{\uparrow}$ is a down-going wave-field, then the annihilation relation holds inside the layer, i.e.
		\begin{equation}
	 	 \cG^{\uparrow, \ell}_{j}(\u^{\downarrow}_{n^{\ell}}, \u^{\downarrow}_{n^{\ell}} ) = 0 , \qquad \text{for } j \leq n^{\ell}.
	  \end{equation}

\end{lemma}

\begin{remark} \label{remark:extrapolator}
We gather from the proof of Lemma \ref{lemma:extrapolator_def} that we can define extrapolators inside the domain as well. In fact, from Proposition \ref{proposition:extrapolation_matrix} in Appendix \ref{appendix:general_facts}, we can easily show that
\begin{equation}
	  \left [ \mathbf{G}^{\ell}(z_j,z_j) \right ]^{-1} \mathbf{G}^{\ell}(z_j,z_{j+1}) \mathbf{G}^{\ell}(z_j,z_{k}) =  \mathbf{G}^{\ell}(z_{j+1},z_{k}) \qquad \text{for } k\leq j,
\end{equation}
and
\begin{equation}
	  \left [ \mathbf{G}^{\ell}(z_{j+1},z_{j+1}) \right ]^{-1} \mathbf{G}^{\ell}(z_{j+1},z_{j}) \mathbf{G}^{\ell}(z_{j+1},z_{k}) =  \mathbf{G}^{\ell}(z_{j},z_{k}) \qquad \text{for } j\leq k.
\end{equation}
\end{remark}

\subsection{Polarized Traces}



Some advantageous cancellations occur when doubling the number of unknowns, and formulating a system on traces of polarized wavefields at interfaces. In this approach, each trace is written as the sum of two polarized components. We define the collection of \emph{polarized traces} as
\begin{equation}
	\underline{\underline{\u}} = \left ( 	\begin{array}{c}
								\underline{\u}^{\downarrow}\\
								\underline{\u}^{\uparrow}
								\end{array}
								\right ),
\end{equation}
such that $\underline{\u} = \underline{\u}^{\uparrow} + \underline{\u}^{\downarrow}$. As previously, we underline symbols to denote collections of traces. We now underline twice to denote polarized traces. The original system gives rise to the equations
\begin{equation} \label{eq:open_system}
	\left [ \begin{array} {cc} \underline{ \mathbf{M} } & \underline{ \mathbf{M} } \end{array}  \right ] \underline{\underline{\u}} = -\mathbf{\underline{f}}.
\end{equation}

This polarized formulation increases the number of unknowns; however, the number of equations remains the same. Some further knowledge should be used to obtain a square system. We present three mathematically equivalent, though algorithmically different approaches to closing the system:
\ben
\item a first approach is to impose the annihilation polarization conditions;
\item a second approach is to impose the extrapolator relations in tandem with the annihilation conditions; and
\item a third approach uses the jump conditions to further simplify expressions.
\een

The second formulation yields an efficient preconditioner consisting of decoupled block-triangular matrices. However, its inversion relies on inverting some of the dense diagonal blocks, which is problematic without extra knowledge about these blocks. The third formulation is analogous to the second one, but altogether alleviates the need for dense block inversion. It is this last formulation that we benchmark in the numerical section. We prefer to present all three versions in a sequence, for clarity of the presentation.



\subsection{Annihilation relations}

We can easily close the system given by Eq. \ref{eq:open_system} by explicitly imposing that the out-going traces at each interface satisfies the annihilation conditions (Eq. \ref{eq:annihilator_cond_up} and Eq. \ref{eq:annihilator_cond_down}).  We then obtain a system of equations given by :
\begin{align}
		\cG^{\downarrow, \ell}_{1}(\u^{\ell,\uparrow}_{0}, \u^{\ell, \uparrow}_{1} ) 	+ \cG^{\uparrow, \ell}_{1}(\u^{\ell, \uparrow}_{n^{\ell}}, \u^{\ell, \uparrow}_{n^{\ell}+1} ) + \cG^{\downarrow, \ell}_{1}(\u^{\ell,\downarrow}_{0}, \u^{\ell, \downarrow}_{1} ) 	+ \cG^{\uparrow, \ell}_{1}(\u^{\ell, \downarrow}_{n^{\ell}}, \u^{\ell, \downarrow}_{n^{\ell}+1} ) 	+  \cN^{\ell}_{1}  \mathbf{f}^{\ell}  &= \u^{\ell, \uparrow}_{1} + \u^{\ell, \downarrow}_{1} , \\
		\cG^{\downarrow, \ell}_{n^{\ell}}(\u^{\ell,\uparrow}_{0}, \u^{\ell, \uparrow}_{1} ) 	+ \cG^{\uparrow, \ell}_{n^{\ell}}(\u^{\ell,\uparrow}_{n^{\ell}}, \u^{\ell, \uparrow}_{n^{\ell}+1}) + \cG^{\downarrow, \ell}_{n^{\ell}}(\u^{\ell,\downarrow}_{0}, \u^{\ell,\downarrow}_{1} ) 	+ \cG^{\uparrow, \ell}_{n^{\ell}}(\u^{\ell,\downarrow}_{n^{\ell}}, \u^{\ell,\downarrow}_{n^{\ell}+1})  	+  \cN^{\ell}_{n^{\ell}}  \mathbf{f}^{\ell}  	&= \u^{\ell, \uparrow}_{n^{\ell}} + \u^{\ell, \downarrow}_{n^{\ell}},\\
		\cG^{\uparrow, \ell}_{n^{\ell}}(\u^{\ell,\downarrow}_{n^{\ell}}, \u^{\ell,\downarrow}_{n^{\ell}+1})  		&= 0 , \label{eq:annihilator_cond_up} \\
		\cG^{\downarrow, \ell}_{1}(\u^{\ell,\uparrow}_{0}, \u^{\ell, \uparrow}_{1} ) 								&= 0. \label{eq:annihilator_cond_down}
\end{align}
plus the continuity conditions.
To obtain the matrix form of this system, we define the global annihilator matrices by
\begin{equation}
		\scriptsize
		\mathbf{\underline{A}}^{\downarrow} = \frac{1}{h} \left [
									\begin{array}{ccccccc}
										-\mathbf{G}^1_{n,n+1}    	& \mathbf{G}^1_{n,n} 		& 	0 							& 0 					& 0							& 0  \\
										0 				 			& 0 							& 0 							& 0 					& 0 					& 0  \\
										0 							& 0 					  		& - \mathbf{G}^2_{n,n+1}     & \mathbf{G}^2_{n,n}	& 0 					& 0\\
										0 							& 0								& 0    						& \ddots				& \ddots 		 			 	& 0 \\
										0 							& 0								& 0	& 0 & 0 & 0	 	\\
										0 							& 0 							& 0 &  0	& - \mathbf{G}^{L-1}_{n,n+1} & \mathbf{G}^{L-1}_{n,n}	\\
										0 							& 0 									&	0						&		0						& 0 	& 0
										\end{array}
								\right ],
	\end{equation}
and
\begin{equation}
		\scriptsize
		\mathbf{\underline{A}}^{\uparrow} = \frac{1}{h}  \left [
									\begin{array}{ccccccc}
										0	& 0		& 	0 							& 0 					& 0							& 0  \\
										\mathbf{G}^2_{1,1} 			& - \mathbf{G}^2_{1,0} 	& 0 		& 0	& 0 					& 0  \\
										0 		& 0		& 0	& 0	& 0 					& 0\\
										0 														&\ddots    						& \ddots				& \ddots 		 			 	& 0& 0 \\
										0 							& 0								& \mathbf{G}^{L-1}_{1,1} 	& - \mathbf{G}^{L-1}_{1,0}	& 0 & 0	 	\\
										0 							& 0 							& 0  & 0	& 0	&0	\\
										0 							& 0 									&	0						&		0						& \mathbf{G}^{L}_{1,1} 	& - \mathbf{G}^{L}_{1,0}
										\end{array}
								\right ].
	\end{equation}

\begin{definition} \label{def:polarized_annihilator}
	We define the polarized system completed with annihilator conditions as
\begin{equation} \label{eq:polarized_out_going_1}
		\left [ \begin{array}{cc}
					\mathbf{\underline{M}} 				& \mathbf{\underline{M}} \\
					\mathbf{\underline{A}}^{\downarrow} & \mathbf{\underline{A}}^{\uparrow}
				\end{array}
		\right]
		\underline{\underline{\u}} =
		\left (
			\begin{array}{c}
				-\mathbf{\underline{f}} \\
				0
			\end{array}
		\right).
	\end{equation}
\end{definition}
By construction, if $\underline{\underline{\u}}$ is solution to Eq. \ref{eq:polarized_out_going_1}, then $\underline{\u} = \underline{\u}^{\uparrow} + \underline{\u}^{\downarrow}$ is solution to Eq. \ref{eq:integral_formulation}. Moreover, the non-zero blocks of $\mathbf{\underline{A}}^{\uparrow}$ and $\mathbf{\underline{A}}^{\downarrow}$ are full rank, and given their nested structure it follows that  $[ \mathbf{\underline{A}}^{\downarrow}  \;\;\mathbf{\underline{A}}^{\uparrow}]$ is full row rank as well.

\subsection{Extrapolation conditions}

One standard procedure for preconditioning a system such as Eq. \ref{eq:integral_formulation} is to use a block Jacobi preconditioner, or a block Gauss-Seidel preconditioner. Several solvers based on domain decomposition use the latter as a preconditioner, and typically call it a multiplicative Schwarz iteration. In our case however, once the system is augmented from $\mathbf{\underline{M}}$ to $[ \mathbf{\underline{M}} \;\; \mathbf{\underline{M}} ]$, and completed, it loses its banded structure. The proper form can be restored in the system of Def. \ref{def:polarized_annihilator} by a sequence of steps that we now outline.

It is clear from Def. \ref{def:annihilator} that some of the terms of $\underline{\mathbf{M}}$ contain the annihilation relations. Those terms should be subtracted from the relevant rows of $\underline{\mathbf{M}}$, resulting in new submatrices $\mathbf{\underline{M}}^{\downarrow}$ and $\mathbf{\underline{M}}^{\uparrow}$. Completion of the system is advantageously done in a different way, by encoding polarization via the extrapolator conditions from Def. \ref{def:extrapolator}, rather than the annihiliation conditions. The system given by Eq. \ref{eq:polarized_out_going_1} is then equivalent to
\begin{align}
		\cG^{\uparrow, \ell}_{1}(\u^{\ell, \uparrow}_{n^{\ell}}, \u^{\ell, \uparrow}_{n^{\ell}+1} ) + \cG^{\downarrow, \ell}_{1}(\u^{\ell,\downarrow}_{0}, \u^{\ell, \downarrow}_{1} ) 	+ \cG^{\uparrow, \ell}_{1}(\u^{\ell, \downarrow}_{n^{\ell}}, \u^{\ell, \downarrow}_{n^{\ell}+1} ) 	+  \cN^{\ell}_{1}  \mathbf{f}^{\ell}  &= \u^{\ell, \uparrow}_{1} + \u^{\ell, \downarrow}_{1} , \label{eq:regularized_up_1} \\
		\cG^{\downarrow, \ell}_{n^{\ell}}(\u^{\ell,\uparrow}_{0}, \u^{\ell, \uparrow}_{1} ) 	+ \cG^{\uparrow, \ell}_{n^{\ell}}(\u^{\ell,\uparrow}_{n^{\ell}}, \u^{\ell, \uparrow}_{n^{\ell}+1}) + \cG^{\downarrow, \ell}_{n^{\ell}}(\u^{\ell,\downarrow}_{0}, \u^{\ell,\downarrow}_{1} ) 	+  \cN^{\ell}_{n^{\ell}}  \mathbf{f}^{\ell}  	&= \u^{\ell, \uparrow}_{n^{\ell}} + \u^{\ell, \downarrow}_{n^{\ell}},\\
		\u^{\ell,\downarrow}_{n^{\ell}+1}	& = \cE^{\downarrow}_{\ell,\ell+1}( \u^{\ell,\downarrow}_{n^{\ell}}), \\
		\u^{\ell+1,\uparrow}_{0} 		& = \cE^{\uparrow}_{\ell,\ell+1}(\u^{\ell+1,\uparrow}_{1}) 	.
\end{align}

We can switch to a matrix form of these equations, by letting
\begin{equation}
		\scriptsize
		\mathbf{\underline{M}}^{\downarrow} = \frac{1}{h} \left [
									\begin{array}{ccccccc}
										 - \I						& 0 					 		& 	0 							& 0 					& 0							& 0  \\
										\mathbf{G}^2_{1,1} 			& - \mathbf{G}^2_{1,0} - \I		& - \mathbf{G}^1_{1,n+1} 		& \mathbf{G}^1_{1,n}	& 0 					& 0  \\
										\mathbf{G}^2_{n,1} 			& - \mathbf{G}^2_{n,0}  		&  - \I	& 0	& 0 					& 0\\
										0 							&\ddots    						& \ddots				& \ddots 		 			 	& 0& 0 \\
										0 							& 0								& \mathbf{G}^{L-1}_{1,1} 	& - \mathbf{G}^{L-1}_{1,0}-\I	& - \mathbf{G}^{L-1}_{1,n+1} & \mathbf{G}^{L-1}_{1,n}	 	\\
										0 							& 0 							& \mathbf{G}^{L-1}_{n,1}  	& - \mathbf{G}^{L-1}_{n,0}  	&  - \I	& 0	\\
										0 							& 0 							&	0						&		0						& \mathbf{G}^{L}_{1,1} 	& - \mathbf{G}^{L}_{1,0} - \I
										\end{array}
								\right ],
	\end{equation}

\begin{equation}
		\scriptsize
		\mathbf{\underline{M}}^{\uparrow} = \frac{1}{h} \left [
									\begin{array}{ccccccc}
										-\mathbf{G}^1_{n,n+1} - \I	& \mathbf{G}^1_{n,n} 		& 	0 							& 0 						& 0									& 0  \\
										0 							& -\I						& -\mathbf{G}^1_{1,n+1} 		& \mathbf{G}^1_{1,n}		& 0 								& 0  \\
										\mathbf{G}^2_{n,1} 			& -\mathbf{G}^2_{n,0}  		& -\mathbf{G}^2_{n,n+1} - \I	& \mathbf{G}^2_{n,n}		& 0 								& 0\\
										0 							& \ddots    				& \ddots						& \ddots 		 			& 0 								& 0 \\
										0 							& 0							&  0							& -\I						& - \mathbf{G}^{L-1}_{1,n+1} 		& \mathbf{G}^{L-1}_{1,n}	 	\\
										0 							& 0 						& \mathbf{G}^{L-1}_{n,1}  		& - \mathbf{G}^{L-1}_{n,0}  & - \mathbf{G}^{L-1}_{n,n+1} - \I	& \mathbf{G}^{L-1}_{n,n}	\\
										0 							& 0 						&	0							& 0							& 0									&  - \I
										\end{array}
								\right ],
	\end{equation}

	\begin{equation}
		\scriptsize
		\mathbf{\underline{E}}^{\downarrow} = \frac{1}{h} \left [
									\begin{array}{ccccccc}
										h\left (\mathbf{G}^1_{n,n}\right)^{-1}\mathbf{G}^1_{n,n+1}  	&  -\I	& 	0 														& 0 	 & 0																	& 0  \\
										0 				 											& 0		& 0 														& 0 	 & 0 																	& 0  \\
										0 															& 0  	& h\left (\mathbf{G}^2_{n,n}\right)^{-1}\mathbf{G}^2_{n,n+1} &  -\I	 & 0 																	& 0 \\
										0 															& 0		& 0    														& \ddots & \ddots 		 			 											& 0 \\
										0 															& 0		& 0															& 0 	 & 0 						 											& 0	 	\\
										0 															& 0		& 0  														& 0		 & h\left (\mathbf{G}^{L-1}_{n,n}\right)^{-1}\mathbf{G}^{L-1}_{n,n+1}	& -\I	\\
										0 															& 0		& 0															& 0      & 0 																	& 0
										\end{array}
								\right ],
	\end{equation}

	\begin{equation}
		\scriptsize
		\mathbf{\underline{E}}^{\uparrow} = \frac{1}{h}  \left [
									\begin{array}{ccccccc}
										0	& 0		& 	0 							& 0 					& 0							& 0  \\
										-\I 			&  h\left(\mathbf{G}^2_{1,1} \right)^{-1} \mathbf{G}^2_{1,0} 	& 0 		& 0	& 0 					& 0  \\
										0 		& 0		& 0	& 0	& 0 					& 0\\
										0 														&\ddots    						& \ddots				& \ddots 		 			 	& 0& 0 \\
										0 							& 0								& -\I 			&  h\left(\mathbf{G}^{L-1}_{1,1} \right)^{-1} \mathbf{G}^{L-1}_{1,0} 	& 0 & 0	 	\\
										0 							& 0 							& 0  & 0	& 0	&0	\\
										0 							& 0 									&	0						&		0						& -\I 			&  h\left(\mathbf{G}^L_{1,1} \right)^{-1} \mathbf{G}^L_{1,0}
										\end{array}
								\right ].
	\end{equation}

The sparsity pattern of the system formed by these block matrices is given by Fig. \ref{fig:spy_matrix_extrapolation} {\it(left)}. We arrive at the following definition.


\begin{definition} \label{def:polarized_out_going}
	We define the polarized system completed with the extrapolation relations as
	\begin{equation} \label{eq:polarized_out_going}
	\left [ \begin{array}{cc}
					\mathbf{\underline{M}}^{\downarrow}			& \mathbf{\underline{M}}^{\uparrow} \\
					\mathbf{\underline{E}}^{\downarrow} 		& \mathbf{\underline{E}}^{\uparrow}
			\end{array}
		\right]
		\underline{\underline{\u}} =
		\left (
			\begin{array}{c}
				-\mathbf{\underline{f}} \\
				0
			\end{array}
		\right ).
\end{equation}
\end{definition}

\begin{figure}[H]
    \begin{center}
        \includegraphics[trim = 0mm 0mm 0mm 0mm, clip, width = 60mm]{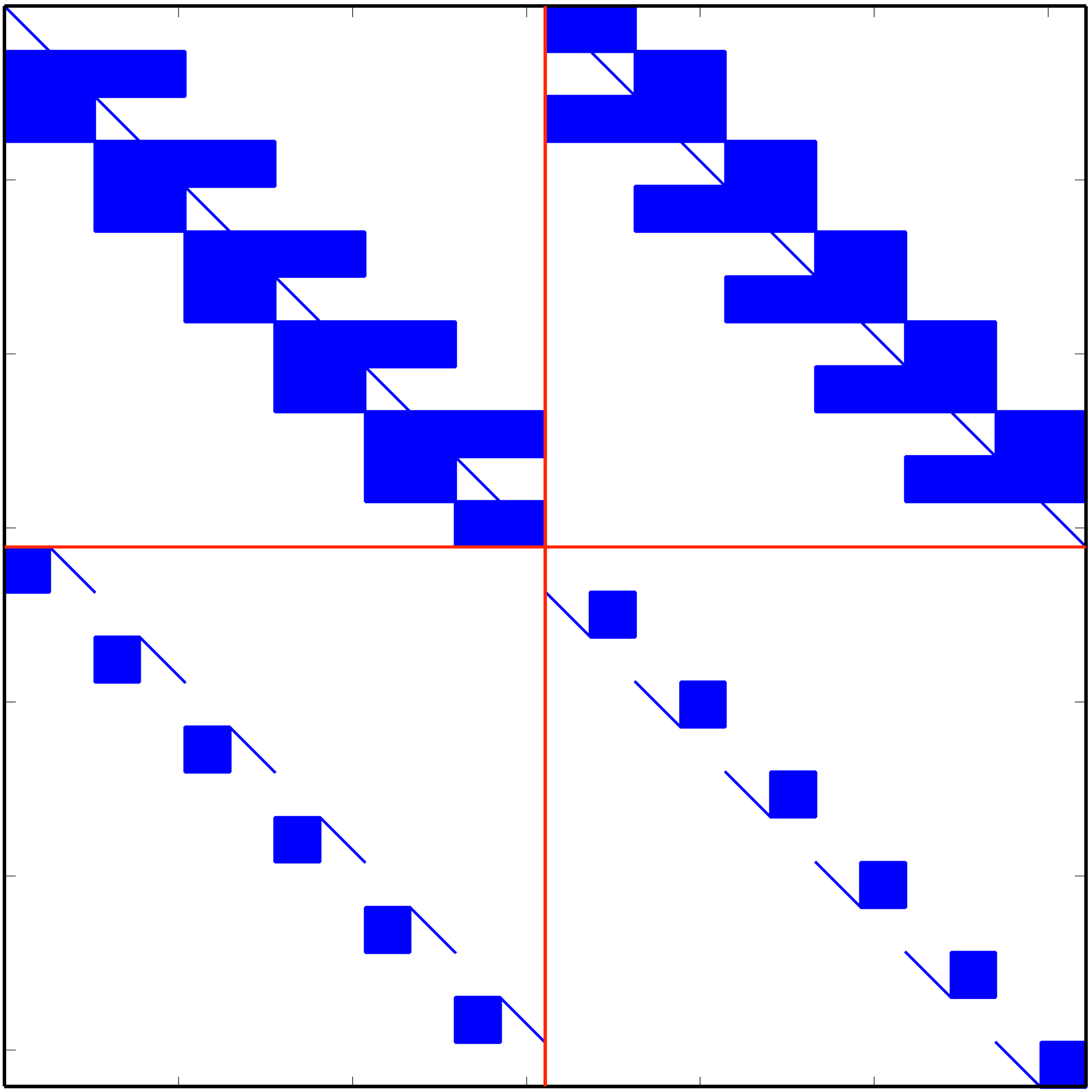} \hspace{0.5cm} \includegraphics[trim = 0mm 0mm 0mm 0mm, clip, width = 60mm]{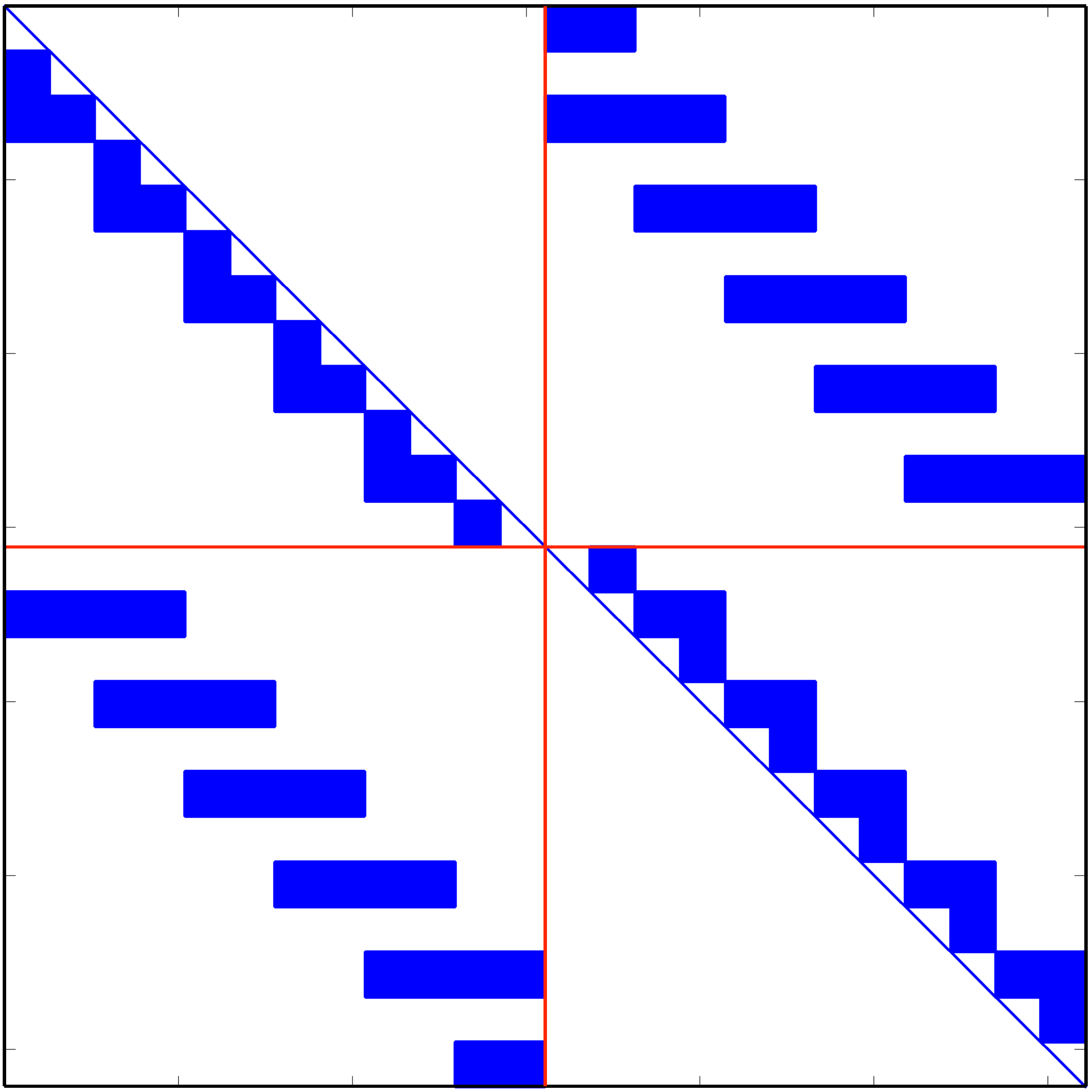}
         \caption{ Left: Sparsity pattern of the system in Eq. \ref{eq:polarized_out_going}. Right: Sparsity pattern of $\mathbf{\underline{D}}^{\text{extrap}} + \mathbf{\underline{R}}^{\text{extrap}}$.}
         \label{fig:spy_matrix_extrapolation}
     \end{center}
\end{figure}

The interesting feature of this system, in contrast to the previous formulation, is that $\u^{\ell,\downarrow}_{n^{\ell}}$ is undisturbed (multiplied by an identity block) both in $\mathbf{\underline{M}}^{\downarrow}$ and in $\mathbf{\underline{E}}^{\downarrow}$. Similarly, $\u^{\ell,\uparrow}_1$ is left undisturbed by $\mathbf{\underline{M}}^{\uparrow}$ and $\mathbf{\underline{E}}^{\uparrow}$. This is apparent from Fig. \ref{fig:spy_matrix_extrapolation} {\it(left)}. Following this observation we can permute the rows of the matrix to obtain

\begin{equation} \label{eq:splitting_out_going}
	\left( 	\left [ \begin{array}{cc}
					\mathbf{\underline{D}}^{\downarrow}			& 0 \\
					0 									 		& \mathbf{\underline{D}}^{\uparrow}
				\end{array}
			\right] +
			\left [ \begin{array}{cc}
					0 							& \mathbf{\underline{U}}	 \\
					\mathbf{\underline{L}}		&  0
				\end{array}
			\right]
	\right )
		\underline{\underline{\u}} = \mathbf{\underline{P}}
		\left (
			\begin{array}{c}
				-\mathbf{\underline{f}} \\
				0
			\end{array}
		\right ),
\end{equation}
where the diagonal blocks $ \mathbf{\underline{D}}^{\downarrow}$ and $ \mathbf{\underline{D}}^{\downarrow}$ are respectively upper triangular and lower triangular, with identity blocks on the diagonal; $ \mathbf{\underline{P}}$ is an appropriate `permutation' matrix; and $\mathbf{\underline{U}}$ and $\mathbf{\underline{L}}$ are block sparse matrices.
We define the matrices
\begin{equation}
	\mathbf{\underline{D}}^{\text{extrap}}	=
	 	\left [ \begin{array}{cc}
					\mathbf{\underline{D}}^{\downarrow}			& 0 \\
					0 									 		& \mathbf{\underline{D}}^{\uparrow}
				\end{array}
			\right ], \qquad
	\mathbf{\underline{R}}^{\text{extrap}}	=
			\left [ \begin{array}{cc}
					0 							& \mathbf{\underline{U}}	 \\
					\mathbf{\underline{L}}		&  0
				\end{array}
			\right].
\end{equation}

We can observe the sparsity pattern of the permuted system in Fig. \ref{fig:spy_matrix_extrapolation} {\it (right)}.

\subsection{Jump condition}

The polarized system in Def. \ref{def:polarized_out_going} can be easily preconditioned using a block Jacobi iteration for the $2\times 2$ block matrix, which yields a procedure to solve the Helmholtz equation. Moreover, using $\mathbf{\underline{D}}^{\text{extrap}} $ as a pre-conditioner within GMRES, yields remarkable results. However, in order to obtain the desired structure, we need to use the extrapolator, whose construction involves inverting small dense blocks. This can be costly and inefficient when dealing with large interfaces (or surfaces in 3D). In order to avoid the inversion of any local operator we exploit the properties of the discrete GRF to obtain an equivalent system that avoids any unnecesary dense linear algebra operation.

Following Remark \ref{remark:M_0}, we can complete the polarized system by imposing,
\begin{equation}
	\left[ \begin{array}{cc} \mathbf{\underline{M}_0} & \mathbf{\underline{M}_0} \end{array} \right] \underline{\underline{\u}} = - \mathbf{\underline{f}_0},
\end{equation}
where $\mathbf{\underline{M}_0}$ encodes the jump condition for the GRF. However, these blocks do not preserve $\u^{\uparrow}$ and $\u^{\downarrow}$ like $\mathbf{\underline{D}}^{\text{extrap}}$ did, because $\mathbf{\underline{M}_0}$ does not have identities on the diagonal. Fortunately, it is possible to include the information contained in the annihilation relations directly into $\mathbf{\underline{M}_0}$, exactly as we did with $\mathbf{\underline{M}}$. Lemma  \ref{lemma:extrapolator_M_0} summarizes the expression resulting from incorporating the extrapolation conditions into $\mathbf{\underline{M}_0}$.
\begin{lemma} \label{lemma:extrapolator_M_0}
If $\underline{\underline{\u}}$ is solution to the system given by Def. \ref{def:polarized_out_going} then
\begin{align} \label{eq:stepper_equivalence_1}
	\u^{\uparrow, \ell}_0  = & \cE^{\uparrow}_{\ell,\ell+1} \u^{\uparrow,\ell}_1   = \cG^{\uparrow, \ell}_{0}(\u^{\ell, \uparrow}_{n^{\ell}}, \u^{\ell, \uparrow}_{n^{\ell}+1} ) + \cG^{\downarrow, \ell}_{0}(\u^{\ell,\downarrow}_{0}, \u^{\ell, \downarrow}_{1} ) 	+ \cG^{\uparrow, \ell}_{0}(\u^{\ell, \downarrow}_{n^{\ell}}, \u^{\ell, \downarrow}_{n^{\ell}+1} ) 	+  \cN^{\ell}_{0}  \mathbf{f}^{\ell}, \\  \label{eq:stepper_equivalence_2}
	\u^{\downarrow, \ell}_{n^{\ell}+1} = & \cE^{\downarrow}_{\ell,\ell+1} \u^{\downarrow,\ell}_{n^{\ell}}   = \cG^{\downarrow, \ell}_{n^{\ell}+1}(\u^{\ell,\uparrow}_{0}, \u^{\ell, \uparrow}_{1} ) 	+ \cG^{\uparrow, \ell}_{n^{\ell}+1}(\u^{\ell,\uparrow}_{n^{\ell}}, \u^{\ell, \uparrow}_{n^{\ell}+1}) + \cG^{\downarrow, \ell}_{n^{\ell}+1}(\u^{\ell,\downarrow}_{0}, \u^{\ell,\downarrow}_{1} ) +  \cN^{\ell}_{n^{\ell}+1}  \mathbf{f}^{\ell}.
\end{align}
\end{lemma}

The jump conditions in Lemma \ref{lemma:jump_condition} are heavily used in the proof, see the Appendix.  We now replace the extrapolation relations by Eq. \ref{eq:stepper_equivalence_1} and Eq. \ref{eq:stepper_equivalence_2}, which leads to the next system :
\begin{align}
		 \cG^{\uparrow, \ell}_{1}(\u^{\ell, \uparrow}_{n^{\ell}}, \u^{\ell, \uparrow}_{n^{\ell}+1} ) + \cG^{\downarrow, \ell}_{1}(\u^{\ell,\downarrow}_{0}, \u^{\ell, \downarrow}_{1} ) 	+ \cG^{\uparrow, \ell}_{1}(\u^{\ell, \downarrow}_{n^{\ell}}, \u^{\ell, \downarrow}_{n^{\ell}+1} ) 	+  \cN^{\ell}_{1}  \mathbf{f}^{\ell}  &= \u^{\ell, \uparrow}_{1} + \u^{\ell, \downarrow}_{1} , \label{eq:jump_1} \\
		\cG^{\downarrow, \ell}_{n^{\ell}}(\u^{\ell,\uparrow}_{0}, \u^{\ell, \uparrow}_{1} ) 	+ \cG^{\uparrow, \ell}_{n^{\ell}}(\u^{\ell,\uparrow}_{n^{\ell}}, \u^{\ell, \uparrow}_{n^{\ell}+1}) + \cG^{\downarrow, \ell}_{n^{\ell}}(\u^{\ell,\downarrow}_{0}, \u^{\ell,\downarrow}_{1} ) 	+ \cN^{\ell}_{n^{\ell}}  \mathbf{f}^{\ell}  	&= \u^{\ell, \uparrow}_{n^{\ell}} + \u^{\ell, \downarrow}_{n^{\ell}},\label{eq:jump_2}  \\
		\cG^{\uparrow, \ell}_{0}(\u^{\ell, \uparrow}_{n^{\ell}}, \u^{\ell, \uparrow}_{n^{\ell}+1} ) + \cG^{\downarrow, \ell}_{0}(\u^{\ell,\downarrow}_{0}, \u^{\ell, \downarrow}_{1} ) 	+ \cG^{\uparrow, \ell}_{0}(\u^{\ell, \downarrow}_{n^{\ell}}, \u^{\ell, \downarrow}_{n^{\ell}+1} ) 	+  \cN^{\ell}_{0}  \mathbf{f}^{\ell}  &= \u^{\ell, \uparrow}_{0}, \label{eq:jump_3}  \\
		\cG^{\downarrow, \ell}_{n^{\ell}+1}(\u^{\ell,\uparrow}_{0}, \u^{\ell, \uparrow}_{1} ) 	+ \cG^{\uparrow, \ell}_{n^{\ell}+1}(\u^{\ell,\uparrow}_{n^{\ell}}, \u^{\ell, \uparrow}_{n^{\ell}+1}) + \cG^{\downarrow, \ell}_{n^{\ell}+1}(\u^{\ell,\downarrow}_{0}, \u^{\ell,\downarrow}_{1} ) +  \cN^{\ell}_{n^{\ell}+1}  \mathbf{f}^{\ell}  &= \u^{\ell, \downarrow}_{n^{\ell}}.	\label{eq:jump_4}
\end{align}

We define the matrix form of Eq. \ref{eq:jump_3} and Eq. \ref{eq:jump_4} by
\begin{equation}
\scriptsize
\mathbf{\underline{M}}_0^{\uparrow} = 	 \frac{1}{h}  \left [
												\begin{array}{cccccccc}
													- \mathbf{G}^1_{n+1,n+1} 	& \mathbf{G}^1_{n+1,n}	& 0 						& 0 					& 0 & 0  &  0 					& 0	\\
													- \I 					& 	0 						& - \mathbf{G}^2_{0,n+1} 	& \mathbf{G}^2_{0,n}	& 0 						& 0 					& 0 & 0  \\
													\mathbf{G}^2_{n+1,1} 	& - \mathbf{G}^2_{n+1,0} 	& - \mathbf{G}^2_{n+1,n+1} 	& \mathbf{G}^2_{n+1,n}	& 0 						& 0 					& 0 & 0  \\
													0 							&  0 					& - \I 					& 	0 						& - \mathbf{G}^3_{0,n+1} 	& \mathbf{G}^3_{0,n} 	& 0 & 0\\
													0 							&  0 					& \mathbf{G}^3_{n+1,1}  	& - \mathbf{G}^3_{n+1,0}& - \mathbf{G}^3_{n+1,n+1}	& \mathbf{G}^3_{n+1,n}	& 0 & 0 \\
													0 							&  0 					& 0							& \ddots 				& \ddots 					& \ddots				& \ddots & 0 \\
													0 							&  0					& 0							& 0						&- \I						& 	0 					& - \mathbf{G}^{L-1}_{0,n+1} 	& \mathbf{G}^{L-1}_{0,n} 	\\
													0 							&  0 					& 0 						& 0						& \mathbf{G}^{L-1}_{n+1,1}  & - \mathbf{G}^{L-1}_{n+1,0} & - \mathbf{G}^{L-1}_{n+1,n+1} & \mathbf{G}^{L-1}_{n+1,n} \\
													0 							&  0 					& 0 						& 0		  				&	0						&		0				& 	- \I							& 0
												\end{array}
										\right ],
\end{equation}
\begin{equation}
\scriptsize
\mathbf{\underline{M}}_0^{\downarrow} = 	 \frac{1}{h} \left [
												\begin{array}{cccccccc}
													0	& -\I	& 0 						& 0 					& 0 & 0  &  0 					& 0	\\
													\mathbf{G}^2_{0,1} 	& -\mathbf{G}^2_{0,0}	& - \mathbf{G}^2_{0,n+1} 	& \mathbf{G}^2_{0,n}	& 0 						& 0 					& 0 & 0  \\
													\mathbf{G}^2_{n+1,1} 	& - \mathbf{G}^2_{n+1,0} 	&0 							& -\I	& 0 						& 0 					& 0 & 0  \\
													0 							&  0 					& \mathbf{G}^3_{0,1} 		& - \mathbf{G}^3_{0,0}	& - \mathbf{G}^3_{0,n+1} 	& \mathbf{G}^3_{0,n} 	& 0 & 0\\
													0 							&  0 					& \mathbf{G}^3_{n+1,1}  	& - \mathbf{G}^3_{n+1,0} &0 	& -\I	& 0 & 0 \\
													0 							&  0 					& 0							& \ddots 				& \ddots 					& \ddots				& \ddots & 0 \\
													0 							&  0					& 0							& 0						& \mathbf{G}^{L-1}_{0,1}	& -	\mathbf{G}^{L-1}_{0,0}	 - \mathbf{G}^{L-1}_{0,n+1} 	& \mathbf{G}^{L-1}_{0,n} 	\\
													0 							&  0 					& 0 						& 0						& \mathbf{G}^{L-1}_{n+1,1}  & - \mathbf{G}^{L-1}_{n+1,0} & 0 & -\I\\
													0 							&  0 					& 0 						& 0		  				&	0						&		0				& 	 \mathbf{G}^{L}_{0,1}	& -\mathbf{G}^{L}_{0,0}
												\end{array}
										\right ],
\end{equation}

	and

	\begin{equation}
	 	\underline{\mathbf{f}}_0 =	\left (
 										\begin{array}{c}
 											\cN^1_{n^1+1}\mathbf{f}^1 \\
 											\cN^2_0\mathbf{f}^2 \\
 											\cN^2_{n^2+1}\mathbf{f}^2 \\
 											\vdots \\
 											\cN^L_0\mathbf{f}^L
 										\end{array}
 									\right ).
	\end{equation}

The resulting matrix equations are as follows.

\begin{definition} \label{def:polarized_jump_condition}
	We define the polarized system completed with jump conditions as
	\begin{equation} \label{eq:polarized_jump_condition}
		\underline{\underline{\mathbf{M}}} \, \underline{\underline{\u}} =  \left [ \begin{array}{cc}
					\mathbf{\underline{M}}^{\downarrow} 	& \mathbf{\underline{M}}^{\uparrow} \\
					\mathbf{\underline{M}_0}^{\downarrow} 	& \mathbf{\underline{M}_0}^{\uparrow}
				\end{array}
		\right ]
		\underline{\underline{\u}} =
		\left (
			\begin{array}{c}
				-\mathbf{\underline{f}} \\
				-\mathbf{\underline{f}_0}
			\end{array}
		\right).
	\end{equation}
\end{definition}

By construction, the system given by Eq. \ref{eq:polarized_jump_condition} has identities at the same locations as
Eq. \ref{eq:polarized_out_going}. The same row permutation $\mathbf{\underline{P}}$ as before will result in triangular diagonal blocks with identity blocks on the diagonal:
\begin{equation} \label{eq:splitting_jump}
    \mathbf{\underline{P}}  \left [ \begin{array}{cc}
					\mathbf{\underline{M}}^{\downarrow}		& \mathbf{\underline{M}}^{\uparrow} \\
					\mathbf{\underline{M}_0}^{\downarrow} 	& \mathbf{\underline{M}_0}^{\uparrow}
				\end{array}
		\right] \underline{\underline{\u}}  =
	\left( \mathbf{\underline{D}}^{\text{jump}} + \mathbf{\underline{R}}^{\text{jump}} \right )
		\underline{\underline{\u}} = \mathbf{\underline{P}}
		\left (
			\begin{array}{c}
				-\mathbf{\underline{f}} \\
				-\mathbf{\underline{f}_0}
			\end{array}
		\right ),
\end{equation}
where,

	\begin{equation} \label{eq:def_permuted_system_jump}
			 \mathbf{\underline{D}}^{\text{jump}} =	\left [ \begin{array}{cc}
															\mathbf{\underline{D}}^{\downarrow, \text{jump}}			& 0 \\
															0 									 		& \mathbf{\underline{D}}^{\uparrow,\text{jump}}
														\end{array}
													\right] , \qquad
			 \mathbf{\underline{R}}^{\text{jump}} =	\left [ \begin{array}{cc}
															0 							& \mathbf{\underline{U}}^{\text{jump}}	 \\
															\mathbf{\underline{L}}^{\text{jump}}		&  0
														\end{array}
													\right].
	\end{equation}
	We can observe the sparsity pattern in Fig. \ref{fig:spy_matrix_jump} {\it(right)}.

\begin{figure}[H]
	\begin{center}
	 	\includegraphics[trim = 0mm 0mm 0mm 0mm, clip, width = 60mm]{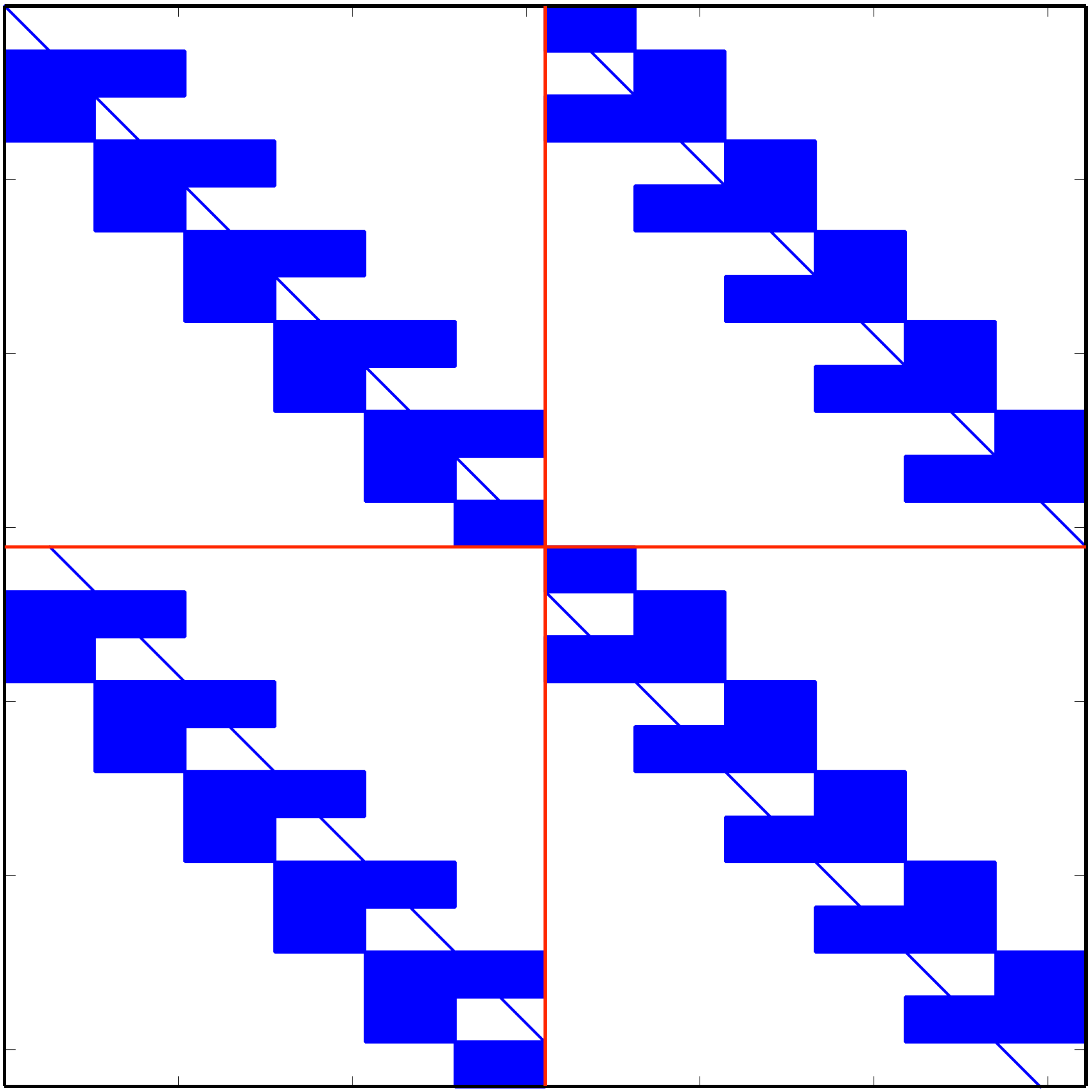} \hspace{0.5cm} \includegraphics[trim = 0mm 0mm 0mm 0mm, clip, width = 60mm]{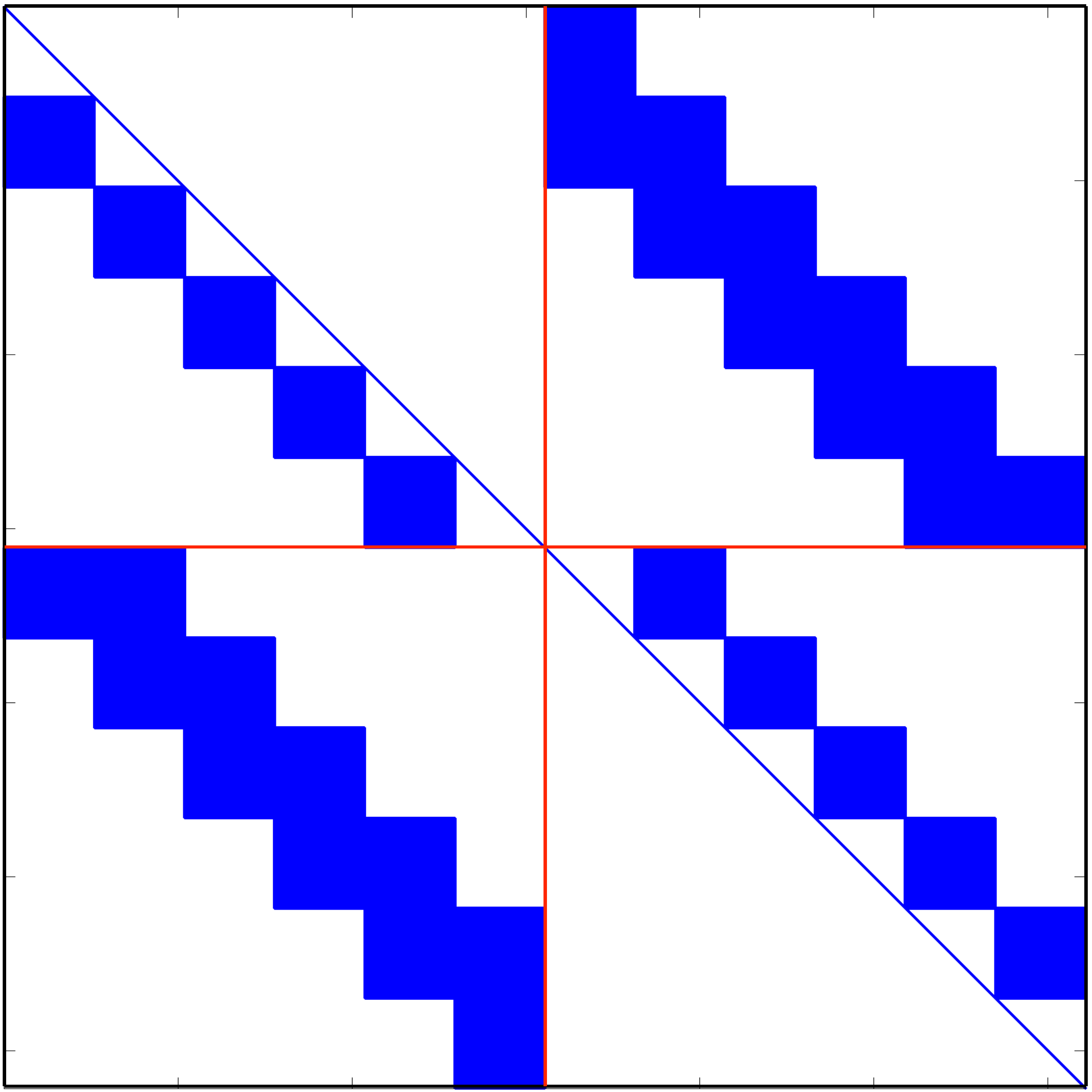}
	 	 \caption{ Left: Sparsity pattern of the system in Eq. \ref{eq:polarized_jump_condition}. Right: Sparsity pattern of $\mathbf{\underline{D}}^{\text{jump}} + \mathbf{\underline{R}}^{\text{jump}}$.}
	 	 \label{fig:spy_matrix_jump}
	 \end{center}
\end{figure}

For reference, here is the explicit form of the blocks of $\mathbf{\underline{D}}^{\text{jump}}$.

	\begin{equation}
		\scriptsize
		\mathbf{\underline{D}^{\downarrow, \text{jump}}} =  	\frac{1}{h} \left [
												\begin{array}{ccccccc}
													-\I 					& 0 						& 	0 							& 0 							& 0								& 0 				& 0			\\
													0 						& -\I 						& 	0 							& 0 							& 0								& 0 				& 0			\\
													\mathbf{G}^2_{n,1} 		& - \mathbf{G}^2_{n,0} 		& - \I 							& 0 							& 0 							& 0 				& 0			\\
													\mathbf{G}^2_{n+1,1} 	& - \mathbf{G}^2_{n+1,0} 	& 0 							& - \I 							& 0 							& 0 				& 0			\\
													0 						&\ddots    					& \ddots						& \ddots 		 			 	& \ddots 						& 0 				& 0			\\
													0 						& 0							& 0								& \mathbf{G}^{L-1}_{n,1} 		& - \mathbf{G}^{L-1}_{n,0}		& -\I 				& 0		 	\\
													0 						& 0 						& 0 							& \mathbf{G}^{L-1}_{n+1,1}  	& - \mathbf{G}^{L-1}_{n+1,0}  	& 0 				& -\I
													\end{array}
												\right ],
	\end{equation}
	\begin{equation}
		\scriptsize
		\mathbf{\underline{D}}^{\uparrow, \text{jump}} =  	\frac{1}{h} \left [
												\begin{array}{ccccccc}

													- \I 					& 0					& -\mathbf{G}^2_{0,n+1} 	&  \mathbf{G}^2_{0,n} 	& 0 			& 0 							& 0							\\
													0						& - \I				& -\mathbf{G}^2_{1,n+1} 	&  \mathbf{G}^2_{1,n} 	& 0 			& 0 							& 0							\\
													0 						& 0 				&\ddots 	    			& \ddots 				& \ddots  		& \ddots 						& 0 						\\
													0 						& 0					& 0							& -\I 					& 0				& -\mathbf{G}^{L-1}_{0,n+1} 	& \mathbf{G}^{L-1}_{0,n}	\\
													0 						& 0 				& 0 						& 0 					& -\I			& -\mathbf{G}^{L-1}_{1,n+1}  	& \mathbf{G}^{L-1}_{1,n}  	\\
													0 						& 0					& 0							& 0 					& 0				& -\I							& 0 						\\
													0 						& 0					& 0							& 0 					& 0				& 0 							& -\I
													\end{array}
												\right ],
	\end{equation}

	\begin{equation}
		\scriptsize
		\mathbf{\underline{L}}^{\text{jump}} = \frac{1}{h} \left [
									\begin{array}{ccccccc}
										\mathbf{G}^2_{0,1} 		& - \mathbf{G}^2_{0,0} 		& - \mathbf{G}^2_{0,n+1} 	& \mathbf{G}^2_{0,n}			& 0 							& 0  						\\
										\mathbf{G}^2_{1,1} 		& - \mathbf{G}^2_{1,0} -\I 	& - \mathbf{G}^2_{1,n+1} 	& \mathbf{G}^2_{1,n}			& 0 							& 0 						\\
										0 						&\ddots 					& \ddots					& \ddots 		 			 	& 0 							& 0 						\\
										0 						& 0							& \mathbf{G}^{L-1}_{0,1} 	& - \mathbf{G}^{L-1}_{0,0} 		& - \mathbf{G}^{L-1}_{0,n+1} 	& \mathbf{G}^{L-1}_{0,n}	\\
										0 						& 0 						& \mathbf{G}^{L-1}_{1,1} 	& - \mathbf{G}^{L-1}_{1,0} -\I 	& - \mathbf{G}^{L-1}_{1,n+1} 	& \mathbf{G}^{L-1}_{1,n} 	\\
										0 						& 0 						&	0						& 0 							&  \mathbf{G}^{L}_{0,1} 		& -\mathbf{G}^{L}_{0,0}		\\
										0 						& 0 						&	0						& 0 							&  \mathbf{G}^{L}_{1,1}			& - \mathbf{G}^{L}_{1,0} -\I
										\end{array}
								\right ],
	\end{equation}
and
		\begin{equation}
		\scriptsize
		\mathbf{\underline{U}}^{\text{jump}} = \frac{1}{h} \left [
									\begin{array}{ccccccc}
										 - \mathbf{G}^1_{n,n+1} -\I	& \mathbf{G}^1_{n,n}    	& 0 							& 0 							& 0 							& 0  \\
										 - \mathbf{G}^1_{n+1,n+1} 	& \mathbf{G}^1_{n+1,n}    	& 0 							& 0 							& 0 							& 0  \\
										\mathbf{G}^2_{n,1} 			& - \mathbf{G}^2_{n,0} -\I 	& - \mathbf{G}^2_{n,n+1} 		& \mathbf{G}^2_{n,n}			& 0 							& 0\\
										\mathbf{G}^2_{n+1,1} 		& - \mathbf{G}^2_{n+1,0} 	& - \mathbf{G}^2_{n+1,n+1} 		& \mathbf{G}^2_{n+1,n}			& 0 							& 0  							\\
										0 							&\ddots    					& \ddots						& \ddots 		 			 	& 0  							& 0 \\
										0 							& 0 						& \mathbf{G}^{L-1}_{n,1} 		& - \mathbf{G}^{L-1}_{n,0} -\I 	& - \mathbf{G}^{L-1}_{n,n+1} 	& \mathbf{G}^{L-1}_{n+1,n} \\
										0 							& 0							& \mathbf{G}^{L-1}_{n+1,1} 		& - \mathbf{G}^{L-1}_{n+1,0} 	& - \mathbf{G}^{L-1}_{n+1,n+1} 	& \mathbf{G}^{L-1}_{n+1,n}			\\
										\end{array}
								\right ].
	\end{equation}


The formulations given by Eq. \ref{eq:polarized_out_going} and Eq. \ref{eq:polarized_jump_condition} are equivalent. The following lemma can be proved from Lemma \ref{lemma:jump_condition}  and Lemma \ref{lemma:extrapolator_M_0}.

\begin{proposition} \label{lemma:equivalence_formulations}
$\underline{\underline{\u}}$ is solution to the system in Def. \ref{def:polarized_out_going} if and only if $\underline{\underline{\u}}$ is solution to the system in Def. \ref{def:polarized_jump_condition}.
\end{proposition}

The numerical claims of this paper concern the system in Eq. \ref{eq:splitting_jump}, and preconditioned with the direct inversion of $ \mathbf{\underline{D}}^{\text{jump}} $ defined by Eq. \ref{eq:def_permuted_system_jump}.

\section{Preconditioners}

\subsection{Gauss-Seidel Iteration}

In this section we let  $\mathbf{\underline{D}}$ for either $\mathbf{\underline{D}}^{\text{jump}}$ or $\mathbf{\underline{D}}^{\text{extrap}}$. While we mostly use the jump formulation in practice, the structure of the preconditioner is common to both formulations.

Inverting any such $\mathbf{\underline{D}}$ is trivial using block back-substitution for each block $\mathbf{\underline{D}}^{\downarrow}$, $\mathbf{\underline{D}}^{\uparrow}$, because they have a triangular structure, and their diagonals consist of  identity blocks. Physically, the inversion of $\mathbf{\underline{D}}$ results in two sweeps of the domain (top-down and bottom-up) for computing transmitted waves from incomplete Green's formulas. This procedure is close enough to Gauss-Seidel to be referred to as such\footnote{Another possibility would be to call it a multiplicative Schwarz iteration.}.

\begin{algorithm} Gauss-Seidel iteration \label{alg:Gauss_Seidel}
	\begin{algorithmic}[1]
		\Function{ $\u$ = Gauss Seidel}{ $\mathbf{f}, \epsilon_{\text{tol}}$}
			\State $\underline{\u}^0 = (\u^{\downarrow} , \u^{\uparrow} )^{t} =  0 $
			\While{  $ \|\underline{\u}^{n+1}  - \underline{\u}^{n}  \| /\| \underline{\u}^{n} \| >  \epsilon_{\text{tol}} $ }
				\State $\underline{\u}^{n+1} = (\mathbf{\underline{D}}\,)^{-1}( \mathbf{\underline{P}} \;\underline{\mathbf{\tilde{f}}}  - \mathbf{\underline{R}} \;  \underline{\u}^{n} ) $
			\EndWhile

			\State $\u = \u^{\uparrow, n} + \u^{\downarrow,n} $

		\EndFunction
	\end{algorithmic}
  \end{algorithm}
Alg. \ref{alg:Gauss_Seidel} is generic: the matrices $\mathbf{\underline{D}}$ and $\mathbf{\underline{R}}$ can either arise from Def. \ref{def:polarized_out_going} or Def. \ref{def:polarized_jump_condition}; and $\underline{\mathbf{\tilde{f}}}$ can either be $( \underline{\mathbf{f}},0)^t$ or $(\underline{\mathbf{f}}, \underline{\mathbf{f}}_0)^t$ depending on the system being solved.

\subsection{GMRES} \label{section:gmres}

Alg. \ref{alg:Gauss_Seidel} is primarily an iterative solver for the discrete integral system given by Def. \ref{def:integral_formulation}, and can be seen as only using the polarized system given by Eq. \ref{eq:polarized_jump_condition} in an auxiliary fashion. Unfortunately, the number of iterations needed for Alg. \ref{alg:Gauss_Seidel} to converge to a given tolerance often increases as a fractional power of the number of sub-domains. We address this problem by solving Eq. \ref{eq:polarized_jump_condition} in its own right, using GMRES combined with Alg. \ref{alg:Gauss_Seidel} as a preconditioner. As we illustrate in the sequel, the resulting number of iterations is now roughly constant in the number of subdomains. The preconditioner is defined as follows.

\begin{algorithm} Gauss-Seidel Preconditioner \label{alg:preconditioner_Gauss_Seidel}
	\begin{algorithmic}[1]
		\Function{ $\underline{\u}$ = Preconditioner}{ $\underline{\mathbf{\tilde{f}}}, n_{\text{it}}$}
			\State $\underline{\u}^0 = (\u^{\downarrow} , \u^{\uparrow} )^{t} =  0 $
			\For{  $ n = 0$, $n< n_{\text{it}}$, $n++$ }
				\State $\underline{\u}^{n+1} = (\mathbf{\underline{D}})^{-1}( \mathbf{\underline{P}}\;\underline{\mathbf{\tilde{f}}}  - \mathbf{\underline{R}} \;  \underline{\u}^{n} ) $
			\EndFor
			\State $\underline{\u} = \underline{\u}^{n_{\text{it}}}$
		\EndFunction
	\end{algorithmic}
\end{algorithm}

If we suppose that $n_{\text{it}} = 1$ (good choices are $n_{\text{it}} = 1$ or $2$), then the convergence of preconditioned GMRES will depend on the clustering of the eigenvalues of
\begin{equation}
	(\mathbf{\underline{D}}\,)^{-1} \; \mathbf{\underline{P}} \; \underline{\underline{\mathbf{M}}} = I + (\mathbf{\underline{D}}\,)^{-1} \mathbf{\underline{R}}  = \;
	 \left [   \begin{array}{cc}
																					I 																	&   (\mathbf{\underline{D}}^{\downarrow})^{-1}\mathbf{\underline{U}}  \\
																					 (\mathbf{\underline{D}}^{\uparrow})^{-1} \mathbf{\underline{L}}	& 	I
																					\end{array} \right ].
\end{equation}
We can compute these eigenvalues from the zeros of the characteristic polynomial. Using a well-known property of Schur complements, we get
\begin{equation}
	\det((\mathbf{\underline{D}}\,)^{-1}  \; \mathbf{\underline{P}} \;  \underline{\underline{\mathbf{M}}}  - \lambda \underline{I}) = \det(I - \lambda I) \det \left ( I - \lambda I - \left (  (\mathbf{\underline{D}}^{\uparrow})^{-1} \mathbf{\underline{L}}   (I - \lambda I)^{-1} (\mathbf{\underline{D}}^{\downarrow})^{-1}\mathbf{\underline{U}} \right )  \right ).
\end{equation}
This factorization means that half of the eigenvalues are exactly one. For the remaining half, we write the characteristic polynomial as
\begin{equation}
	 \det \left ( (I - \lambda I)^2 - \left (  (\mathbf{\underline{D}}^{\uparrow})^{-1} \mathbf{\underline{L}} (\mathbf{\underline{D}}^{\downarrow})^{-1}\mathbf{\underline{U}} \right )  \right ).
\end{equation}
Let $\mu = (1-\lambda)^2$, and consider the new eigenvalue problem
\begin{equation}
	 \det \left ( \mu I - \left (  (\mathbf{\underline{D}}^{\uparrow})^{-1} \mathbf{\underline{L}} \, (\mathbf{\underline{D}}^{\downarrow})^{-1}\mathbf{\underline{U}} \right )  \right ) = 0.
\end{equation}
Hence the eigenvalues of $(\mathbf{\underline{D}}\,)^{-1}  \; \mathbf{\underline{P}} \;  \underline{\underline{\mathbf{M}}}$ are given by $1$ and $1 \pm \sqrt{\mu_i}$, where $\mu_i$ are the eigenvalues of   $ (\mathbf{\underline{D}}^{\uparrow})^{-1} \mathbf{\underline{L}} \, (\mathbf{\underline{D}}^{\downarrow})^{-1}\mathbf{\underline{U}}$, and $\pm \sqrt{\mu_i}$ is either complex square root of $\mu_i$. Notice that $\pm \sqrt{\mu_i}$ coincide with the eigenvalues of the iteration matrix $ (\mathbf{\underline{D}}\,)^{-1} \underline{\mathbf{R}}$ of the Gauss-Seidel method.

The smaller the bulk of the $|\mu_i|$, the more clustered $1 \pm \sqrt{\mu_i}$ around 1, the faster the convergence of preconditioned GMRES. This intuitive notion of clustering is robust to the presence of a small number of outliers with large $|\mu_i|$. The spectral radius $\rho \left( (\mathbf{\underline{D}}\,)^{-1} \underline{\mathbf{R}} \right) = \max_i \sqrt{|\mu_i|}$, however, is not robust to outlying $\mu_i$, hence our remark about preconditioned GMRES being superior to Gauss-Seidel. These outliers do occur in practice in heterogeneous media; see Fig. \ref{fig:eigen_marmousi}.

To give a physical interpretation to the eigenvalues $\mu_i$ of $(\mathbf{\underline{D}}^{\uparrow})^{-1} \mathbf{\underline{L}} \, (\mathbf{\underline{D}}^{\downarrow})^{-1}\mathbf{\underline{U}}$, consider the role of each block:
\bit
\item $\mathbf{\underline{U}}$ maps $\mathbf{u}^{\uparrow}$ to $\mathbf{u}^{\downarrow}$ traces within a layer; it takes into account all the \emph{reflections} and other scattering phenomena that turn waves from up-going to down-going inside the layer. Vice-versa for $\mathbf{\underline{L}}$, which maps down-going to up-going traces.
\item $(\mathbf{\underline{D}}^{\downarrow})^{-1}$ maps down-going traces at interfaces to down-going traces at all the other layers below it; it is a \emph{transmission} of down-going waves in a top-down sweep of the domain. Vice-versa for $(\mathbf{\underline{D}}^{\uparrow})^{-1} $, which transmits up-going waves in a bottom-up sweep. It is easy to check that transmission is done via the computation of incomplete discrete GRF.
\eit
Hence the combination $(\mathbf{\underline{D}}^{\uparrow})^{-1} \mathbf{\underline{L}} \, (\mathbf{\underline{D}}^{\downarrow})^{-1}\mathbf{\underline{U}}$ generates reflected waves from the heterogeneities within a layer, propagates them down to every other layer below it, reflects them again from the heterogeneities within each of those layers, and propagates the result back up through the domain. The magnitude of the succession of these operations is akin to a coefficient of ``double reflection" accounting for scattering through the whole domain.  We therefore expect the size of the eigenvalues $|\mu_i|$ to be proportional to the strength of the medium heterogeneities, including how far the PML are to implementing absorbing boundary conditions. The numerical exeriments support this interpretation.

As an example, Fig. \ref{fig:eigen_homogeneous} shows the eigenvalues of $(\mathbf{\underline{D}}\,)^{-1} \; \mathbf{\underline{P}} \; \underline{\underline{\mathbf{M}}}$ when the media is homogeneous, but with a PML of varying quality. Fig \ref{fig:eigen_marmousi} shows the eigenvalues of  $(\mathbf{\underline{D}}\,)^{-1} \; \mathbf{\underline{P}} \; \underline{\underline{\mathbf{M}}}$ in  a rough medium, and again with PML of varying quality.

We can expect that, as long as the medium is free of resonant cavities, the eigenvalues of the preconditioned matrix will cluster around $1$, implying that GMRES will converge fast. Assuming that the reflection coefficients depend weakly on the frequency, we can expect that the performance of the preconditioner will deteriorate no more than marginally as the frequency increases.

\begin{figure}[H]
\begin{center}
 \includegraphics[trim = 27mm 27mm 28mm 12mm, clip, width = 50mm]{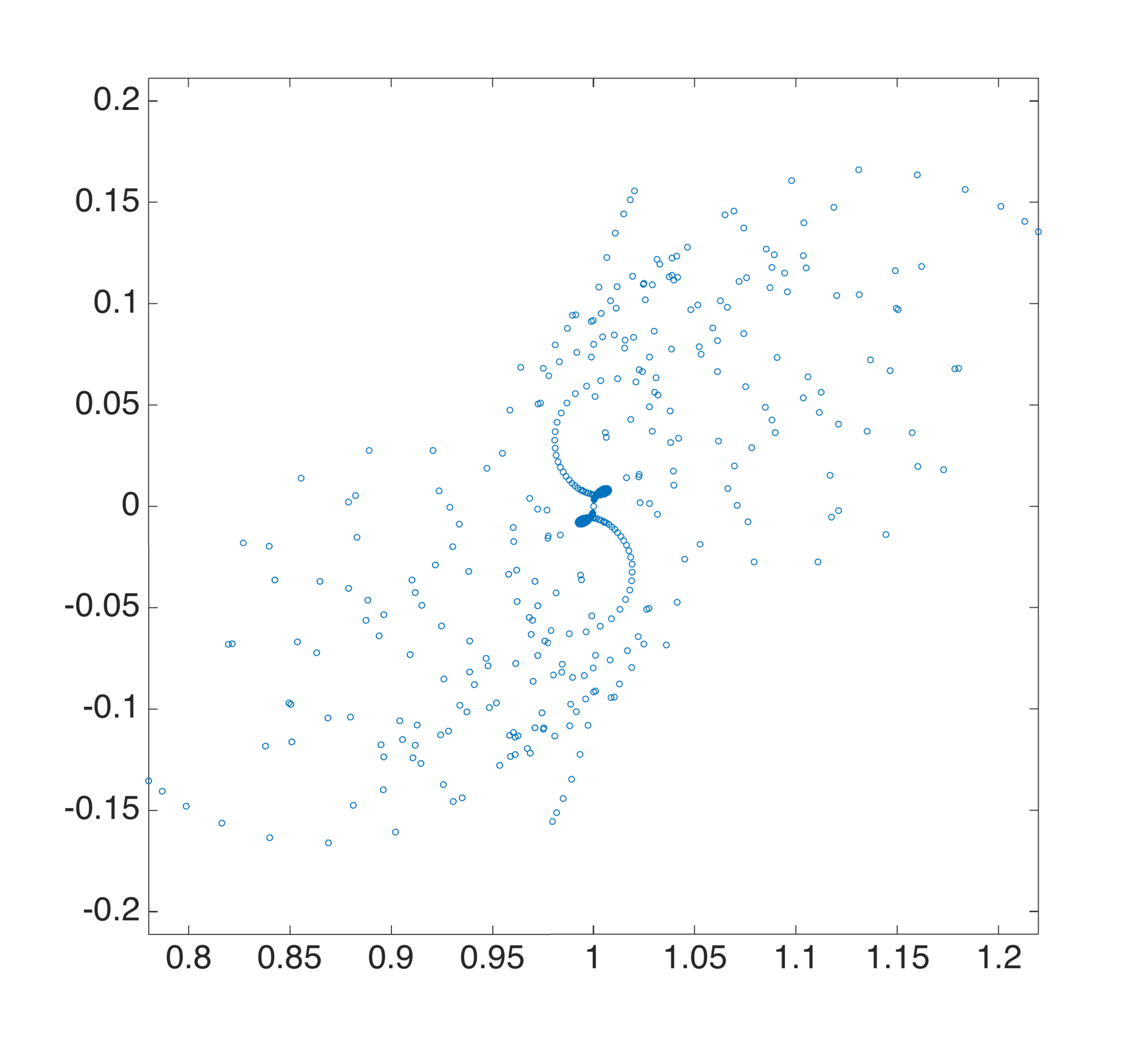}  \includegraphics[trim = 27mm 27mm 28mm 12mm, clip, width = 50mm]{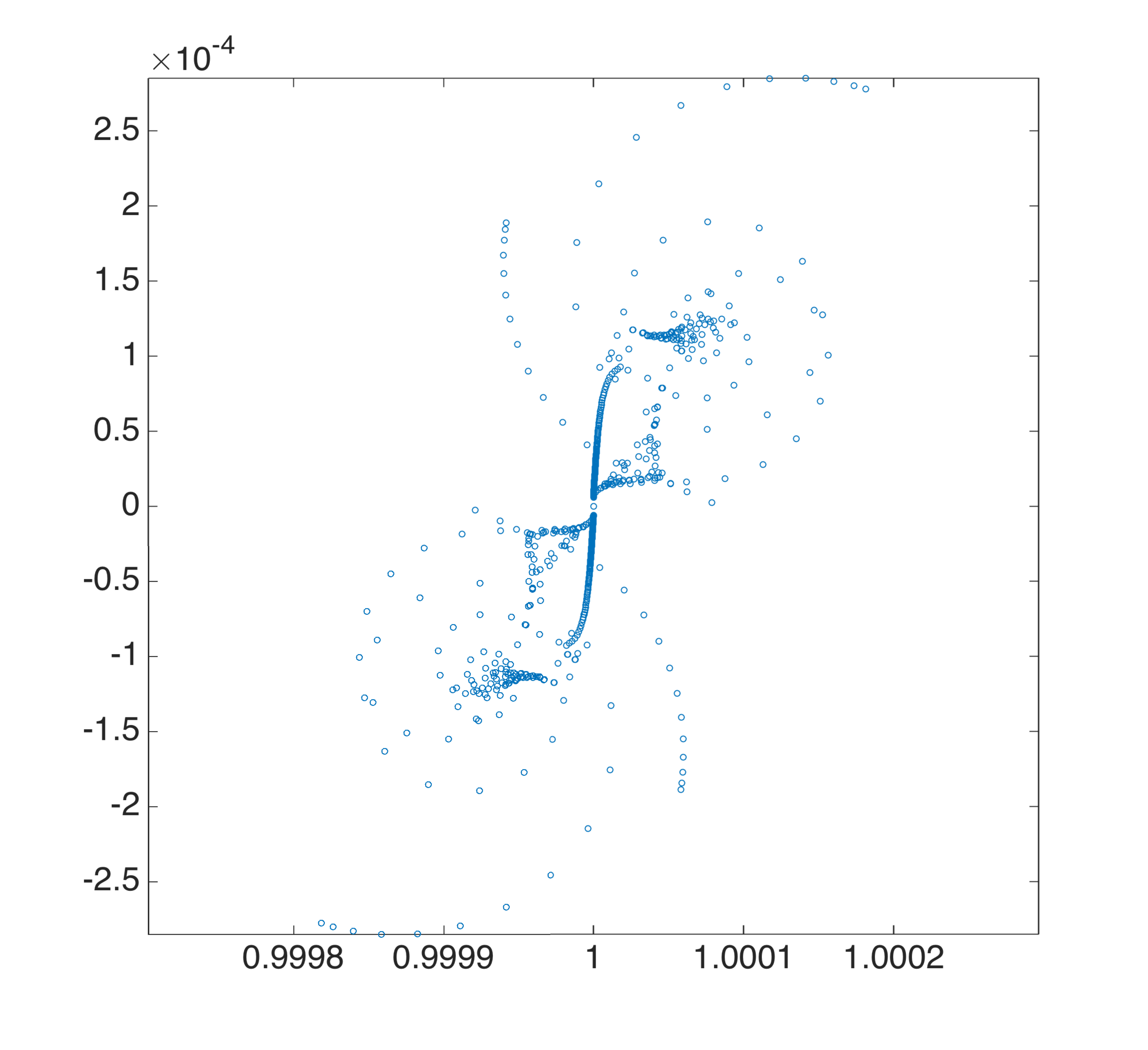} \includegraphics[trim = 27mm 27mm 28mm 12mm, clip, width = 50mm]{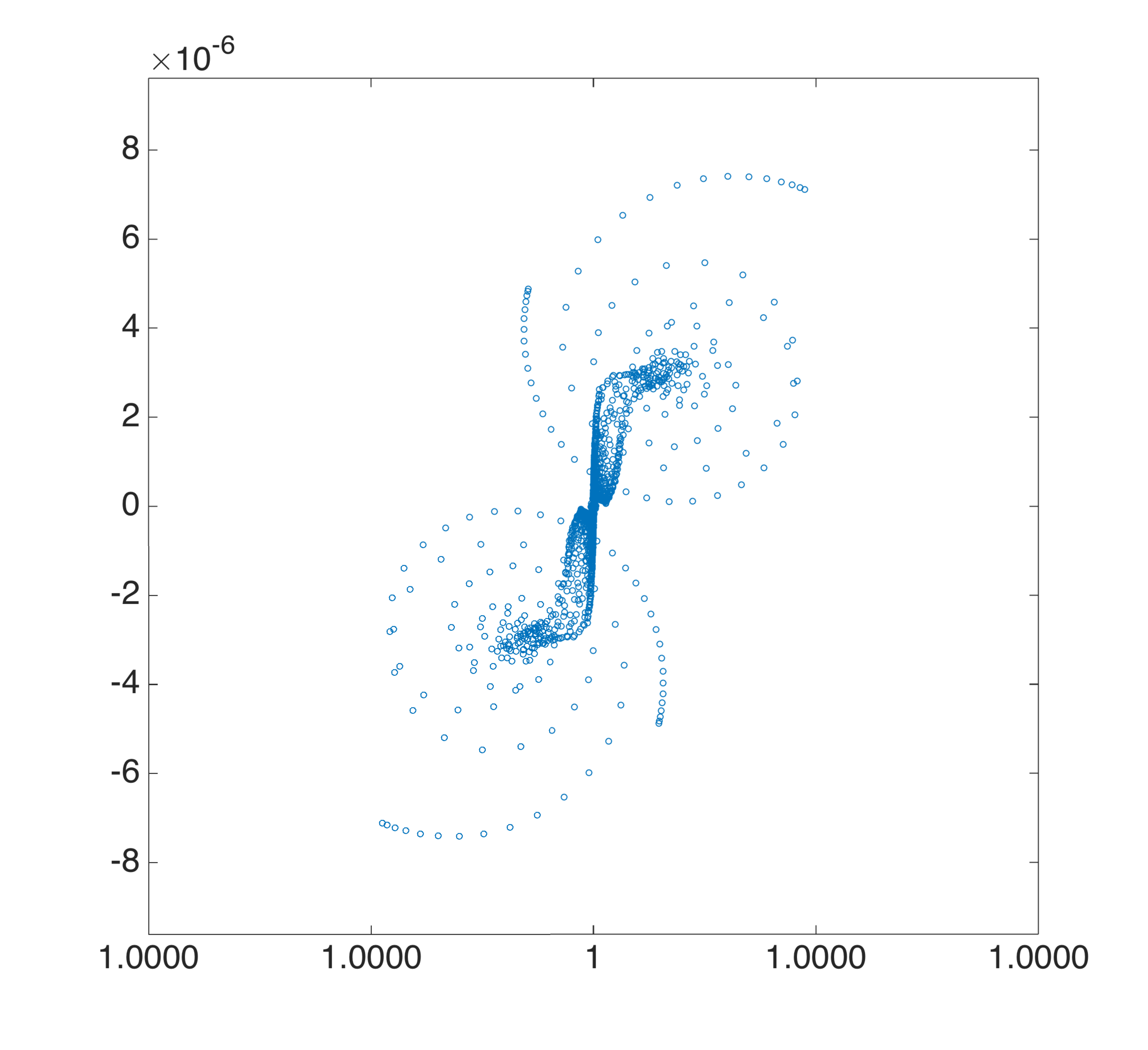}
  \caption{ Eigenvalues of $(\mathbf{\underline{D}}\,)^{-1} \; \mathbf{\underline{P}} \; \underline{\underline{\mathbf{M}}}$, for a homogeneous media , $L=3$, $n = $, $\omega = 30$; and $5$ ({\it left}), $30$ ({\it center}) and $100$ ({\it right}) PML points. Notice the scale of the axes. }
  \label{fig:eigen_homogeneous}
 \end{center}
\end{figure}

\begin{figure}[H]
\begin{center}
 \includegraphics[trim = 27mm 27mm 28mm 12mm, clip, width = 50mm]{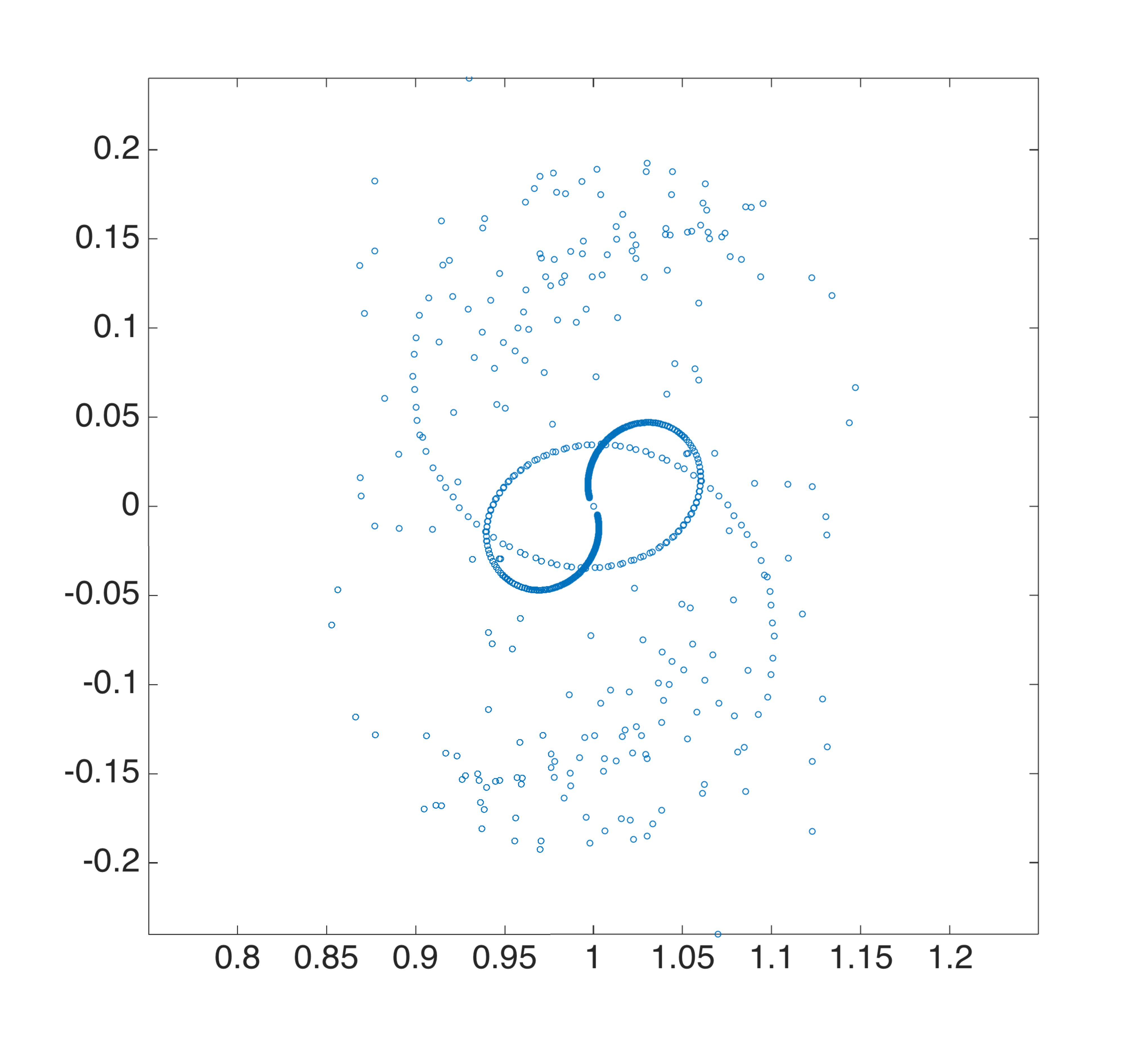}  \includegraphics[trim = 27mm 27mm 28mm 12mm, clip, width = 50mm]{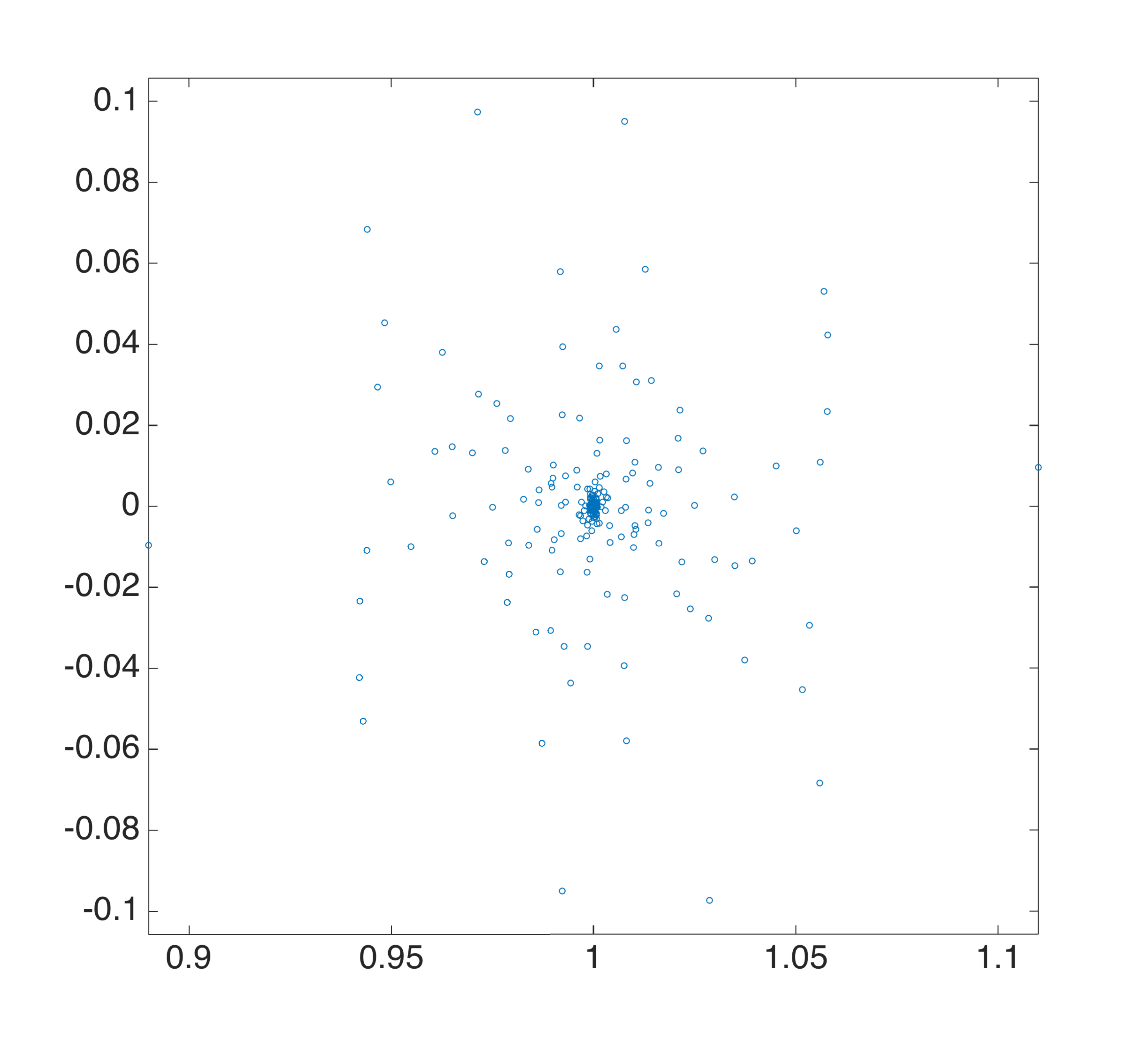}   \includegraphics[trim = 27mm 27mm 28mm 12mm, clip, width = 50mm]{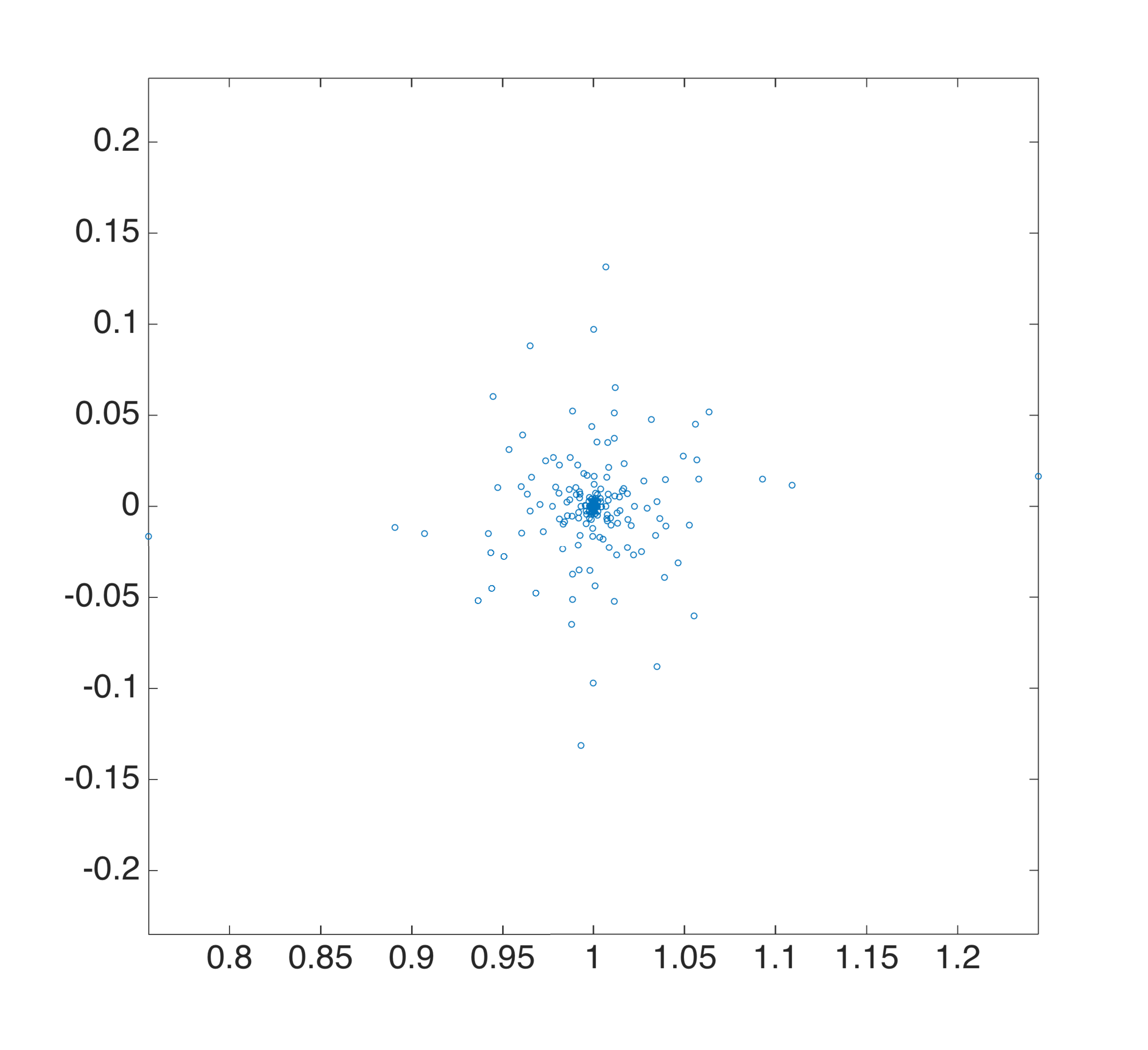}
  \caption{ Eigenvalues of $(\mathbf{\underline{D}}\,)^{-1} \; \mathbf{\underline{P}} \; \underline{\underline{\mathbf{M}}}$, for the Marmousi2 model (Fig. \ref{fig:Marmousi_2}), $L=3$, $n = 300$, $\omega = 30$; and $5$ ({\it left}), $30$ ({\it center}) and $100$ ({\it right}) PML points.  }
  \label{fig:eigen_marmousi}
 \end{center}
\end{figure}

\section{Partitioned low-rank matrices} \label{section:PLR}

Let $n \sim \sqrt{N}$ for the number of points per dimension, and $L$ be the number of subdomains. So far, the complexity of solving the discrete integral system in polarized form, assuming a constant number of preconditioned GMRES iterations, is dominated by the application of $\mathbf{\underline{D}}^{-1}$ and $\mathbf{\underline{R}}$ in the update step $\underline{\u}^{n+1} = (\mathbf{\underline{D}})^{-1}( \mathbf{\underline{P}}\;\underline{\mathbf{\tilde{f}}}  - \mathbf{\underline{R}} \;  \underline{\u}^{n} ) $. Each of the $\cO(L)$ nonzero blocks of $\mathbf{\underline{D}}^{-1}$ and $ \mathbf{\underline{R}}$ is $n$-by-$n$. A constant number of applications of these matrices therefore results in a complexity that scales as $\cO(n^2 L)$. This scaling is at best linear in the number $N$ of volume unknowns.

It is the availability of fast algorithms for $\mathbf{\underline{D}}^{-1}$ and $\underline{\mathbf{R}}$, i.e., for the blocks of $\underline{\underline{\mathbf{M}}}$, that can potentially lower the complexity of solving Eq. \ref{eq:integral_formulation} down to sublinear in $N$. In this setting, the best achievable complexity would be $\cO(n L)$ -- the complexity of specifying $\cO(L)$ traces of size $\cO(n)$. As mentioned earlier, the overall online complexity can be sublinear in $N$ if we parallelize the operations of reading off $\f$ in the volume, and forming the volume unknowns $\u$ from the traces.


We opt for what is perhaps the simplest and best-known algorithm for fast application of arbitrary kernels in discrete form: an adaptive low-rank partitioning of the matrix. This choice is neither original nor optimal in the high-frequency regime, but it gives rise to elementary code. More sophisticated approaches have been proposed elsewhere, including by one of us in \cite{Candes_Demanet_Ying:A_Fast_Butterfly_Algorithm_for_the_Computation_of_Fourier_Integral_Operators}, but the extension of those ideas to the kernel-independent framework is not immediate. This section is therefore added for the sake of algorithmic completeness, and for clarification of the ranks and complexity scalings that arise from low-rank partitioning in the high-frequency regime.



\subsection{Compression}

The blocks of $\underline{\underline{\mathbf{M}}}$, which stem from the discretization of interface-to-interface operators, are compressed using the recursive Alg. \ref{alg:PLR_matrix}. The result of this algorithm is a quadtree structure on the original matrix, where the leaves are maximally large square submatrices with fixed $\epsilon$-rank. We follow \cite{Beylkin:wave_propagation_using_bases_for_bandlimited_functions} in calling this structure partitioned low-rank (PLR). An early reference for PLR matrices is the work of Jones, Ma, and Rokhlin in 1994 \cite{Jones_Ma_Rokhlin:A_Fast_Direct_Algorithm_for_the_Solution_of_the_Laplace_Equation_on_Regions_with_Fractal_Boundaries}. PLR matrices are a special case of $\cH$-matrices\footnote{We reserve the term $\cH$-matrix for hierarchical structures preserved by algebraic operations like multiplication and inversion, like in \cite{Bebendorf:2008}.}.

It is known \cite{Bebendorf:2008}, that the blocks of matrices such as $\underline{\underline{\mathbf{M}}}$ can have low rank, provided they obey an admissibility condition that takes into account the distance between blocks. In regimes of high frequencies and rough heterogeneous media, this admissibility condition becomes more stringent in ways that are not entirely understood yet. See \cite{Candes_Demanet_Ying:A_Fast_Butterfly_Algorithm_for_the_Computation_of_Fourier_Integral_Operators, DemanetYing:FIO} for partial progress.

Neither of the usual non-adaptive admissibility criteria seems adequate in our case. The ``nearby interaction" blocks (from one interface to itself) have a singularity along the diagonal, but tend to otherwise have large low-rank blocks in media close to uniform \cite{MartinssonRohklin:a_fast_direct_solver_for_scattering_problems_involving_elongated_structures}. The ``remote interaction" blocks (from one interface to the next) do not have the diagonal problem, but have a wave vector diversity that generates higher ranks. Adaptivity is therefore a very natural choice, and is not a problem in situations where the only operation of interest is the matrix-vector product.


For a fixed accuracy $\epsilon$ and a fixed rank $r_{\text{max}}$, we say that a partition is admissible if every block has $\epsilon$-rank less or equal than $r_{\text{max}}$. Alg. \ref{alg:PLR_matrix} finds the smallest admissible partition within the quadtree generated by recursive dyadic partitioning of the indices.

\begin{algorithm} Partitioned Low Rank matrix \label{alg:PLR_matrix}
\begin{algorithmic}[1]
\Function{H = PLR}{$M$, $r_{\text{max}}$, $\epsilon$}
\State  $[U, \Sigma ,V] = \texttt{svds}(M, r_{\text{max}}+1)$
\If{ $\Sigma(r_{\text{max}}+1,r_{\text{max}}+1) < \epsilon$}
\State \texttt{H.data = \{$U\cdot \Sigma$, $V^{t}$\} }
\State \texttt{H.id = `c' }    \Comment{leaf node}
\Else
\State $M = \left[ \begin{array}{cc}
										M_{1,1} & M_{1,2}  \\
										M_{2,1} & M_{2,2}
					\end{array}
			 \right ] $  		\Comment{block partitioning}
\For{i = 1:2}
\For{j = 1:2}
\State \texttt{H.data\{i,j\} = } PLR($M_{i,j}$, $r_{\text{max}}$, $\epsilon$ )
\EndFor
\EndFor
\State \texttt{H.id =  `h' }     \Comment{branch node}
\EndIf
\EndFunction
\end{algorithmic}
  \end{algorithm}


  Fig. \ref{fig:H_matrices_representation} depicts the hierarchical representation of a PLR matrix of the compressed matrix for the nearby interactions {\it (left)} and the remote interactions {\it (right)}.
  Once the matrices are compressed in PLR form, we can easily define a fast matrix-vector multiplication using Alg. \ref{alg:PLR_matvec}.

  \begin{algorithm} PLR-vector Multiplication\label{alg:PLR_matvec}

\begin{algorithmic}[1]
\Function{y = matvec}{x}

\If{ \texttt{H.id == `c' }  } 																				\Comment{If leaf node}
\State y =  \texttt{H.data\{1\}$\cdot$(H.data\{2\}$\cdot$x) }												\Comment{perform mat-vec using SVD factors}
\Else 																										\Comment{If branch node}
\For{i = 1:2}																								\Comment{recurse over children}
\State $\texttt{y}_{1}\texttt{+=}\text{matvec}\texttt{(H.data\{i,1\},}\texttt{x}_{\texttt{1:end/2}})$
\State $\texttt{y}_{2}\texttt{+=}\text{matvec}\texttt{(H.data\{i,2\},}\texttt{x}_{\texttt{end/2:end}})$
\EndFor
\State \texttt{y = $ \left[ \begin{array}{c}
										\texttt{y}_1\\
										\texttt{y}_2
										\end{array}
							  \right ]$ } 																	\Comment{concatenate solution from recursion}
\EndIf
\EndFunction
\end{algorithmic}
  \end{algorithm}

\begin{figure}[H]
\begin{center}
 \includegraphics[trim = 20mm 25mm 20mm 20mm, clip, width = 80mm]{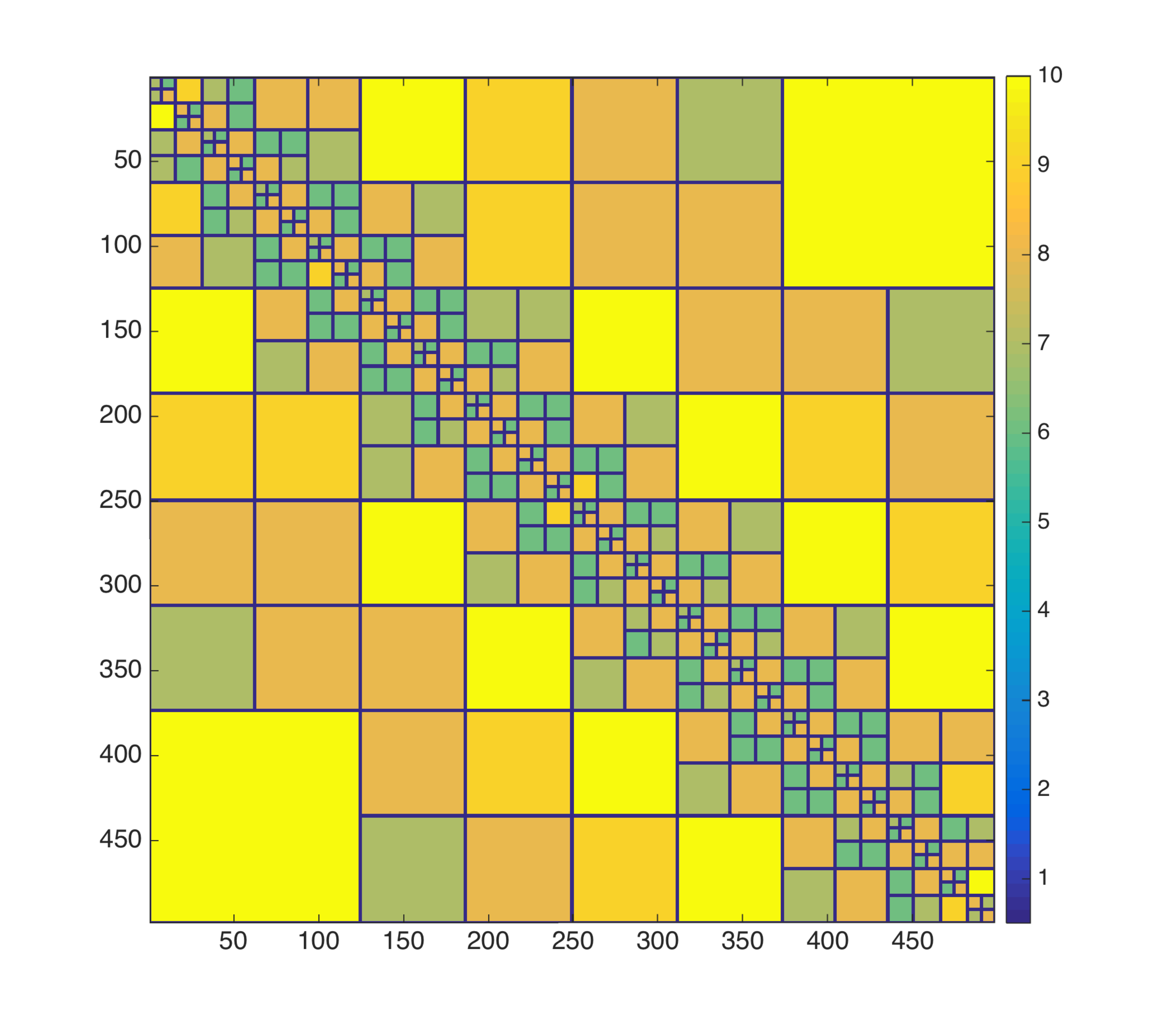}  \includegraphics[trim = 20mm 25mm 20mm 20mm, clip, width = 80mm]{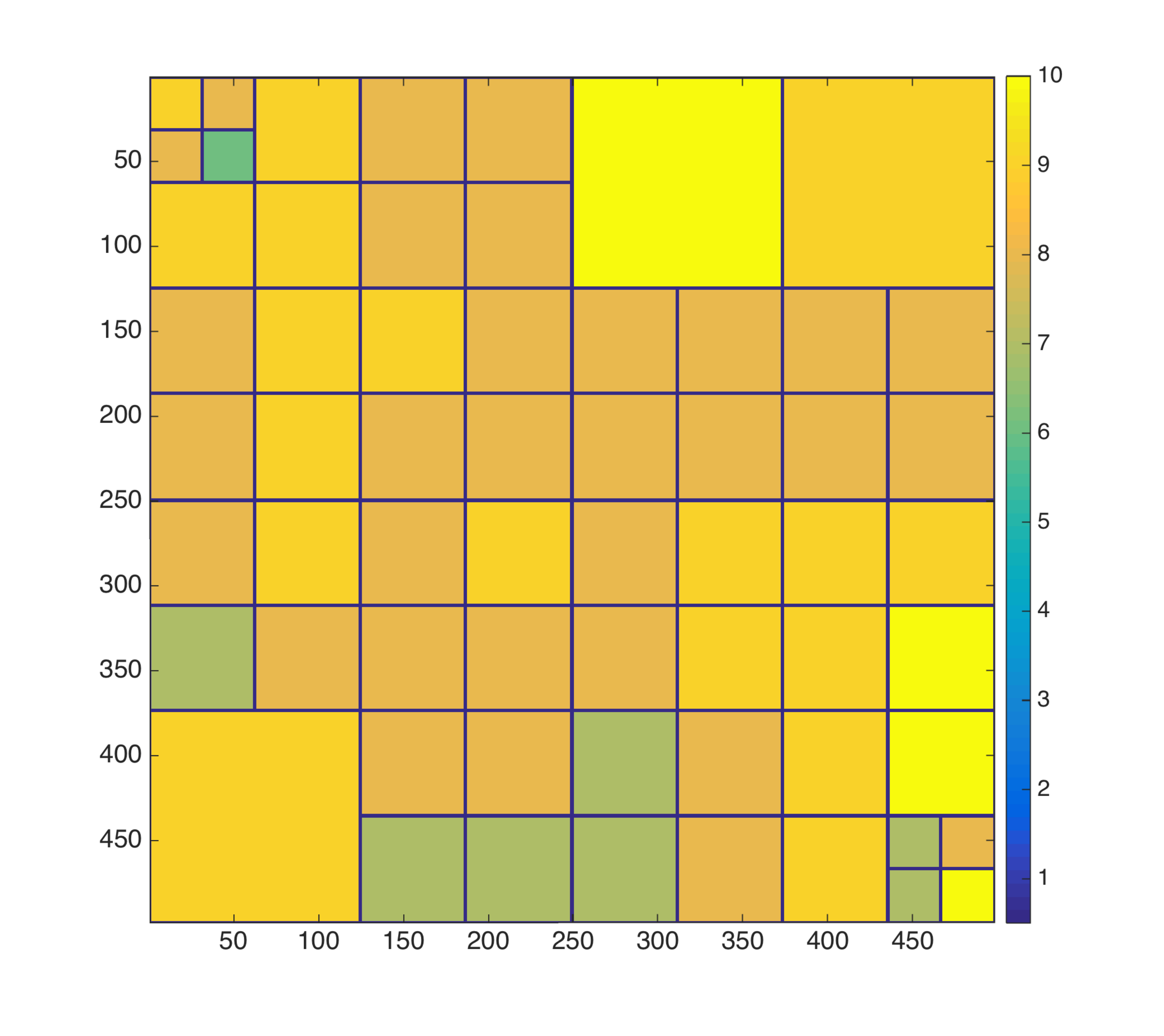}
 \caption{Illustration of compressed Green's matrices in PLR form ($\epsilon \text{-ranks} \leq 10$, $\epsilon = 10^{-9}$). Each color represents a different numerical rank. Left: nearby interactions. Right: remote interactions. }  \label{fig:H_matrices_representation}
 \end{center}
\end{figure}

Alg. \ref{alg:PLR_matvec} yields a fast matrix vector multiplication; however, given its recursive nature the constant for the scaling can become large. This phenomenon is overwhelming if Alg. \ref{alg:PLR_matvec} is implemented in a scripting language such as MATLAB or Python. In the case of compiled languages, recursions tend to not be correctly optimized by the compiler, increasing the constant in front of the asymptotic complexity scaling. To reduce the constants, the PLR matrix-vector multiplication was implemented via a sparse factorization as illustrated by Fig \ref{fig:PLR_2_sparse}. This factorization allows us to take advantage of highly optimized sparse multiplication routines. An extra advantage of using such routines is data locality and optimized cache management when performing several matrix-vector multiplications at the same time, which can be performed as a matrix-matrix multiplication. This kind of sparse factorization is by no means new; we suggest as a reference Section 4 in \cite{Sivaram_Darve:HODLR}.

\begin{figure}[H]
\begin{center}
 \includegraphics[trim = 0mm 0mm 0mm 0mm, clip, width = 140mm]{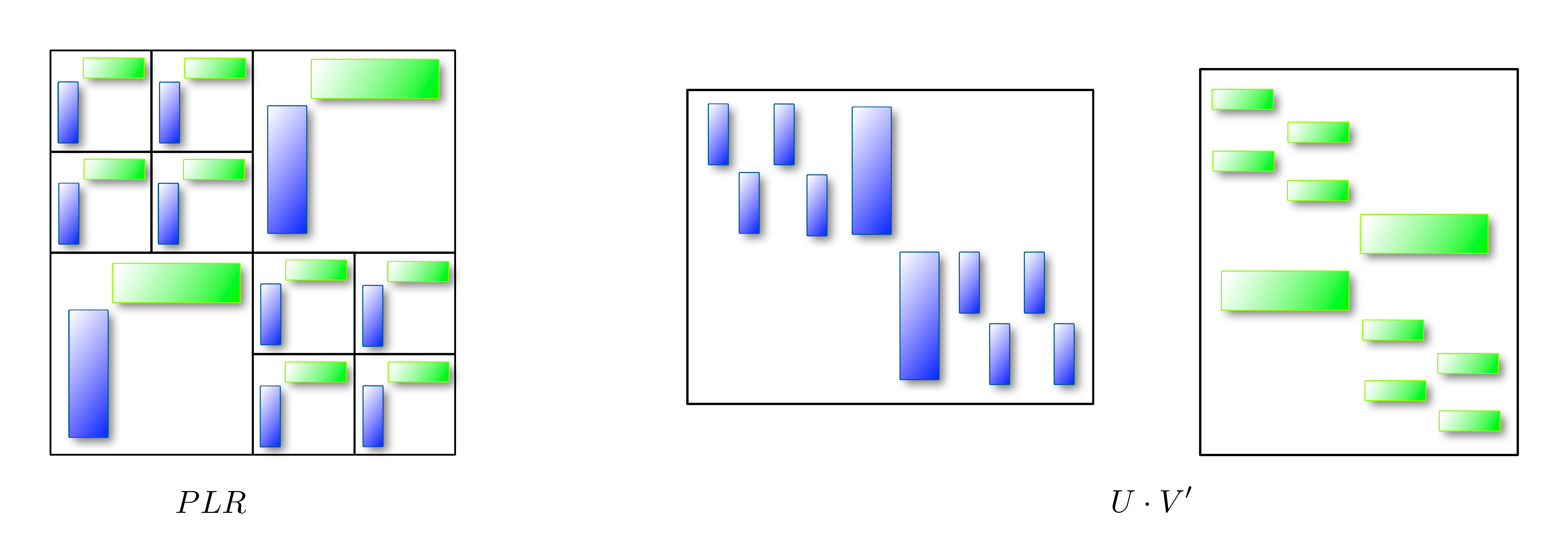}
  \caption{Illustration of the sparse form of a PLR matrix. Left: PLR matrix. Right: its sparse factorization form.  }
  \label{fig:PLR_2_sparse}
 \end{center}
\end{figure}

The maximum local rank of the compression scheme can be seen as a tuning parameter. If it is too small, it will induce small off-diagonal blocks, hindering compression and deteriorating the complexity of the matrix-vector multiplication. On the other hand, a large maximum rank will induce a partition with big dense diagonal blocks that should be further compressed, resulting in the same adverse consequences.

\subsection{Compression scalings} \label{subsection:compression_scalings}

It is difficult to analyze the compressibility of Green's functions without access to an explicit formula for their kernel. In this section we make the assumption of smooth media and single-valued traveltimes, for which a geometrical optics approximation such as
\beq\label{eq:GO}
G(\x,\y; \omega) \simeq a_{\omega}(\x,\y) e^{i\omega \tau(\x,\y)},
\eeq
holds. Here $\tau(\x, \y)$ solves an eikonal equation; and $a_{\omega}(\x,\y)$ is an amplitude factor, smooth except at $\x = \y$, and with a minor\footnote{In the standard geometrical optics asymptotic expansion, $a \sim \sum_{j \geq 0} a_j \omega^{-j}$ is polyhomogeneous with increasingly negative orders in $\omega$.} dependence on $\omega$. Assume that both $a$ and $\tau$ are $C^{\infty}_{\x,\y}$ away from $x=y$, with smoothness constants bounded independently of $\omega$.

The following result is a straightforward generalization of a result in \cite{Candes_Demanet_Ying:A_Fast_Butterfly_Algorithm_for_the_Computation_of_Fourier_Integral_Operators, DemanetYing:FIO} and would be proved in much the same way.

\begin{lemma} (High-frequency admissibility condition)
Consider $G$ as in Eq. \ref{eq:GO}, with $\x \in A$ and $\y \in B$ where $A$ and $B$ are two rectangles. Let $d_A$, $d_B$ be the respective diameters of $A$ and $B$.  If
\[
d_A d_B \leq \frac{\mbox{dist}(A,B)}{\omega},
\]
then the $\epsilon$-rank of the restriction of $G$ to $A \times B$ is, for some $R_{\epsilon} > 0$, bounded by $R_{\epsilon}$ (independent of $\omega$).
\end{lemma}

The scaling is often tight, except when dealing with 1D geometries and uniform media as in\cite{MartinssonRohklin:a_fast_direct_solver_for_scattering_problems_involving_elongated_structures}.

In our case, we refer to ``nearby interactions" as the case when $\x$ and $\y$ are on the same interface (horizontal edge), and ``remote interactions" when $\x$ and $\y$ are on opposite horizontal edges of a layer $\Omega^{\ell}$.
In both cases, $A$ and $B$ are 1D segments, but the geometry is 2D for the remote interactions.  Our partitioning scheme limits $A$ and $B$ to have the same length $d_A = d_B$, hence the lemma implies that low ranks can in general only occur provided
\[
d_A \leq C \, \frac{1}{\sqrt{\omega}}.
\]
In other words, $d_A$ is at best proportional to the square root of the (representative) spatial wavelength, and even smaller when $A$ and $B$ are close. A square block of the discrete $\mathbf{G}$ with $d_A \sim  \frac{1}{\sqrt{\omega}}$, on a grid with $n$ points per dimension, would be of size $\sim \frac{n}{\sqrt{\omega}} \times \frac{n}{\sqrt{\omega}}$. We call such a block representative; it suffices to understand the complexity scalings under the assumption that all blocks are representative\footnote{The complexity overhead generated by smaller blocks close to the diagonal only increase the overall constant, not the asymptotic rate, much as in the fast multipole method.}.

Let $\omega \sim n^{\rho}$ for some $\rho$ such as $\frac{1}{2}$ or $1$. This implies that the representative block of $\mathbf{G}$ has size $n^{1-\rho/2} \times n^{1-\rho/2}$ (where we drop the constants for simplicity). Given that the $\epsilon$-rank is bounded by $R_{\epsilon}$, this block can be compressed using Alg \ref{alg:PLR_matrix} in two matrices of size $n^{1-\rho/2} \times R_{\epsilon}$ and  $ R_{\epsilon} \times n^{1-\rho/2}$. We have $\cO(n^{\rho})$ such blocks in $\mathbf{G}$.

We can easily compute an estimate for the compression ratio; we need to store $\cO( 2 R_{\epsilon} n^{1+\rho/2})$ complex numbers for the PLR compressed $\mathbf{G}$. Thus, the compression ratio is given by  $\frac{2 R_{\epsilon} n^{1+\rho/2}}{ n^2} \sim n^{ \rho/2 -1}$.

Moreover, multiplying each block by a vector has asymptotic complexity $ 2 R_{\epsilon} n^{1 - \rho/2}$, so that the overall asymptotic complexity of the PLR matrix vector multiplication is given by $ 2 R_{\epsilon} n^{1 + \rho/2}$.

If $\rho = 1/2$ and $N \sim n^2$, we have that the asymptotic complexity of the PLR matrix-vector multiplication is given by $ 2 R_{\epsilon} n^{5/4} \sim N^{5/8}$. If $\rho = 1$, the complexity becomes $\sim N^{3/4}$.

The estimate for the asymptotic complexity relies on a perfect knowledge of the phase functions, which is unrealistic for a finite difference approximation of the Green's functions. We empirically observe a deterioration of these scalings due to the discretization error, as shown in Table \ref{table:compression_scaling}.

One possible practical fix for this problem is to allow the maximum ranks in the compression to grow as $\sqrt{n}$. This small adjustment allows us to reduce the complexity in our numerical experiments; though the theoretical predictions of those scalings (from a completely analogous analysis) are quite unfavorable. The scalings we observe numerically are reported in square brackets in the table. They are \emph{pre-asymptotic} and misleadingly good: if the frequency and $N$ were both increased, higher scaling exponents would be observed\footnote{The same pre-asymptotic phenomenon occurs in the $n$-by-$n$ FFT matrix: a block of size $n/2$ by $n/2$ will look like it has $\epsilon$-rank $\cO(n^{7/8})$, until $n \sim 2^{14}$ and beyond, where the $\epsilon$-rank because very close to $n/4$.}. The correct numbers are without square brackets.

\begin{table}
\begin{center}
\begin{tabular}{|c|c|c|c|c|}
\hline
Step 				&  $\omega \sim \sqrt{n} \;\; \vert \;\; r \sim 1$   & $\omega \sim n \;\; \vert \;\; r \sim 1$   & $\omega \sim \sqrt{n} \;\; \vert \;\; r \sim \sqrt{n} $ & $\omega \sim n \;\; \vert \;\; r \sim \sqrt{n} $ \\
\hline
Analytic 			&  $\cO(N^{5/8})$ 		 				& $\cO(N^{3/4})$   				& $ \left [ \cO(N^{5/8}) \right ]$  $\to \cO(N^{3/4})$    & $ \left [ \cO(N^{7/8})\right ]$  $\to \cO(N)$   \\
\hline
Finite Differences 	&  $\cO(N^{3/4})$    					& $\cO(N^{7/8})$        		& $ \left [ \cO(N^{5/8}) \right ]$     &  $ \left [ \cO(N^{7/8})\right ]$     \\
\hline
\end{tabular}
\caption{Compression scaling for the remote interactions, sampling a typical oscillatory kernel (Analytic) and using the finite differences approximation (FD). The observed pre-asymptotic complexities are in square brackets.}\label{table:compression_scaling}
\end{center}
\end{table}

\section{Computational Complexity} \label{sec:complexity}

The complexities of the various steps of the algorithm were presented in section \ref{sec:intro_complexity} and are summarized in Table \ref{table:complexity}. In this section we provide details about the heuristics and evidence supporting these complexity claims.

\subsection{Computational cost}

  For the five-point discretization of the 2D Laplacian, the sparse LU factorization with nested dissection is known to have $\cO(N^{3/2})$ complexity, and the back-substitution is known to have linear complexity up to logarithmic factors, see \cite{GeorgeNested_dissection,Hoffman:Complexity_Bounds_for_Regular_Finite_Difference_and_Finite_Element_Grids}. Given that each sub-domain has size $\cO(N/L)$, the factorization cost is $\cO((N/L)^{3/2})$. Moreover, given that the factorizations are completely independent, they can be computed simultaneously in $L$ different nodes. We used the stored LU factors to compute, by back-substitution, the local solutions needed for constructing the right-hand side of Eq. \ref{eq:integral_formulation} and the reconstruction of the local solutions, leading to a complexity $\cO(N \log(N)/L)$ per solve. The local solves are independent and can be performed in $L$ different nodes.

  To compute the Green's functions from the LU factors, we need to perform $\cO(n)$ solves per layer. Each solve is completely independent of the others, and they can be carried in parallel in $nL$ different processors. We point out that there is a clear trade-off between the parallel computation of the Green's functions and the communication cost. It is possible to compute all the solves of a layer in one node using shared memory parallelism, which reduces scalability; or compute the solves in several nodes, which increases scalability but increases the communication costs. The best balance between both approaches will heavily depend on the architecture and we leave the details out of this presentation.

  We point out that the current complexity bottleneck is the extraction of the Green's functions. It may be possible to reduce the number of solves to a fractional power of $n$ using randomized algorithms such as \cite{Lin_Ying:Fast_Construction_of_Hierarchical_Matrix_Representation_from_Matrix_vector_Multiplication}); or to $\cO(1)$ solves, if more information of the underlying PDE is available or easy to compute (see \cite{RosalieLaurent:compressed_PML}).

  Once the Green's matrices are computed, we use Alg. \ref{alg:PLR_matrix} to compress them. A simple complexity count indicates that the asymptotic complexity of the recursive adaptive compression is bounded by $\cO(\left( N\log(N) \right))$ for each integral kernel\footnote{Assuming the svds operation is performed with a randomized SVD.}.  The compression of each integral kernel is independent of the others, and given that the Green's functions are local to each node, no communication is needed.

There are two limitations to the compression level of the blocks of $\underline{\underline{\mathbf{M}}}$. One is theoretical, and was explained in section \ref{subsection:compression_scalings}; while the second is numerical. In general, we do not have access to the Green's functions, only to a numerical approximation, and the numerical errors present in those approximations hinder the compressibility of the operators for large $\omega$. Faithfulness of the discretization is the reason why we considered milder scalings of $\omega$ as a function of $n$, such as $\omega \sim \sqrt{n}$, in the previous section
\footnote{To deduce the correct scaling between $\omega$ and $n$, we use the fact that the second order five point stencil scheme has a truncation error dominated by $h^2 (\partial_x^4 + \partial_y^4) u \sim (\omega/c)^4h^2 $. Given that $h\sim 1/n$, we need $\omega \sim \sqrt{n}$ in order to have a bounded error in the approximation of the Green's function. In general, the scaling needed for a p-th order scheme to obtain a bounded truncation error is $\omega \sim n^{\frac{p}{p+2}}$. }.


With the scaling $\omega \sim \sqrt{n}$, we have that the complexity of matrix-vector product is dominated by $\cO(N^{5/8})$ for each block of  $\underline{\underline{\mathbf{M}}}$ (see subsection \ref{subsection:compression_scalings}). Then, the overall complexity of the GMRES iteration is $\cO(LN^{5/8})$, provided that the number of iterations remains bounded.

Within an HPC environment we are allowed to scale $L$, the number of sub-domains, as a small fractional power of $N$. If $L \sim N^{\delta}$, then the overall execution time of the online computation, take away the communication cost, is $\cO(N^{\max (\delta + 5/8, (1-\delta))})$. Hence, if $0 < \delta < 3/8$, then the online algorithm runs in sub-linear time. In particular, if $\delta = 3/16$, then the runtime of the online part of the solver becomes $\cO(N^{13/16})$ plus a sub-linear communication cost.

If the matrix-vector product is $\cO(N^{3/4})$ (a more representative figure in Table \ref{table:compression_scaling}), then the same argument leads to $\delta = 1/8$ and an overall online runtime of $\cO(N^{7/8})$.

\subsection{Communication cost}
We use a commonly-used communication cost model (\cite{BallardDemmel:minimizing_communication_in_numerical_linear_algebra,DemmelGrigori:communication_avoiding_rank_revealing_qr_factorization_with_column_pivoting,Poulson_Demanet:a_parallel_butterfly_algorithm,Thakur:optimization_of_collective_communication_operations_in_mpich})  to perform the  cost analysis. The model assumes that each process is only able to send or receive a single message at a time. When the messages has size $m$, the time to communicate that message is $\alpha + \beta m$. The parameter $\alpha$ is called latency, and represents the minimum time required to send an arbitrary message from one process to another, and is a constant overhead for any communication. The parameter $\beta$ is called inverse bandwidth and represents the time needed to send one unit of data.


We assume that an interface restriction of the Green's functions can be stored in one node. Then, all the operations would need to be performed on distributed arrays, and communication between nodes would be needed at each GMRES iteration. However, in 2D, the data  to be transmitted are only traces, so the overhead of transmitting such small amount of information several times would overshadow the complexity count. We gather the compressed Green's functions in a master node and use them as blocks to form the system $\underline{\underline{\mathbf{M}}}$. Moreover, we reuse the compressed blocks to form the matrices used in the preconditioner ($\mathbf{\underline{D}}^{-1}$  and $\mathbf{\underline{R}}$).

We suppose that the squared slowness model $m$ and the sources $\f$ are already on the nodes, and that the local solutions for each source will remain on the nodes to be gathered in a post-processing step. In the context of seismic inversion this assumption is very natural, given that the model updates are performed locally as well.


Within the offline step, we suppose that the model is already known to the nodes (otherwise we would need to communicate the model to the nodes, incurring a cost $\cO(\alpha + \beta N /L)$ for each layer). The factorization of the local problem, the extraction of the Green's functions and their compression are zero-communication processes. The compressed Green's functions are transmitted to the master node, with a maximum cost of $\cO(\alpha +  4 \beta N)$ for each layer. This cost is in general lower because of the compression of Green's function. In summary, the whole communication cost for the offline computations is dominated by $\cO(L(\alpha +  4 \beta N))$. Using the compression scaling for the Green's matrices in section \ref{subsection:compression_scalings} the communication cost can be reduced. However, it remains bounded from below by $\cO(L \alpha + \beta N)$.

For the online computation, we suppose that the sources are already distributed (otherwise we would incur a cost $\cO(\alpha + \beta N /L)$ for each layer). Once the local solves, which are zero-communication processes, are performed, the traces are sent to the master node, which implies a cost of $\cO(L(\alpha + \beta N^{1/2}))$. The inversion of the discrete integral system is performed and the traces of the solution are sent back to the nodes, incurring another $\cO(L(\alpha + \beta N^{1/2}))$ cost. Finally, from the traces, the local solutions are reconstructed  and saved in memory for the final assembly in a post-processing step not accounted for here (otherwise, if the local solutions are sent back to the master node we would incur a $\cO(\alpha + \beta N /L)$ cost per layer). In summary, the communication cost for the online computation is $\cO(\alpha L + \beta L N^{1/2})$.

Finally, we point out that for the range of 2D numerical experiments performed in this paper, the communication cost was negligible with respect to the floating points operations.

\section{Numerical Experiments}

In this section we show numerical experiments that illustrate the behavior of the proposed algorithm. Any high frequency Helmholtz solver based on domain decomposition should ideally have three properties:
\begin{itemize}
	\item the number of iterations should be independent of the number of the sub-domains,
	\item the number of iterations should be independent of the frequency,
	\item the number of iterations should depend weakly on the roughness of the underlying model.
\end{itemize}
We show that our proposed algorithm satisfies the first two properties. The third property is also verified in cases of interest to geophysicists, though our solver is not expected to scale optimally (or even be accurate) in the case of resonant cavities.

We also show some empirical scalings supporting the sub-linear complexity claims for the GMRES iteration, and the sub-linear run time of the online part of the solver.

The code for the experiments was written in Python. We used the Anaconda implementation of Python 2.7.8 and the Numpy 1.8.2 and Scipy 0.14.0 libraries linked to OpenBLAS.  All the experiments were carried out in 4 quad socket servers with AMD Opteron 6276 at 2.3 GHz and 256 Gbytes of RAM, linked with a gigabit Ethernet connection. All the system were preconditioned by two iterations of the Gauss-Seidel preconditioner.

\subsection{Precomputation}
To extract the Green's functions and to compute the local solutions, a pivoted sparse LU factorization was performed at each slab using UMFPACK \cite{Davis:UMFPACK}, and the LU factors were stored in memory. The LU factors for each slab are independent of the others, so they can be stored in different cluster nodes. The local solves are local to each node, enhancing the granularity of the algorithm. Then the Green's function used to form the discrete integral were computed by solving the local problem with columns of the identity in the right-hand side. The computation of each column of the Green's function is independent of the rest and can be done in parallel (shared or distributed memory).


Once the Green's functions were computed, they were compressed in sparse PLR form, following Alg. \ref{alg:PLR_matrix}.  $\underline{\underline{\mathbf{M}}}$ and $\mathbf{\underline{R}}$ were implemented as a block matrices, in which each block is a PLR matrix in sparse form. The compression was accelerated using the randomized SVD \cite{MartinssonRokhlin:A_randomized_algorithm_for_the_decomposition_of_matrices}. The inversion of $\underline{\mathbf{D}}$ was implemented as a block back-substitution with compressed blocks.

We used a simple GMRES algorithm\footnote{Some authors refer to GMRES algorithm as the Generalized conjugate residual algorithm \cite{Bruno_Turc:Electromagnetic_integral_equations_requiring_small_numbers_of_krylov_subspace_iterations}, \cite{Rokhlin:rapid_solution_of_integral_equations_of_classical_potential_theory} } ( Algorithm 6.9 in \cite{Saad:iterative_methods_for_sparse_linear_systems}). Given that the number of iteration remains, in practice, bounded by ten, neither low rank update to the Hessenberg matrix nor re-start were necessary.


\subsection{Smooth Velocity Model}
For the smooth velocity model, we choose a smoothed Marmousi2 model (see Fig. \ref{fig:Marmousi2_smooth}.) Table   \ref{table:GMRES_iteration_smooth_Marmousi}  and Table \ref{table:runtime_iteration_smooth_Marmousi} were generated by timing 200 randomly generated right-hand-sides. Table \ref{table:GMRES_iteration_smooth_Marmousi} shows the average runtime of one GMRES iteration, and Table \ref{table:runtime_iteration_smooth_Marmousi} shows the average runtime of the online part of the solver.
We can observe that for the smooth case the number of iterations is almost independent of the frequency and the number of sub-domains. In addition, the runtimes scales sub-linearly with respect to the number of volume unknowns.

\begin{figure}[H]
	\begin{center}
	 	\includegraphics[trim = 90mm 5mm 70mm 18mm, clip, width = 160mm]{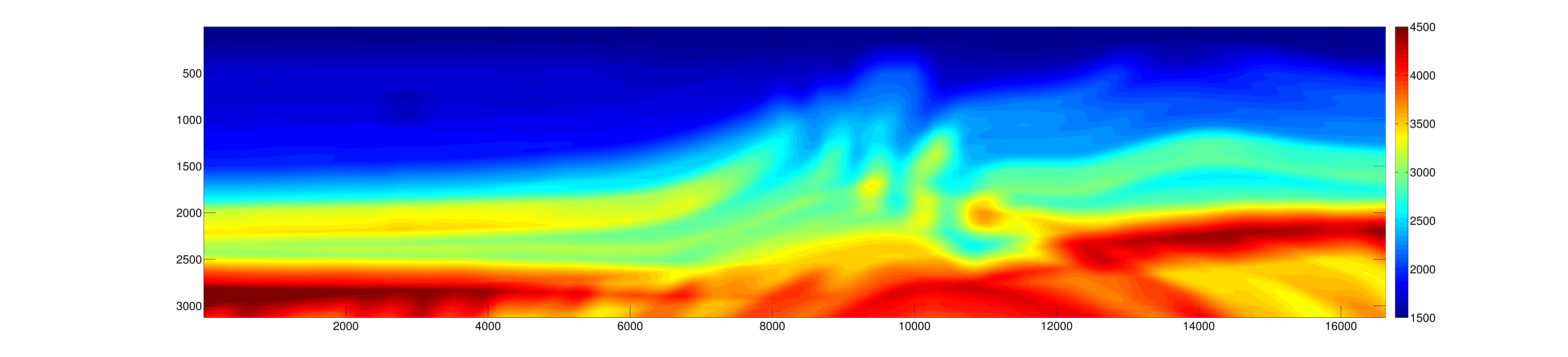}
	 	 \caption{ Smooth version of the Marmousi2 model. The model was smoothed with a box filter of size 375 meters.}
	 	 \label{fig:Marmousi2_smooth}
	 \end{center}
\end{figure}

\begin{table}
    \begin{center}
        \begin{tabular}{|c|c|r|r|r|r|r|}
            \hline
            $N$ & $\omega/2\pi$ [Hz] & $L = 8$ & $L = 16$ & $L= 32$ & $L = 64$ & $L=128$ \\
            \hline
            $195 \times 870$    &  5.57 &  \textbf{(3)} 0.094	&  \textbf{(3)}  0.21   	&    \textbf{(3)} 0.49  &   \textbf{(3-4)} 1.02   	&  \textbf{(3-4)} 2.11  \\
            $396 \times 1720 $  &  7.71 &  \textbf{(3)}   0.14  & \textbf{(3)} 0.32 		&    \textbf{(3)} 0.69 	&   \textbf{(3-4)} 1.55 	&  \textbf{(3-4)} 3.47    \\
            $792 \times 3420 $  &  11.14 &  \textbf{(3)}  0.41	& \textbf{(3)} 0.83   		&    \textbf{(3)} 1.81 	&   \textbf{(3-4)} 3.96		&  \textbf{(3-4)} 9.10     \\
            $1586 \times 6986 $ &  15.86 &  \textbf{(3)}  0.72	& \textbf{(3)} 1.56 		&    \textbf{(3)} 3.19 	&   \textbf{(3-4)} 6.99 	&   - \\
            \hline
        \end{tabular}\caption{Number of GMRES iterations (bold) required to reduce the relative residual to $10^{-7}$, along with average execution time (in seconds) of one GMRES iteration for different $N$ and $L$. The solver is applied to the smooth Marmousi2 model. The frequency is scaled such that $\omega \sim \sqrt{n}$. The matrices are compressed using $\epsilon = 10^{-9}/L$ and $\text{rank}_{\text{max}} \sim \sqrt{n}$. }                                  \label{table:GMRES_iteration_smooth_Marmousi}
    \end{center}
\end{table}

\begin{table}
    \begin{center}
        \begin{tabular}{|c|c|r|r|r|r|r|}
            \hline
            $N$ & $\omega/2\pi$ [Hz] & $L = 8$ & $L = 16$ & $L= 32$ & $L = 64$ & $L=128$ \\
            \hline
            $195 \times 870$    &  5.57  & 0.36	&  0.72	& 1.53	 & 3.15   & 7.05 \\
            $396 \times 1720 $  &  7.71  & 0.79	&  1.26	& 2.35	 & 4.99   & 11.03	 \\
            $792 \times 3420 $  &  11.14 & 2.85	&  3.69	& 6.41	 & 13.11  & 28.87 \\
            $1586 \times 6986 $ &  15.86 & 9.62	&  9.76	& 13.85	 & 25.61  & - \\
            \hline
        \end{tabular}\caption{Average execution time (in seconds) for the online computation, with a GMRES tolerance of $10^{-7}$, for different $N$ and $L$. The solver is applied to the smooth Marmousi2 model. The frequency is scaled such that $\omega \sim \sqrt{n}$. The matrices are compressed using $\epsilon = 10^{-9}/L$ and $\text{rank}_{\text{max}} \sim \sqrt{n}$. }                                  \label{table:runtime_iteration_smooth_Marmousi}
    \end{center}
\end{table}

\subsection{Rough Velocity Model}

In general, iterative solvers are highly sensitive to the roughness of the velocity model.  Sharp transitions generates strong reflections that hinder the efficiency of iterative methods, increasing the number of iterations. Moreover, for large $\omega$, the interaction of high frequency waves with short wavelength structures such as discontinuities, increases the reflections, further deteriorating the convergence rate.

The performance of the method proposed in this paper deteriorates only marginally as a function of the frequency and number of subdomains. In the experiments performed in this section, the layered partitioning was performed along the direction of higher aspect ratio. For geophysical models the discontinuities in the model tend to be orthogonal to the interfaces of the layered decomposition. The decomposition can be performed in the other direction resulting in an increment on the number of iterations for convergence with a marginal deterioration as the frequency and number of subdomains increase.


\begin{figure}[h]
\begin{center}
 \includegraphics[trim = 90mm 5mm 70mm 5mm, clip, width = 160mm]{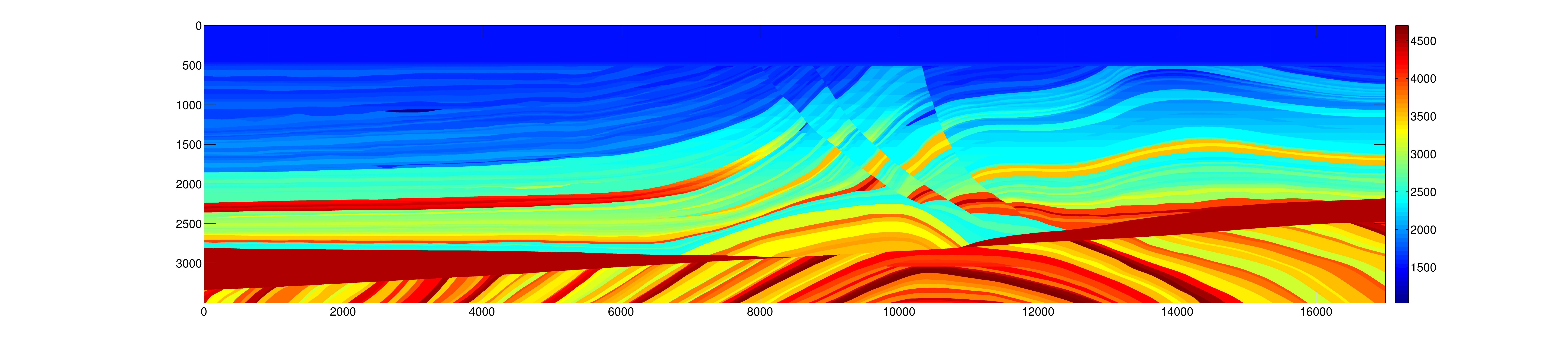}
  \caption{Geophysical benchmark model Marmousi2.}
  \label{fig:Marmousi_2}
 \end{center}
\end{figure}


We use the Marmousi2 model  \cite{Marmousi_2}, and another geophysical community benchmark, the BP 2004 model \cite{BP_model}, depicted in Fig. \ref{fig:Marmousi_2} and Fig. \ref{fig:BP}  respectively.

Tables \ref{table:GMRES_iteration_Marmousi}, \ref{table:runtime_iteration_Marmousi}, \ref{table:GMRES_iteration_BP}, and \ref{table:runtime_iteration_BP}  were generated by running 200 randomly generated right hand sides inside the domain. The number of points for the perfectly matched layers is increased linearly with $n$. From Table \ref{table:GMRES_iteration_smooth_Marmousi} and Table \ref{table:GMRES_iteration_BP} we can observe that, in general, the number of iteration to convergence grows slowly. We obtain slightly worse convergence rates if the number of points of the PML increases slowly or if it remains constant. This observation was already made in  \cite{CStolk_rapidily_converging_domain_decomposition} and \cite{Poulson_Engquist:a_parallel_sweeping_preconditioner_for_heteregeneous_3d_helmholtz}. However, the asymptotic complexity behavior remains identical; only the constant changes.

The runtime for one GMRES iteration exhibits a slightly super-linear growth when $L$, the number of sub-domains, is increased to be a large fraction of $n$. This is explained by the different compression regimes for the interface operators and by the different tolerances for the compression of the blocks. When the interfaces are very close to each other the interactions between the source depths and the target depths increases, producing higher ranks on the diagonal blocks of the remote interactions.

Tables \ref{table:runtime_iteration_Marmousi} and  \ref{table:runtime_iteration_BP} show that the average runtime of each solve scales sublinearly with respect to the volume unknowns.

\begin{figure}[H]
\begin{center}
 \includegraphics[trim = 94mm 0mm 72mm 5mm, clip, width = 160mm]{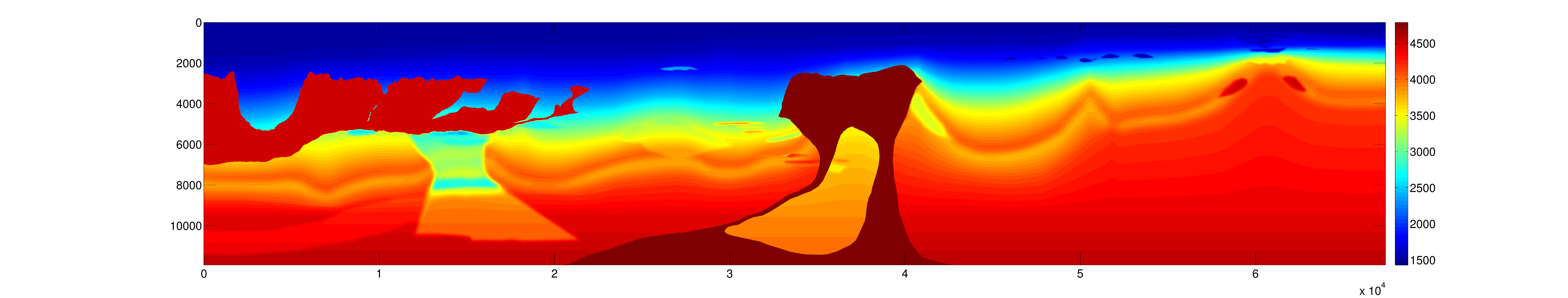}
  \caption{Geophysical benchmark model BP 2004.}
  \label{fig:BP}
 \end{center}
\end{figure}

Table \ref{table:GMRES_iteration_BP_constant_nppwl} and Table \ref{table:GMRES_iteration_marmousi_constant_nppwl} illustrate the maximum and minimum number of iterations and the average runtime of one GMRES iteration in the high frequency regime, i.e. $\omega \sim n$. We can observe that the number of iterations are slightly higher but with a slow growth. Moreover, we can observe the slightly sub-linear behavior of the runtimes, which we called the pre-asymptotic regime in Section \ref{section:PLR}.

Finally, Fig. \ref{fig:complexities_sqrt} summarizes the scalings for the different stages of the computation for fixed $L$. For a small number of unknowns the cost is dominated by the solve of the integral system -- a case in which the online part seems to have a sub-linear complexity. However, after this pre-asymptotic regime, the LU solve dominates the complexity, so it is lower-bounded by the $\cO(N \log(N)/L)$ runtime. The GMRES timing is still in the pre-asymptotic regime, but the complexity would probably deteriorate to $\cO(N^{3/4})$ for frequencies high enough. The complexity of the online part is non-optimal given that we only used a fixed number or layers and processors. (We do not advocate choosing such small $L$.)

\begin{table}
    \begin{center}
        \begin{tabular}{|c|c|r|r|r|r|r|}
            \hline
            $N$ & $\omega/2\pi$ [Hz] & $L = 8$ & $L = 16$ & $L= 32$ & $L = 64$ & $L=128$ \\
            \hline
            $195 \times 870$    &  5.57 &  \textbf{(4-5)} 0.08 	&  \textbf{(5)}  0.18   	&    \textbf{(5)}   0.41    &   \textbf{(5-6)} 0.87   	&  \textbf{(5-6)} 1.83  \\
            $396 \times 1720 $  &  7.71 &  \textbf{(5)}   0.17	& \textbf{(5-6)} 0.41 		&    \textbf{(5-6)} 0.88 	&   \textbf{(5-6)} 2.06 	&  \textbf{(5-7)} 4.57    \\
            $792 \times 3420 $  &  11.14 &  \textbf{(5)}  0.39	& \textbf{(5-6)} 0.85   	&    \textbf{(5-6)} 1.82 	&   \textbf{(5-6)} 4.06		&  \textbf{(6-7)} 9.51     \\
            $1586 \times 6986 $ &  15.86 &  \textbf{(5)}  0.94	& \textbf{(5-6)} 1.89 		&    \textbf{(5-6)} 4.20 	&   \textbf{(6)}   9.41 	&  \textbf{(6-7)} 21.10 \\
            \hline
        \end{tabular}\caption{Number of GMRES iterations (bold) required to reduce the relative residual to $10^{-7}$, along with average execution time (in seconds) of one GMRES iteration for different $N$ and $L$. The solver is applied to the Marmousi2 model. The frequency is scaled such that $\omega \sim \sqrt{n}$. The matrices are compressed using $\epsilon = 10^{-9}/L$ and $\text{rank}_{\text{max}} \sim \sqrt{n}$. }                                  \label{table:GMRES_iteration_Marmousi}
    \end{center}
\end{table}

\begin{table}
    \begin{center}
        \begin{tabular}{|c|c|r|r|r|r|r|}
            \hline
            $N$ & $\omega/2\pi$ [Hz] & $L = 8$ & $L = 16$ & $L= 32$ & $L = 64$ & $L=128$ \\
            \hline
            $195 \times 870$    &  5.57  & 0.47    	& 0.99    	& 2.15   &  4.53     & 10.94   \\
            $396 \times 1720 $  &  7.71  & 1.30		& 2.36    	& 5.04   &  12.66    & 29.26   \\
            $792 \times 3420 $  &  11.14 & 3.82		& 5.62 		& 11.45  &	25.78	 & 64.01   \\
            $1586 \times 6986 $ &  15.86 & 13.68 	& 16.08	    & 28.78  &	62.45	 & 145.98  \\
            \hline
        \end{tabular}\caption{Average execution time (in seconds) for the online computation, with a tolerance on the GMRES of $10^{-7}$, for different $N$ and $L$. The solver is applied to the Marmousi2 model. The frequency is scaled such that $\omega \sim \sqrt{n}$. The matrices are compressed using $\epsilon = 10^{-9}/L$ and $\text{rank}_{\text{max}} \sim \sqrt{n}$. }                                  \label{table:runtime_iteration_Marmousi}
    \end{center}
\end{table}

\begin{table}
    \begin{center}
        \begin{tabular}{|c|c|r|r|r|r|r|}
            \hline
            $N$ & $\omega/2\pi$ & $L = 8$ & $L = 16$ & $L= 32$ & $L = 64$ & $L=128$ \\
            \hline
            $136 \times 354$    &  1.18  & \textbf{(5-6)} 0.07 &	\textbf{(6)}    0.18 &	\textbf{(7)}    0.38 &	\textbf{(7)}    0.80 &	\textbf{(7-8)}  1.58	 \\
            $269 \times 705 $   &  1.78  & \textbf{(6)}   0.10 & 	\textbf{(6)}    0.22 &  \textbf{(6-7)}	0.50 &  \textbf{(7-8)}	1.07 &  \textbf{(7-8)}	2.22	 \\
            $540 \times 1411 $  &  2.50  & \textbf{(6)}	  0.22 & 	\textbf{(6-7)}	0.52 &	\textbf{(6-7)}	1.22 &  \textbf{(7)}	2.80 &  \textbf{(8-9)}	5.97	 \\
            $1081 \times 2823 $ &  3.56  & \textbf{(6-7)} 0.38 &	\textbf{(6-7)}  0.87 &	\textbf{(7-8)}  1.93 &	\textbf{(7-8)}  4.33 &	\textbf{(8-9)} 10.07    \\
            \hline
        \end{tabular}\caption{Number of GMRES iterations (bold) required to reduce the relative residual to $10^{-7}$, along with average execution time (in seconds) of one GMRES iteration for different $N$ and $L$. The solver is applied to the BP 2004 model.  The frequency is scaled such that $\omega \sim \sqrt{n}$. The matrices are compressed using $\epsilon = 10^{-9}/L$ and $\text{rank}_{\text{max}} \sim \sqrt{n}$. }                                  \label{table:GMRES_iteration_BP}
    \end{center}
\end{table}

\begin{figure}[H]
\begin{center}
 \includegraphics[trim = 0mm 0mm 0mm 0mm, clip, width = 100mm]{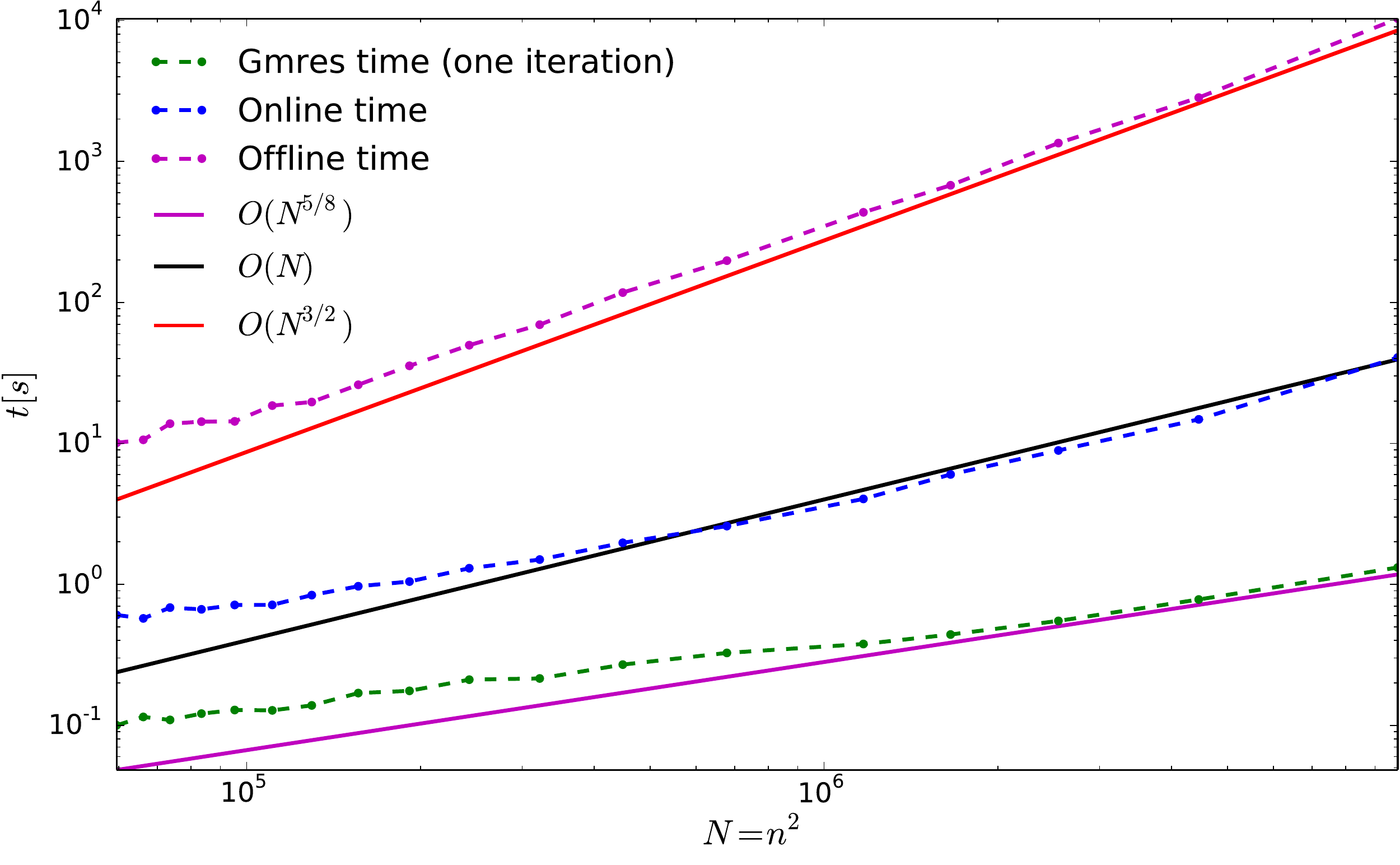}
  \caption{Run-time with their empirical complexities, for the Marmousi2 model with $L=3$ and $\omega \sim \sqrt{n}$ and $\text{max}_{\text{rank}} \sim \sqrt{n} $.}
  \label{fig:complexities_sqrt}
 \end{center}
\end{figure}

\begin{figure}[H]
\begin{center}
 \includegraphics[trim = 0mm 0mm 0mm 0mm, clip, width = 100mm]{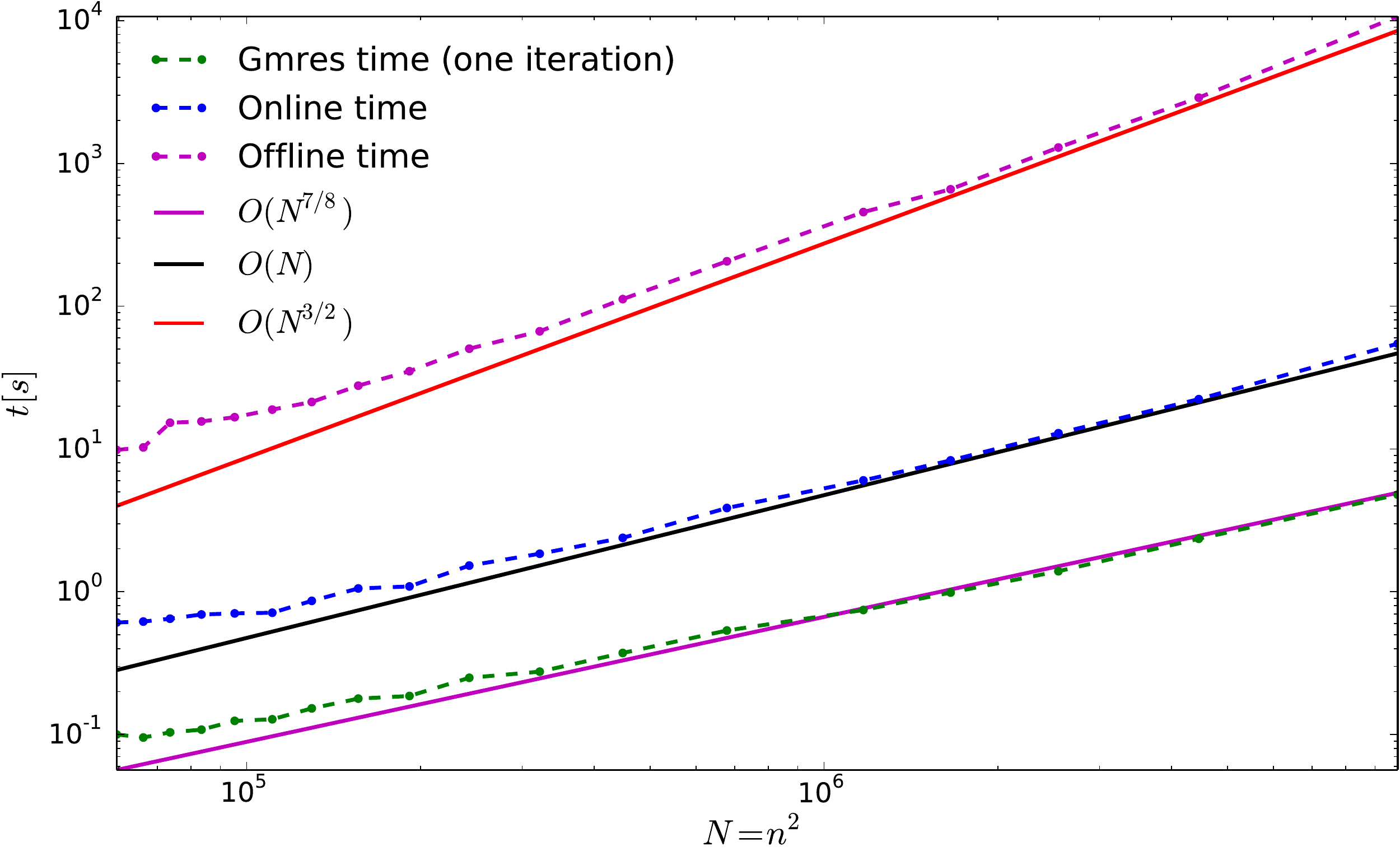}
  \caption{Run-times and empirical complexities, for the Marmousi2 model with $L=3$ and $\omega \sim n$  and $\text{max}_{\text{rank}} \sim \sqrt{n} $.}
  \label{fig:complexities_linear}
 \end{center}
\end{figure}

\begin{table}
    \begin{center}
        \begin{tabular}{|c|c|r|r|r|r|r|}
            \hline
            $N$ & $\omega/2\pi$  & $L = 8$ & $L = 16$ & $L= 32$ & $L = 64$ & $L=128$ \\
            \hline
            $136 \times 354$    &  1.18 & 0.45 & 1.09 & 2.69  & 5.67  & 12.44 \\
            $269 \times 705 $   &  1.78 & 0.70 & 1.44 & 3.45  & 7.86  & 17.72 \\
            $540 \times 1411 $  &  2.50 & 1.85 & 3.74 & 8.83  & 20.09 & 48.86 \\
            $1081 \times 2823 $ &  3.56 & 4.93 & 7.87 & 15.91 & 35.69 & 83.63 \\
            \hline
        \end{tabular}\caption{Average execution time (in seconds) for the online computation, with a tolerance on the GMRES of $10^{-7}$, for different $N$ and $L$. The solver is applied to the BP 2004 model. The frequency is scaled such that $\omega \sim \sqrt{n}$. The matrices are compressed using $\epsilon = 10^{-9}/L$ and $\text{rank}_{\text{max}} \sim \sqrt{n}$. }                                  \label{table:runtime_iteration_BP}
    \end{center}
\end{table}

\begin{table}
    \begin{center}
        \begin{tabular}{|c|c|r|r|r|r|r|}
            \hline
            $N$ & $\omega/2\pi $ [Hz] & $L = 8$ & $L = 16$ & $L= 32$ & $L = 64$ & $L=128$ \\
            \hline
            $195 \times 870$    &   7.2 & \textbf{(4-5)}  0.08	& \textbf{(5)}     0.19  &  \textbf{(5-6)}	    0.40   	& \textbf{(5-6)} 0.83  & \textbf{(6)}	1.74 \\
            $396 \times 1720 $  &   15 & \textbf{(5)}    0.23	& \textbf{(5-6)}   0.49  &  \textbf{(5-6}		0.98   	& \textbf{(5-6)} 2.17  & \textbf{(6-7)}	5.13	 \\
            $792 \times 3420 $  &   30 & \textbf{(5-6)}  0.69	& \textbf{(5-6)}   1.54  &  \textbf{(6-7)}	    2.48    & \textbf{(6-7)} 4.93  & \textbf{(7-8)}	10.64 \\
            $1586 \times 6986 $ &   60 & \textbf{(5-6)}  2.29	& \textbf{(5-6)}   5.04  &  \textbf{(6-7)}  	8.47    & \textbf{(6-8)} 14.94  &    - \\
            \hline
        \end{tabular}\caption{Number of GMRES iterations (bold) required to reduce the relative residual to $10^{-7}$, along with average execution time (in seconds) of one GMRES iteration for different $N$ and $L$. The solver is applied to the Marmousi2 model. The frequency is scaled such that $\omega \sim n$. The matrices are compressed using $\epsilon = 10^{-9}/L$ and $\text{rank}_{\text{max}} \sim \sqrt{n}$. }                                  \label{table:GMRES_iteration_marmousi_constant_nppwl}
    \end{center}
\end{table}

\begin{table}
    \begin{center}
        \begin{tabular}{|c|c|r|r|r|r|r|}
            \hline
            $N$ & $\omega/2\pi $ [Hz] & $L = 8$ & $L = 16$ & $L= 32$ & $L = 64$ & $L=128$ \\
            \hline
            $136 \times 354$    &      1.4 & \textbf{(6)}    0.073 & \textbf{(6-7)}  0.18  & \textbf{(7-8)} 0.37 & \textbf{(7-8)}  0.84  & \textbf{(8-9)}   1.52    \\
            $269 \times 705 $   &      2.7 & \textbf{(6-7)}  0.10  & \textbf{(6-7)}  0.23  & \textbf{(7)}   0.50 & \textbf{(8-9)}  1.10  & \textbf{(8-9)}   2.14    \\
            $540 \times 1411 $  &      5.5 & \textbf{(7)}    0.32  & \textbf{(7-8)}  0.64  & \textbf{(8-9)} 1.30 & \textbf{(9)}    3.06  & \textbf{(10-12)} 6.22    \\
            $1081 \times 2823 $ &     11.2 & \textbf{(7)}    0.87  &	\textbf{(7-8)}  1.46  &	\textbf{(8-9)} 2.78 & \textbf{(9-10)} 5.75  & \textbf{(10-12)} 12.65   \\
            \hline
        \end{tabular}\caption{Number of GMRES iterations (bold) required to reduce the relative residual to $10^{-7}$, along with average execution time (in seconds) of one GMRES iteration for different $N$ and $L$. The solver is applied to the BP 2004 model. The frequency is scaled such that $\omega \sim n$. The matrices are compressed using $\epsilon = 10^{-9}/L$ and $\text{rank}_{\text{max}} \sim \sqrt{n}$. }                                  \label{table:GMRES_iteration_BP_constant_nppwl}
    \end{center}
\end{table}

\subsection{Cavities}

As stated in the introduction, the complexity of the algorithm degrades when the underlying model contains sharp interfaces and resonant cavities. In general, the increase in complexity is mainly due to the increased number of iterations needed for convergence. This fact can be easily understood by looking at the spectrum of the preconditioned integral system.

The model in Fig. \ref{fig:resonator} (left) is used for the numerical experiments. The shape is kept constant, with an adjustable $c_{\text{red}}$ and with a smooth background speed $c_{\text{blue}}(x,y) = 1+0.1x +0.1y$ for $(x,y)\in [0,1]^2$, such that the problem can not be reduced to an integral equations posed on the interfaces.

\begin{figure}[H]
    \begin{center}
        \includegraphics[trim = 27mm 25mm 10mm 15mm, height = 7cm, clip]{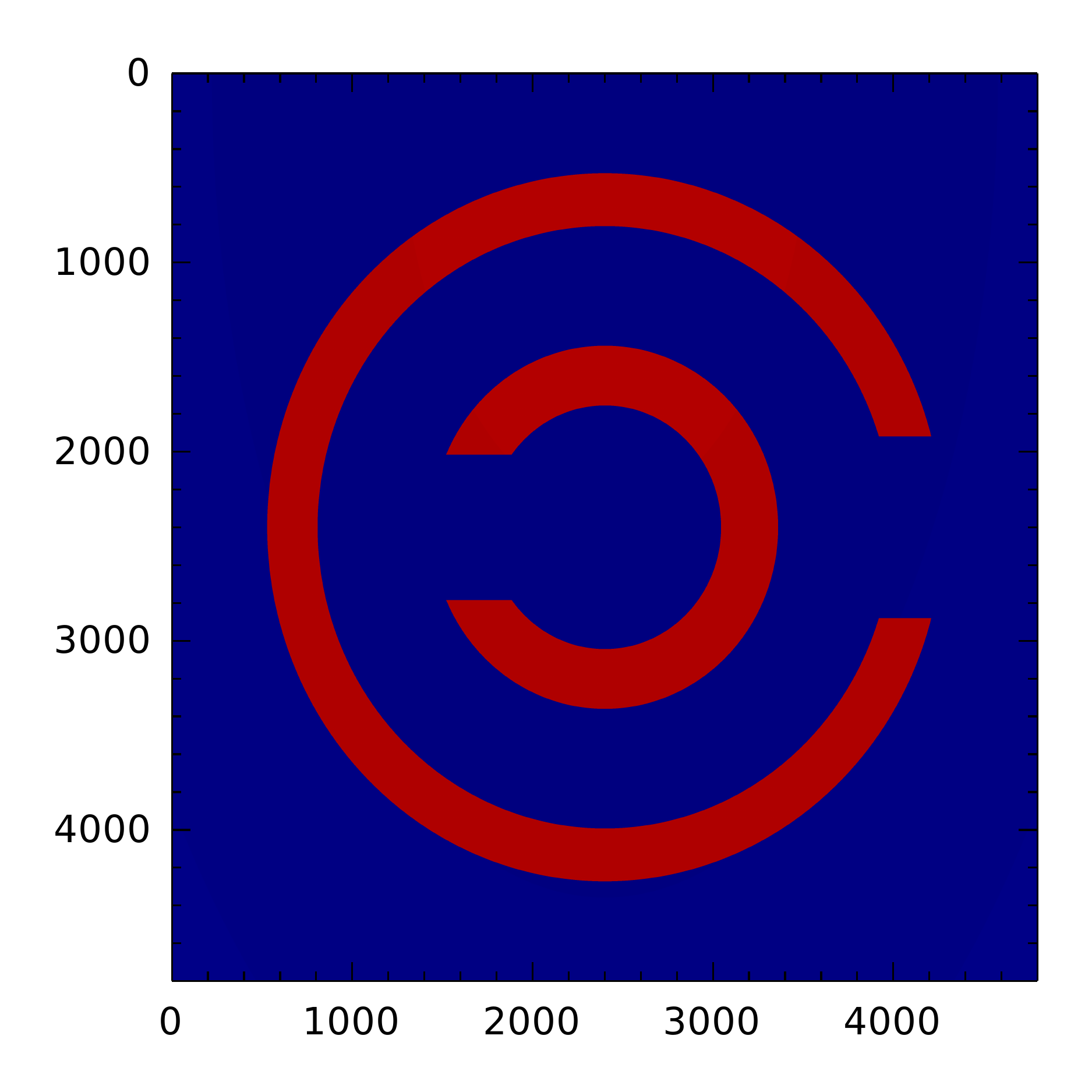} \hspace{0.5cm}
        \includegraphics[trim = 45mm 37mm 35mm 26mm, height = 7cm, clip]{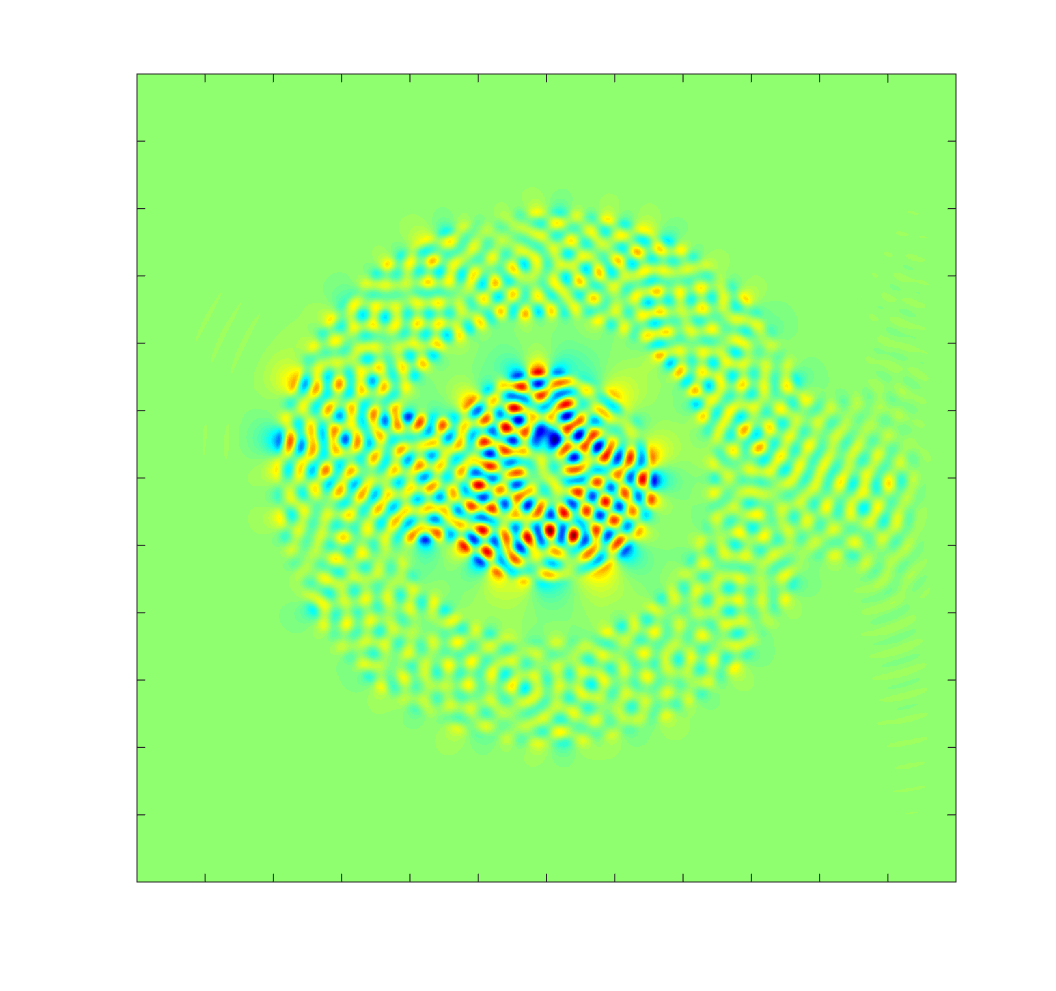}
    \end{center}
    \caption{Left: Resonant wave-guide inspired in the Comedy Central logo; right: real part of a typical solution.} \label{fig:resonator}
\end{figure}

Fig. \ref{fig:eigen_cavity} shows the eigenvalues of the preconditioned polarized system, using $L=4$ layers, for different frequencies. We can observe that the eigenvalues are spread in the unit disk centered at one as stated in Section \ref{section:gmres}. Moreover, as the frequency increases, the eigenvalues tend to be less clustered around 1, in particular, we can observe some eigenvalues close to zero. In such circumstances, GMRES is known to experience slow convergence. Finally, we point out that the number of eigenvalues far from 1 grows as $\omega$, which can be expected from the number of resonant modes in function of the frequency in 2D. We point out the radically different behavior from the spectra shown in Figs. \ref{fig:eigen_marmousi} and \ref{fig:eigen_homogeneous}, in which all the eigenvalues are tightly clustered around 1, far from zero.

\begin{figure}[H]
    \begin{center}
        \includegraphics[trim = 19mm 5mm 20mm 5mm, height = 5.1cm, clip]{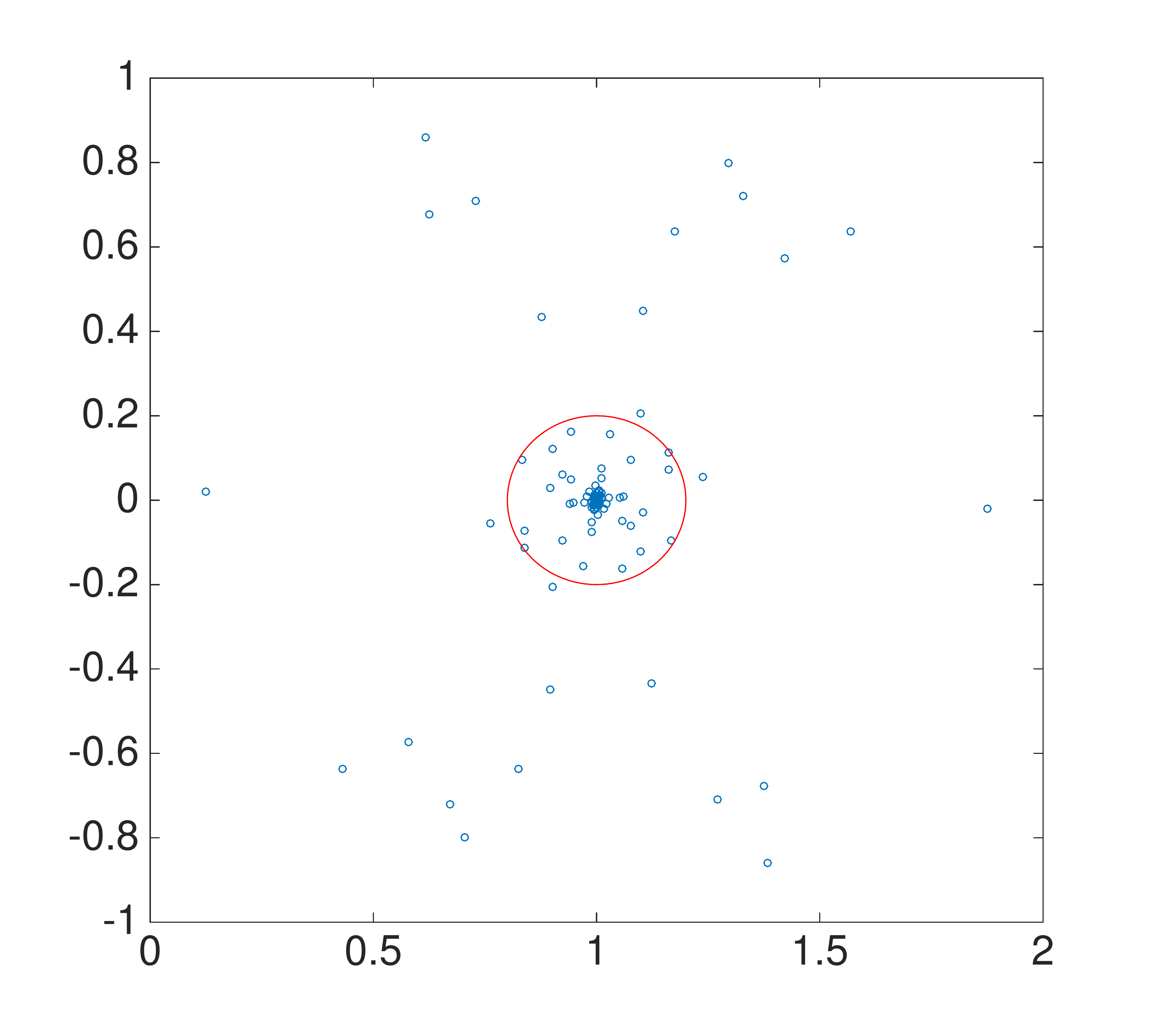}
        \includegraphics[trim = 19mm 5mm 20mm 5mm, height = 5.1cm, clip]{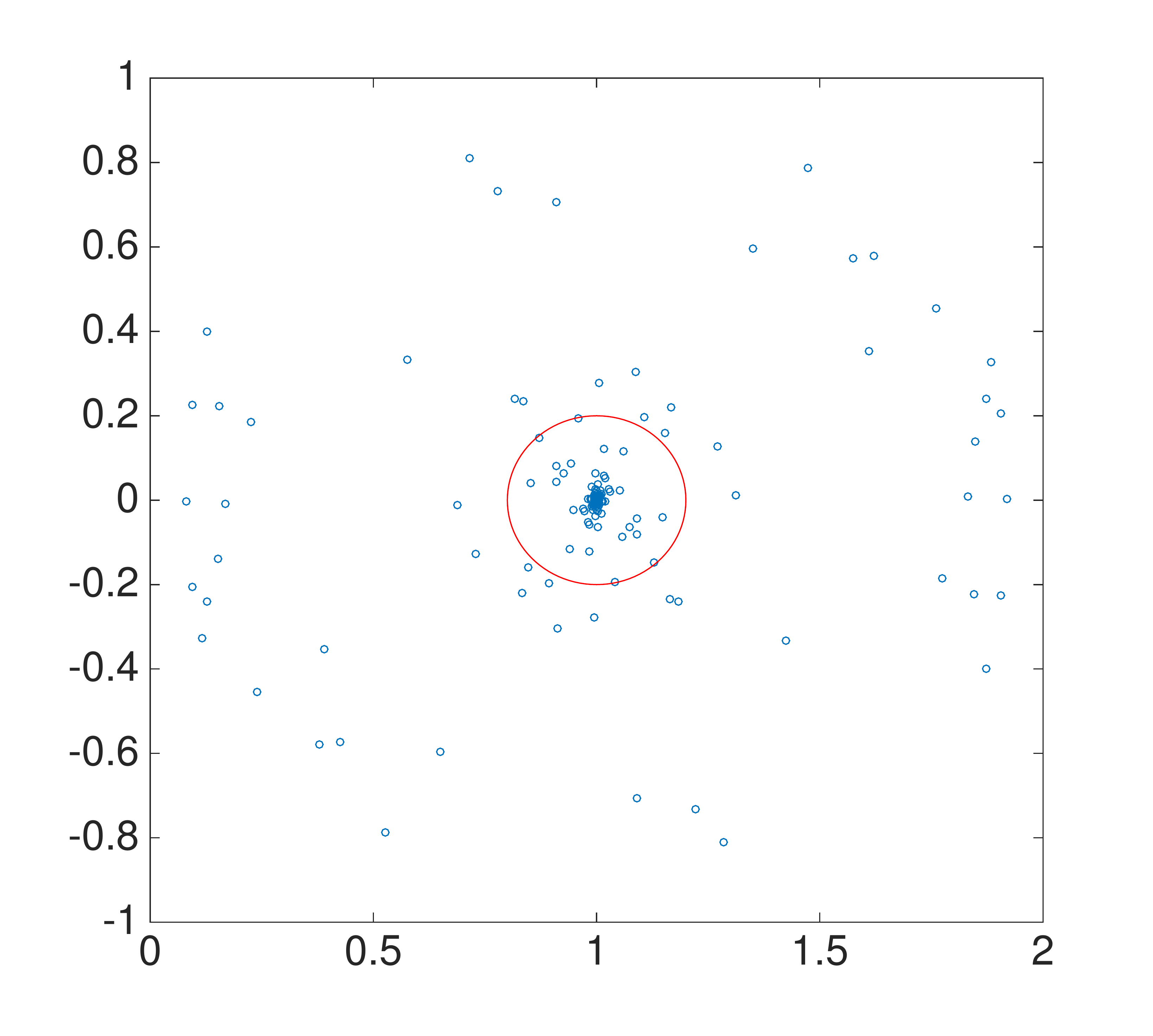}
        \includegraphics[trim = 19mm 5mm 20mm 5mm, height = 5.1cm, clip]{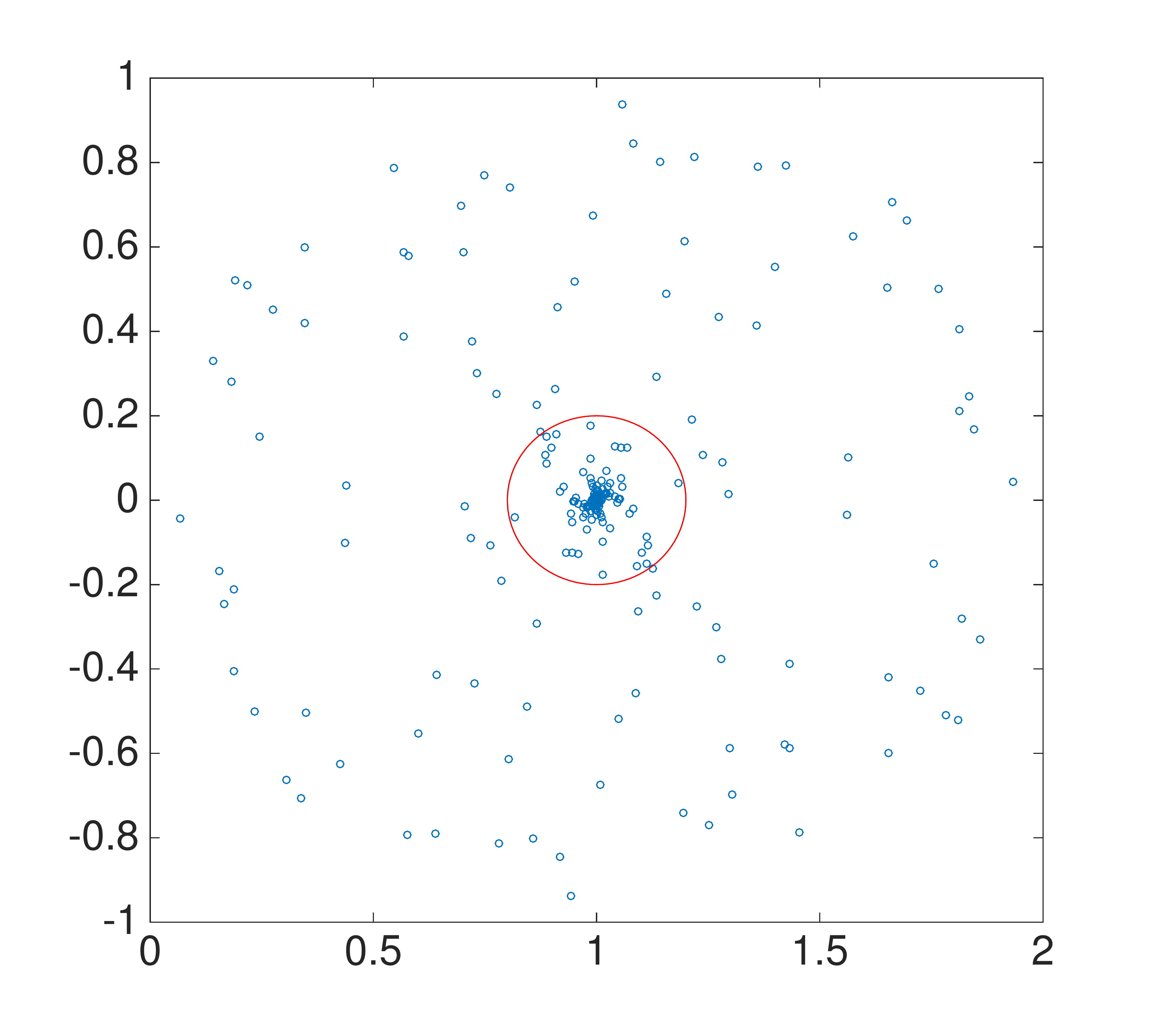}
    \end{center}
    \caption{Eigenvalues in the complex plane of $(\mathbf{\underline{D}}\,)^{-1} \; \mathbf{\underline{P}} \; \underline{\underline{\mathbf{M}}}$, for the resonant cavity in Fig. \ref{fig:resonator} with $ c_{\text{red}}= 5$, for $\omega = 8$ (left), $\omega = 16$ (center); and $\omega = 32$ (right). There are $26$ (left), $58$ (center); and $108$ (right) eigenvalues outside the red circle of radius $0.2$ centered at $1$.} \label{fig:eigen_cavity}
\end{figure}

\section{Acknowledgments}
The authors thank Thomas Gillis and Adrien Scheuer for their help with the Python code used to generate the complexity plots; Russell Hewett, Paul Childs and Chris Stolk for insightful discussions; Alex Townsend for his valuable remarks on earlier versions of the manuscript; Cris Cecka for insightful discussions concerning the multilevel compression algorithms; and the anonymous reviewers for their helpful insight.

The authors thank Total SA for supporting this research. LD is also supported by AFOSR, ONR, NSF, and the Alfred P. Sloan Foundation.

\appendix
\section{Discretization} \label{appendix:quadratures}

\subsection*{Computations for Eq. \ref{eq:discrete_GRF} in 1D}

The stencils for the derivatives of $G$ and $u$ are given by the discrete Green's representation formula. We present an example in 1D, which is easily generalized to higher dimension.
Let $u = \{ u_i\}_{i = 0}^{n+1}$ and  $v= \{ v_i\}_{i = 0}^{n+1}$. Define $\Omega = \{ 1, ... n\}$ with the discrete inner product
\[
	\langle u,v \rangle_{\Omega} = \sum_{i = 1}^n u_i v_i.
\]
Let $\triangle^h u$ be the three-point stencil approximation of the second derivative. I.e.
\[
	\left(  \triangle^h u \right)_i =  u_{i-1} - 2 u_i + u_{i+1}.
\]
We can use summation by parts to obtain the expression
\begin{equation}
	\langle \triangle^h u, v \rangle_{\Omega} = \langle  u, \triangle^h v \rangle_{\Omega} - v_0( u_1 - u_0) + u_0( v_1 - v_0) + v_{n+1}(u_{n+1} - u_n) - u_{n+1}(v_{n+1} - v_n).
\end{equation}
The formula above provides the differentiation rule to compute the traces and the derivatives of $u$ and $G$ at the interfaces. For example, let $G^{k}_i$ be the solution to 1D discrete Helmholtz operator defined by $ H G^{k} =  (- \triangle^h - m\omega^2) G^{k} = \delta^{k}_i$, for $j \in [1,...,n]$. Then
\begin{equation}
	u_k = \langle H G^{k}, u \rangle_{\Omega} = \langle  G^{k}, H u \rangle_{\Omega} - G^{k}_0 (u_1 - u_0) + u_0 (G^{k}_1 - G^{k}_0)+ G^{k}_{n+1} (u_{n+1} - u_n) - u_{n+1}(G^{k}_{n+1} - G^{k}_n).
\end{equation}
To simplify the notations we define
\begin{equation}
	\partial^{+} G_0 = G_1 - G_0, \qquad \partial^{-} G_{n+1} = G_{n+1} - G_n,
\end{equation}
which are upwind derivatives with respect to $\Omega$, i.e. they look for information inside $\Omega$.
Let us consider a partition of $\Omega = \Omega^1 \cup \Omega^2$, where $\Omega^1 = \{ 1, n^1\}$ and $\Omega^2 = \{n^1+1, n^2\}$. We can define local inner products between $u$ and $v$ analogously,
\[ \langle u,v \rangle_{\Omega^1} = \sum_{i = 1}^{n^1} u_i v_i, \qquad  \langle u,v \rangle_{\Omega^2} = \sum_{i = n^1+1}^{n^2} u_i v_i, \]
in such a way that
\[ \langle u,v \rangle_{\Omega^1} +  \langle u,v \rangle_{\Omega^2}  =  \langle u,v \rangle_{\Omega}. \]
We can use summation by parts in each sub-domain, obtaining
\begin{align}
	\langle H G, u \rangle_{\Omega^1} &= \langle  G, H u \rangle_{\Omega^1} - G_0 \partial^{+}u_0 + u_0 \partial^{+} G_0 + G_{n^1+1} \partial^{-}u_{n^1+1} - u_{n^1+1}\partial^{-} G_{n^1+1},  \label{eq:quadrature} \\
	\langle H G, u \rangle_{\Omega^2} &= \langle  G, H u \rangle_{\Omega^2} - G_{n^1} \partial^{+}u_{n^1} + u_{n^1} \partial^{+} G_{n^1} + G_{n^2+1} \partial^{-}u_{n^2+1} - u_{n^2+1}\partial^{-} G_{n^2+1},
\end{align}
which are the 1D version of Eq. \ref{eq:discrete_GRF}.

\subsection*{Computations for Eq. \ref{eq:discrete_GRF} in 2D}

The 2D case is more complex given that we have to consider the PML in the $x$ direction when integrating by parts. To simplify the proofs, we need to introduce the symmetric formulation of the PML's for the Helmholtz equation. We show in this section  that the symmetric and unsymmetric formulations lead to the same Green's representation formula. In addition, we present Eq. \ref{eq:equivalence_symmstric_unsymmetric}, which links the Green's functions of both formulations. In Appendix  \ref{appendix:green_representation} we use the symmetric formulation in all of the proofs; however, owing to Eq. \ref{eq:equivalence_symmstric_unsymmetric} the proofs are valid for the unsymmetric formulation as well.

Following \cite{EngquistYing:Sweeping_H} we define the symmetrized Helmholtz equation with PML's given by
\begin{equation}\label{eq:Hemholtz_symmetric}
- \left ( \partial_x \frac{\alpha_x(\x)}{\alpha_z(\x)} \partial_x  +  \partial_z \frac{\alpha_z(\x)}{\alpha_x(\x)} \partial_z \right ) u - \frac{m \omega^2}{\alpha_x(\x) \alpha_z(\x)} u = \frac{f} {\alpha_x(\x) \alpha_z(\x)},
\end{equation}
which is obtained by dividing Eq. \ref{eq:Helmholtz_pml} by $\alpha_x(\x) \alpha_z(\x)$ and using the fact that $\alpha_x$ is independent of $z$, and $\alpha_z$ is independent of $x$.
We can use the same discretization used in Eq. \ref{eq:pml_partial_z} to obtain the system
\begin{align} \notag
	\left( \mathbf{H^s} \u \right )_{i,j} =	& -\frac{1}{h^2} 	\left ( \frac{\alpha_x(\x_{i+1/2,j})}{ \alpha_z(\x_{i+1/2,j})} (\u_{i+1,j} - \u_{i,j}) - \frac{\alpha_x(\x_{i-1/2,j})}{ \alpha_z(\x_{i-1/2,j}) } ( \u_{i,j} - \u_{i-1,j} ) \right)  \\
											& - \frac{ 1}{h^2}  \left ( \frac{\alpha_z(\x_{i,j+1/2})}{ \alpha_x(\x_{i,j+1/2})} (\u_{i,j+1} - \u_{i,j}) - \frac{\alpha_z(\x_{i,j-1/2})}{ \alpha_x(\x_{i,j-1/2}) }( \u_{i,j} - \u_{i,j-1} ) \right) \notag \\
											& -  \frac{m(\x_{i,j}) \omega^2 }{\alpha_x(\x_{i,j})\alpha_z(\x_{i,j})} \u_{i,j}  = \mathbf{f}_{i,j}, \label{eq:Helmholtz_pml_symmetric}
\end{align}
in which we used the fact that the support of $\mathbf{f}$ is included in the physical domain, where $\alpha_x$ and $\alpha_z$ are equal to $1$.  Moreover, from Eq. \ref{eq:Hemholtz_symmetric} and Eq. \ref{eq:def_local_Green_function}, it is easy to prove that the Green's function associated to $\mathbf{H^s}$ satisfies
\begin{equation} \label{eq:equivalence_symmstric_unsymmetric}
	\mathbf{G}(\x,\y) \alpha_x(\y) \alpha_z(\y) = \mathbf{G^s}(\x,\y).
\end{equation}
Given that $\mathbf{f}$ is supported inside the physical domain we have that $\alpha_x(\y) \alpha_z(\y) \mathbf{f}(\y) = \mathbf{f}(\y)$, then applying $\mathbf{G}$ and $\mathbf{G^s}$ to $\mathbf{f}$ yield the same answer, i.e., $\u$ and $\u^s$ (the solution to the symmetrized system) are identical.

To deduce Eq. \ref{eq:discrete_GRF}, we follow the same computation as in the 1D case. We use the discrete $\ell^2$ inner product, and we integrate by parts to obtain
\begin{align} \notag
\langle \mathbf{H^s}\u, \v \rangle  = 	& \sum_{i,j= 1}^{n_x,n_z} \left( \mathbf{H^s}\u \right)_{i,j} \v_{i,j}, \\ \notag
									=	& \,\, \langle \u, \mathbf{H^s} \v \rangle \\ \notag
										& - \sum_{j=1}^{n_z} \bigg[ \frac{\alpha_x(\x_{1/2,j})}{\alpha_z(\x_{1/2,j})} \left (  \u_{0,j} \partial^+_x \v_{0,j} - \v_{0,j} \partial^+_x \u_{0,j} \right)  \\ \notag
										& +  \frac{\alpha_x(\x_{n_x + 1/2,j})}{\alpha_z(\x_{n_x + 1/2,j})} \left (  \u_{n_x +1,j} \partial^-_x \v_{n_x +1,j} - \v_{n_x +1,j} \partial^-_x \u_{n_x +1,j} \right)        \bigg]\\ \notag
										& - \sum_{i=1}^{n_x} \bigg[ \frac{\alpha_z(\x_{i,1/2})}{\alpha_x(\x_{i,1/2})} \left (  \u_{i,0} \partial^+_z \v_{i,0} - \v_{i,0} \partial^+_z \u_{i,0} \right)  \\ \notag
										& +  \frac{\alpha_z(\x_{i,n_z + 1/2})}{\alpha_x(\x_{i,n_z + 1/2})} \left (  \u_{i,n_z +1} \partial^-_z \v_{i,n_z +1} - \v_{i,n_z +1} \partial^-_z \u_{i,n_z +1} \right)        \bigg].
\end{align}
This is the general formula for the GRF. Moreover, given the nature of the layered partition, we can further simplify this expression. In each subdomain we have
\begin{align} \notag
\langle \mathbf{H^s}\u, \v \rangle_{\Omega^{\ell}}  = 	&\sum_{i = -n_{\text{pml}}+1}^{n_x+ n_{\text{pml}}}  \sum_{j= 1}^{n_z} \left( \mathbf{H^s}\u \right)_{i,j} \v_{i,j}, \\ \notag
									=	& \,\, \langle \u, \mathbf{H^s} \v \rangle _{\Omega^{\ell}}  \\ \notag
										& - \sum_{j=1}^{n_z} \bigg[ \frac{\alpha_x(\x_{1/2-n_{\text{pml}}),j})}{\alpha_z(\x_{1/2-n_{\text{pml}}),j})} \left (  \u_{-n_{\text{pml}},j} \partial^+_x \v_{-n_{\text{pml}},j} - \v_{-n_{\text{pml}},j} \partial^+_x \u_{-n_{\text{pml}},j} \right)  \\ \notag
										& +  \frac{\alpha_x(\x_{n_x +n_{\text{pml}}+ 1/2,j})}{\alpha_z(\x_{n_x + n_{\text{pml}} + 1/2,j})} \left (  \u_{n_x+n_{\text{pml}} +1,j} \partial^-_x \v_{n_x+n_{\text{pml}} +1,j} - \v_{n_x +n_{\text{pml}}+1,j} \partial^-_x \u_{n_x+n_{\text{pml}} +1,j} \right)        \bigg]\\ \notag
										& - \sum_{i=-n_{\text{npml}}+1}^{n_x+n_{\text{npml}}} \bigg[ \frac{\alpha_z(\x_{i,1/2})}{\alpha_x(\x_{i,1/2})} \left (  \u_{i,0} \partial^+_z \v_{i,0} - \v_{i,0} \partial^+_z \u_{i,0} \right)  \\ \notag
										& +  \frac{\alpha_z(\x_{i,n_z + 1/2})}{\alpha_x(\x_{i,n_z + 1/2})} \left (  \u_{i,n_z +1} \partial^-_z \v_{i,n_z +1} - \v_{i,n_z +1} \partial^-_z \u_{i,n_z +1} \right)        \bigg].
\end{align}
In each subdomain this expression is never evaluated inside the PML for $z$, hence $ \alpha_z(\x) = 1$. Moreover, if $\u$ and $\v$ both satisfy homogeneous Dirichlet boundary conditions at the vertical boundaries ($ i = -n_{\text{pml}} $ and $i = n_x+n_{\text{pml}}+1 $, which are imposed in the formulation of the PML), we obtain
\begin{align*}
\langle \mathbf{H^s}\u, \v \rangle_{\Omega^{\ell}}  =	& \,\, \langle \u, \mathbf{H^s} \v \rangle_{\Omega^{\ell}}   \\
										& - \sum_{i=-n_{\text{pml}} + 1}^{n_x + n_{\text{pml}} } \bigg[ \frac{1}{\alpha_x(\x_{i,1/2})} \left (  \u_{i,0} \partial^+_z \v_{i,0} - \v_{i,0} \partial^+_z \u_{i,0} \right)   \\
										& -  \frac{1}{\alpha_x(\x_{i,n_z + 1/2})} \left (  \u_{i,n_z +1} \partial^-_z \v_{i,n_z +1} - \v_{i,n_z +1} \partial^-_z \u_{i,n_z +1} \right)        \bigg] .
\end{align*}
We can then replace $\u$ and $\v$ by $\u^s$, the solution to $\H^s \u^s = \mathbf{f}$, and $\mathbf{G^s}$. By construction both satisfy homogeneous Dirichlet boundary conditions at $ i = -n_{\text{pml}}$ and $i = n_x+1 + n_{\text{pml}}$, therefore, we obtain
\begin{align} \notag
\langle \mathbf{H^s}\mathbf{G^s}, \u^s \rangle_{\Omega^{\ell}}  =	& \,\, \mathbf{u}^s \\
										= & \,\, \langle \mathbf{G^s}, \mathbf{f} \rangle_{\Omega^{\ell}}   - \sum_{i=-n_{\text{pml}} + 1}^{n_x + n_{\text{pml}} }  \bigg[ \frac{1}{\alpha_x(\x_{i,1/2})} \left (  \mathbf{G^s}_{i,0} \partial^+_z \u^s_{i,0} - \u^s_{i,0} \partial^+_z \mathbf{G^s}_{i,0} \right)  \\
										& -  \frac{1}{\alpha_x(\x_{i,n_z + 1/2})} \left (  \mathbf{G^s}_{i,n_z +1} \partial^-_z \u^s_{i,n_z +1} - \u^s_{i,n_z +1} \partial^-_z \mathbf{G^s}_{i,n_z +1} \right)        \bigg].
\end{align}
Moreover, using the relation between both Green's functions, the independence of $\alpha_x$ with respect to z, the point-wise equality between $\u$ and $\u^s$, and the properties of the support of $\mathbf{f}$; we obtain the Green's representation formula for the layered partition as
\begin{align} \label{eq:GRF_from_symmetric}
 \mathbf{u} & = \,\, \langle \mathbf{G}, \mathbf{f} \rangle   + \sum_{i=-n_{\text{pml}} + 1}^{n_x + n_{\text{pml}} }  \bigg[  \left ( - \mathbf{G}_{i,0} \partial^+_z \u_{i,0} + \u_{i,0} \partial^+_z \mathbf{G}_{i,0} \right) + \left (  \mathbf{G}_{i,n_z +1} \partial^-_z \u_{i,n_z +1} - \u_{i,n_z +1} \partial^-_z \mathbf{G}_{i,n_z +1} \right)        \bigg].
\end{align}
It is the discrete version of Eq. \ref{eq:GRF_slab}.

\section{Triangular and block triangular matrices} \label{appendix:general_facts}

In this section we introduce the necessary tools for the proofs in Appendix \ref{appendix:green_representation}. They are by no means original.

Let $H$ be the symmetric discretized Helmholtz operator in 1D, where
	\begin{equation}
		H = \left [  \begin{array}{ccccc}
				b_{1}   & c_{1} 	& 0  		& \hdots 	& 0  	 \\
				a_{1} 	& b_{2} 	&  c_{2}    & \hdots 	& 0  	\\
				0 		& \ddots 	&  \ddots  	& \ddots 	& 0  	 \\
				0 		& \ddots 	&  a_{n-2}  &  b_{n-1}	& c_{n-1} \\
				0 		& \hdots 	&  0  		& a_{n-1}   & b_{n}  		 \\
				\end{array} \right],
 	\end{equation}
 	in which $a = c$ by symmetry. We denote by $G$ the inverse of $H$.

 	We follow \cite{Usmani:inversion_of_jacobi_tridiagonal_matrix} in writing a simple recurrence formula based on the minors
 	\begin{align}
 		\theta_{-1} &= 0, &\theta_0    &= 1,  &\theta_i &=  b_{i} \theta_{i-1} - a_{i} c_{i-1} \theta_{i-2}, \\
 		\phi_{n+2}  &= 0,  &\phi_{n+1} &= 1,  &\phi_{i} &= b_{i} \phi_{i+1} - c_{i} a_{i+1} \phi_{i+2}.
 	\end{align}

 \begin{lemma} [ Lemma 2 in \cite{Usmani:inversion_of_jacobi_tridiagonal_matrix}  ] \label{lemma:casorati_1D}
 	We have the following identity:
 	\begin{equation}
 		\theta_i \phi_{i+1} - a_{i+1}c_i \theta_{i-1} \phi_{i+2} = \theta_{n}, \qquad \forall  1\leq i \leq n.
 	\end{equation}
 \end{lemma}

\begin{theorem}[ Theorem 1 in \cite{Usmani:inversion_of_jacobi_tridiagonal_matrix} ] \label{thm:inverse_1D}
	If $H$ is nonsingular, then its inverse is given by
	\begin{equation}
		(H^{-1})_{i,j} = \left \{  \begin{array}{cr}
							(-1)^{i+j} \left( \Pi_{k=i}^{\ell-1} c_k \right) \frac{\theta_{i-1} \phi_{j+1}}{\theta_n} 	& \text{ if } i<j, \\
							\frac{\theta_{i-1}\phi_{i+1}}{\theta_n}     											& \text{ if } i=j, \\
							(-1)^{i+j} \left(\Pi_{k=j+1}^{i} a_k \right) \frac{\theta_{j-1} \phi_{i+1}}{\theta_n}	& \text{ if } i>j. \\
							\end{array}    \right.
	\end{equation}
\end{theorem}

\begin{proposition} (Rank-one property).  We have
 \begin{align}
	\left ( H^{-1}_{i+1,i+1} \right)^{-1}  H^{-1}_{i+1,i}   H^{-1}_{i+1,k}  &=  H^{-1}_{i,k}, \qquad &\text{for } i < k, \label{eq:extrapolator_1D_up} \\
	\left ( H^{-1}_{i-1,i-1} \right)^{-1}  H^{-1}_{i-1,i}   H^{-1}_{i-1,k}  &=  H^{-1}_{i,k}, \qquad &\text{for } i > k. \label{eq:extrapolator_1D_down}
\end{align}
Moreover
\begin{align}
	\left [ \left ( H^{-1}_{i+1,i+1} \right)^{-1}  H^{-1}_{i+1,i} \right ] H^{-1}_{i+1,i}  = H^{-1}_{i,i}  + c_{i}^{-1}  \left [ \left ( H^{-1}_{i+1,i+1} \right)^{-1}  H^{-1}_{i+1,i} \right ],\label{eq:extrapolator_jump_1D_up} \\
	\left [ \left ( H^{-1}_{i-1,i-1} \right)^{-1}  H^{-1}_{i-1,i} \right ] H^{-1}_{i-1,i}  = H^{-1}_{i,i}  + a_{i}^{-1}  \left [ \left ( H^{-1}_{i-1,i-1} \right)^{-1}  H^{-1}_{i-1,i} \right ].\label{eq:extrapolator_jump_1D_down}
\end{align}
\end{proposition}

\begin{proof}
Eq. \ref{eq:extrapolator_1D_up} and Eq. \ref{eq:extrapolator_1D_down} are direct application of the expression for the inverse of $H$ given by Thm. \ref{thm:inverse_1D}.

We only prove Eq. \ref{eq:extrapolator_jump_1D_up} -- the proof of Eq. \ref{eq:extrapolator_jump_1D_down} is analogous. Using the expression of the inverse given by Thm. \ref{thm:inverse_1D} we have
\begin{equation}
	\left ( H^{-1}_{i+1,i+1} \right)^{-1}  H^{-1}_{i+1,i} = -c_i\frac{\theta_{i-1}}{\theta_i}.
\end{equation}
Thus,
\begin{align*}
	\left [ \left ( H^{-1}_{i+1,i+1} \right)^{-1}  H^{-1}_{i+1,i} \right ] H^{-1}_{i+1,i} 	& = c_i  \frac{\theta_{i-1}}{\theta_{i}}  a_{i+1} \frac{\theta_{i-1} \phi_{i+2}}{\theta_n} \\
																							& = \frac{\theta_{i-1}\phi_{i+1}}{\theta_n} - \frac{\theta_{i-1}}{ \theta_{i} } \left ( \frac{    \theta_i \phi_{i+1} - a_{i+1} c_i \theta_{i-1}\phi_{i+1} }{\theta_{n}} \right ) \\
																							& =  \frac{\theta_{i-1}\phi_{i+1}}{\theta_n} - \frac{\theta_{i-1}}{ \theta_{i} }\\
																							& = H^{-1}_{i,i}  + c_{i}^{-1}  \left [ \left ( H^{-1}_{i+1,i+1} \right)^{-1}  H^{-1}_{i+1,i} \right ].
\end{align*}
Above, we used Lemma \ref{lemma:casorati_1D} and the expression for the inverse given by Thm. \ref{thm:inverse_1D}.
\end{proof}

For the 2D case, we introduce the same ordering as in \cite{EngquistYing:Sweeping_H}, where we increase the index in $x$ first,
\begin{equation}
	\u = (u_{1,1}, u_{2,1}, ..., u_{n_x,1}, u_{1,2}, ... , u_{n_x,n_z}).
\end{equation}
For simplicity of exposition, and only for this section, we do not take into account the degrees of freedom within the PML.
Let $\mathbf{H}$ be the discrete symmetric Helmholtz operator in 2D (Eq. \ref{eq:Hemholtz_symmetric}), which we rewrite as
\begin{equation} \label{eq:def_H_appendix}
		\mathbf{H} = \left [\begin{array}{ccccc}
								\mathbf{H}_{1}   & \mathbf{C}_{1} 	& 0  		& \hdots 	& 0  	 \\
								\mathbf{C}_{1} 	& \mathbf{H}_{2} 	&  \mathbf{C}_{2}    & \hdots 	& 0  	\\
								0 		& \ddots 	&  \ddots  	& \ddots 	& 0  	 \\
								0 		& \ddots 	&  \mathbf{C}_{n_z-2}  &  \mathbf{H}_{n_z-1}	& \mathbf{C}_{n_z-1} \\
								0 		& \hdots 	&  0  		& \mathbf{C}_{n_z-1}   & \mathbf{H}_{n_z}  		 \\
							\end{array} \right],
 	\end{equation}
in which each sub-block is a matrix in $\CC^{n_x \times n_x}$ matrix (or a  matrix in $\CC^{n_x+2n_{\text{pml}},n_x+2n_{\text{pml}}}$  if we count the degrees of freedom within the PML). Each $\mathbf{C}_i$ is a constant times the identity, and each $\mathbf{H}_{i}$ is a tridiagonal symmetric matrix. The ordering of the unknowns implies that every block in $\mathbf{H}$ correspond to degrees of freedom with fixed depth (fixed z).

In order to prove the equivalent of the rank-one property in the 2D case we follow \cite{Bevilacqua:parallel_solution_of_block_tridiagonal_linear_systems} and \cite{Meurant:a_review_on_the_inverse_of_tridiagonal_matrices}.

\begin{definition} \label{def:recurrence_delta_sigma}
	Let $\Delta_i$ and $\Sigma_i$ be defined by the following recurrence relations
	\begin{align}
		\Delta_1 	 &= \mathbf{H}_{1}, 	\qquad &\Delta_{i} &= \mathbf{H}_i - \mathbf{C}_i \left (\Delta_{i-1} \right )^{-1} \mathbf{C}_i^t;  \\
		\Sigma_{n_z} &= \mathbf{H}_{n_z}, 	\qquad &\Sigma_{i} &= \mathbf{H}_i - \mathbf{C}_{i+1}^t \left (\Sigma_{i+1} \right )^{-1} \mathbf{C}_{i+1}.
	\end{align}
\end{definition}

\begin{proposition} \label{prop:H_inverted_UV_construction}
$\mathbf{H}$ is proper\footnote{A block Hessenberg matrix, with invertible upper and lower diagonal blocks. See Def. 2.2 in \cite{Bevilacqua:parallel_solution_of_block_tridiagonal_linear_systems}.}, and its inverse is given by
\begin{equation} \label{eq:inverse_2D_v1}
		\mathbf{H}^{-1}_{j,k} = \left \{  \begin{array}{cr}
							\mathbf{U}_j \mathbf{V}^t_k  	& \text{ if } j\leq k, \\
							\mathbf{V}_j \mathbf{U}^{t}_k 	& \text{ if } j\geq k. \\
							\end{array}    \right.
	\end{equation}
	where $ \mathbf{U}_j = \mathbf{C}_j^{-t} \Delta_{j-1}... \mathbf{C}_2^{-t} \Delta_1$ and $\mathbf{V}^t_j = \Sigma_1^{-1} \mathbf{C}_2^t ... \mathbf{C}_j \Sigma^{-1}_j$.
\end{proposition}

\begin{proof}
$\mathbf{H}$  is proper because $\mathbf{C}_j$ are invertible; Eq. \ref{eq:inverse_2D_v1} is a direct application of Theorem 3.4 in \cite{Meurant:a_review_on_the_inverse_of_tridiagonal_matrices}.
\end{proof}

\begin{proposition} \label{proposition:extrapolation_matrix}
$ \left ( \mathbf{H}^{-1}_{j+1,j+1} \right)^{-1}  \mathbf{H}^{-1}_{j+1,j}$ is symmetric and we have
\begin{equation} \label{eq:extrapolation_matrix_form_up}
	\left ( \mathbf{H}^{-1}_{j+1,j+1} \right)^{-1}  \mathbf{H}^{-1}_{j+1,j}   \mathbf{H}^{-1}_{j+1,k}  =  \mathbf{H}^{-1}_{j,k}, \qquad \text{for } j < k.
\end{equation}
Moreover,
\begin{equation} \label{eq:extrapolation_matrix_form_down}
	\left ( \mathbf{H}^{-1}_{j-1,j-1} \right)^{-1}  \mathbf{H}^{-1}_{j-1,j}   \mathbf{H}^{-1}_{j-1,k}  =  \mathbf{H}^{-1}_{j,k}, \qquad \text{for } j > k.
\end{equation}

\end{proposition}

\begin{proof}
	First, it is easy to prove that $\Delta_j$ are symmetric matrices using an inductive argument and Def. \ref{def:recurrence_delta_sigma}. Then, using Eq. \ref{eq:inverse_2D_v1} we have
	\begin{align}
		\left ( \mathbf{H}^{-1}_{j+1,j+1} \right)^{-1}  \mathbf{H}^{-1}_{j+1,j} 	& = \left( \mathbf{V}_{j+1} \mathbf{U}_{j+1}^t \right)^{-1}  \mathbf{V}_{j+1} \mathbf{U}_{j}^t,  \notag \\
															& = \mathbf{U}_{j+1}^{-t} \mathbf{U}_{j}^t, \notag \\
															& = \left (\mathbf{C}_{j+1}^{-t} \Delta_{j}... \mathbf{C}^{-t}_2 \Delta_1 \right)^{-t} \left( \mathbf{C}_j^{-t} \Delta_{j-1}... \mathbf{C}^{-t}_2 \Delta_1 \right )^t, \notag \\
															& = \mathbf{C}_{j+1} \Delta_{j}^{-t}, \notag \\
															& = \mathbf{C}_{j+1} \Delta_{j}^{-1}.
	\end{align}
Using the symmetry of $\Delta_j$ and the fact that $\mathbf{C}_{j+1}$ is an identity times a constant, we obtain the desired symmetry property.

	Finally, a simple computation using Proposition \ref{prop:H_inverted_UV_construction} leads to Eq. \ref{eq:extrapolation_matrix_form_up}. The proof of Eq. \ref{eq:extrapolation_matrix_form_down} is analogous.
\end{proof}

\begin{definition}
Let $\mathbf{D}_j$ and $\mathbf{E}_j$ be defined by the following recurrences
\begin{align}
\mathbf{D}_1 		&= \mathbf{I}, &\mathbf{D}_2       &= - \mathbf{C}^{-1}_1\mathbf{H}_1, 			 &\mathbf{D}_j &= -\left( \mathbf{D}_{j-2} \mathbf{C}_{j-1} + \mathbf{D}_{j-1} \mathbf{H}_{j-1}  \right) \mathbf{C}_{j}^{-1} \,\,\, j = 3,...,n_z;     \\
\mathbf{E}_{n_z} 	&= \mathbf{j}, &\mathbf{E}_{n_z-1} &= - \mathbf{H}_{n_z}\mathbf{C}^{-1}_{n_z-1}, &\mathbf{E}_j &= -  \mathbf{C}_{j}^{-1} \left(  \mathbf{H}_{j+1} \mathbf{E}_{j+1} +  \mathbf{C}_{j+1} \mathbf{E}_{j+2} \right) \,\,\, j = n_z-2,...,1,
\end{align}
and define the generalized Casorati determinant by
\begin{equation}
	 \mathbf{R}_{j} = \mathbf{C}_{j-1} \left( \mathbf{D}_j \mathbf{E}_{j-1} - \mathbf{D}_{j-1} \mathbf{E}_j \right ), \qquad j=2,...,n_z; \qquad \mathbf{R}:=\mathbf{R}_{n_z}.
\end{equation}
\end{definition}

\begin{remark} \label{remark:invertibility}
We note that $\mathbf{D}_j$ and $\mathbf{E}_j$ are invertible. Indeed, we can see that $\mathbf{D}_j$ has $n$ linearly independent solutions to the three-term recurrence. Then the determinant of $\mathbf{D}_j$ is always different from zero. The same is true for $\mathbf{E}_j$ but using the  backwards recurrence.
\end{remark}

\begin{proposition}
For the sequence of matrices $\mathbf{D}_j$ and $\mathbf{E}_j$, its generalized Casorati determinant is constant.
\end{proposition}

\begin{proof}
We compute
\begin{align}
	 \mathbf{R}_{j} -  \mathbf{R}_{j+1} &=  \mathbf{C}_{j-1} \left( \mathbf{D}_j \mathbf{E}_{j-1} - \mathbf{D}_{j-1} \mathbf{E}_j \right )  - \mathbf{C}_{j} \left( \mathbf{D}_{j+1} \mathbf{E}_{j} - \mathbf{D}_{j} \mathbf{E}_{j+1} \right ), \\ \notag
 										&=  \mathbf{D}_j \left( \mathbf{C}_{j-1}\mathbf{E}_{j-1} + \mathbf{C}_{j}\mathbf{E}_{j+1} \right) -   \left( \mathbf{D}_{j+1}\mathbf{C}_{j} +  \mathbf{D}_{j-1} \mathbf{C}_{j-1} \right) \mathbf{E}_{j},   \\ \notag
 										&=  \mathbf{D}_j  \mathbf{H}_j  \mathbf{E}_j - \mathbf{D}_j  \mathbf{H}_j  \mathbf{E}_j, \\ \notag
 										&= 0,
\end{align}
where we used the fact that the $\mathbf{C}_j$ are a constant times the identity (then they commute with all the matrices) and the recurrences satisfied by $\mathbf{D}_j $ and $ \mathbf{E}_j$.
\end{proof}

\begin{proposition}
The inverse of $ \mathbf{H}$ is given by
	\begin{equation} \label{eq:inverse_matrix_form_v2}
		\mathbf{H}^{-1}_{j,k} = \left \{  \begin{array}{cr}
							-\mathbf{D}_j^t \mathbf{R}^{-t}   \mathbf{E}^t_k  	& \text{ if } j\leq k, \\
							-\mathbf{E}_j   \mathbf{R}^{-1}   \mathbf{D}_k 		& \text{ if } j\geq k. \\
							\end{array}    \right.
	\end{equation}
\end{proposition}

\begin{proof}
It is a direct application of Proposition 2.4 in \cite{Bevilacqua:parallel_solution_of_block_tridiagonal_linear_systems}.
\end{proof}

\begin{proposition} \label{proposition:jump_condition}
We have
\begin{equation} \label{eq:jump_condition_proposition_up}
	\left [ \left ( \mathbf{H}^{-1}_{j+1,j+1} \right)^{-1}  \mathbf{H}^{-1}_{j+1,j} \right ] \mathbf{H}^{-1}_{j+1,j}  = \mathbf{H}^{-1}_{j,j}  + \mathbf{C}_{j+1}^{-1}  \left [ \left ( \mathbf{H}^{-1}_{j+1,j+1} \right)^{-1}  \mathbf{H}^{-1}_{j+1,j} \right ],
\end{equation}
and
\begin{equation} \label{eq:jump_condition_proposition_down}
	\left [ \left ( \mathbf{H}^{-1}_{j-1,j-1} \right)^{-1}  \mathbf{H}^{-1}_{j-1,j} \right ] \mathbf{H}^{-1}_{j-1,j}  = \mathbf{H}^{-1}_{j,j}  + \mathbf{C}_{j-1}^{-1}  \left [ \left ( \mathbf{H}^{-1}_{j-1,j-1} \right)^{-1}  \mathbf{H}^{-1}_{j-1,j} \right ].
\end{equation}

\end{proposition}

\begin{proof}
From Eq. \ref{eq:inverse_matrix_form_v2} we have
\begin{equation}
	\left ( \mathbf{H}^{-1}_{j+1,j+1} \right)^{-1}  \mathbf{H}^{-1}_{j+1,j} = \mathbf{D}_{j+1}^{-1} \mathbf{D}_i.
\end{equation}
Then we can compute
\begin{align}
 	\left [ \left ( \mathbf{H}^{-1}_{j+1,j+1} \right)^{-1}  \mathbf{H}^{-1}_{j+1,j} \right ] \mathbf{H}^{-1}_{j+1,j}  	&= -\mathbf{D}_{j+1}^{-1} \mathbf{D}_j \mathbf{E}_{j+1} \mathbf{R}^{-1} \mathbf{D}_{j}, \\
 																														&= -\mathbf{E}_j \mathbf{R}^{-1} \mathbf{D}_j -  \mathbf{D}_{j+1}^{-1} \mathbf{D}_j \left (  \mathbf{E}_{j+1} - \mathbf{D}_j ^{-1}\mathbf{D}_{j+1}\mathbf{E}_j\right)\mathbf{R}^{-1} \mathbf{D}_{j}, \\
 																														&= -\mathbf{E}_j \mathbf{R}^{-1} \mathbf{D}_j -  \mathbf{D}_{j+1}^{-1} \mathbf{D}_j \mathbf{D}_j ^{-1} \left (\mathbf{D}_j\mathbf{E}_{j+1} - \mathbf{D}_{j+1}\mathbf{E}_j\right)\mathbf{R}^{-1} \mathbf{D}_{j}, \\
 																														& = -\mathbf{E}_j \mathbf{R}^{-1} \mathbf{D}_j + \mathbf{D}_{j+1}^{-1} \mathbf{D}_j \mathbf{D}_j ^{-1} \mathbf{C}_{j}^{-1} \mathbf{R}_{j+1} \mathbf{R}^{-1} \mathbf{D}_{j},\\
 																														& = -\mathbf{E}_j \mathbf{R}^{-1} \mathbf{D}_j + \mathbf{C}_{j}^{-1} \mathbf{D}_{j+1}^{-1} \mathbf{D}_j, \\
 																														& = \mathbf{H}^{-1}_{j,j}  + \mathbf{C}_{j}^{-1}  \left [ \left ( \mathbf{H}^{-1}_{j+1,j+1} \right)^{-1}  \mathbf{H}^{-1}_{j+1,j} \right ],
 \end{align}
in which we used the fact that the generalized Casorati determinand is constant.
\end{proof}

\section{Properties of the Discrete Green's representation formula} \label{appendix:green_representation}

\subsection*{Proof of Lemma \ref{lemma:discrete_GRF} }

\begin{proof} We carry out the proof in 1D. The extension to the 2D case is trivial because of the Eq. \ref{eq:GRF_from_symmetric}, which is the discrete Green's representation formula in 2D, that takes in account the boundary conditions of $\u$ and $\mathbf{G}^{\ell}$.
Let $H$ be the discrete Helmholtz operator in 1D, and let $u$ be the solution of
	\[ \left( H u \right)_i = f_i, \qquad \text{for } i \in \ZZ, \]
Let $\Omega = \{ 1, ..., n\}$. We define the discrete inner product as
	\[ \langle u, v \rangle_{\Omega} = \sum_{i =1}^{n} u_i v_i, \]
and the Green's function $G_i^k$, such that
	\[ \left( H G^k \right)_i = \delta_i^k, \qquad \text{for } i \in \ZZ. \]

Following the discretization given in Eq \ref{eq:quadrature}, we can write
\begin{equation} \label{eq:discrete_GRF_1D_appendix}
	\langle u, HG^k \rangle_{\Omega} - \langle G^k, Hu \rangle_{\Omega} = \cG^{\downarrow}_k(u_0, u_{1}) + \cG^{\uparrow}_k(u_n, u_{n+1}),\qquad \text{if } k \in \ZZ.
\end{equation}
Applying the properties of $G_k$ and the fact that $u$ is the solution, we have
\begin{equation}
	\langle u, \delta^k \rangle_{\Omega} - \langle G^k, f \rangle_{\Omega} = \cG^{\downarrow}_k(u_0, u_{1}) + \cG^{\uparrow}_k(u_n, u_{n+1}) ,\qquad \text{if } k \in \ZZ.
\end{equation}
Following Def. \ref{def:Newton_potential} and Def. \ref{def:incomplete_green}, we have
\begin{equation} \label{eq:discrete_GRF_1D_proof_1}
	u_k = \cG^{\downarrow}_k(u_0, u_{1})  +\cG^{\uparrow}_k(u_n, u_{n+1}) + \cN_k f,\qquad \text{if } 1\leq k \leq n .
\end{equation}
If $k < 1$ or $k > n$, the formula given by Eq. \ref{eq:discrete_GRF_1D_proof_1} is still valid. However, we have
\[\langle u, \delta^k \rangle_{\Omega} = 0,\]
because the support of the Dirac's delta is outside the integration domain. We obtain
\begin{equation}
	 - \langle G^k, f \rangle_{\Omega} = \cG^{\uparrow}_k(u_n, u_{n+1}) + \cG^{\downarrow}_k(u_0, u_{1}),\qquad \text{if } k<1 \text{ or } k>n,
\end{equation}
thus,
\begin{equation}
	0 = \cG^{\uparrow}_k(u_n, u_{n+1}) + \cG^{\downarrow}_k(u_0, u_{1}) + \cN_k f,\qquad \text{if } k<1 \text{ or } k>n.
\end{equation}

Finally, the only property that we used from $G_k$ is that is should satisfy
\[ \left( H G^k \right)_i = \delta_i^k, \qquad \text{for } 0 \leq l \leq n+1. \]
We can then replace the global Green's function by local ones, and the results still hold, i.e.,
\begin{align}
	u_k &= \cG^{\uparrow,1}_k(u_n, u_{n+1}) + \cG^{\downarrow,1}_k(u_0, u_{1}) + \cN_k f,\qquad \text{if } 1\leq k \leq n ,\\
	0   &= \cG^{\uparrow,1}_k(u_n, u_{n+1}) + \cG^{\downarrow,1}_k(u_0, u_{1}) + \cN_k f,\qquad \text{if } k<1 \text{ or } k>n,
\end{align}
which finishes the proof.

\end{proof}

\subsection*{Proof of Lemma \ref{lemma:jump_condition} }

\begin{proof}
	We carry out the proof for $k=0$, the case for $k = n^{\ell}+1$ is analogous.
	Let us fix $\ell$. By definition
	\begin{equation}
		\cG^{\downarrow, \ell}_{1}(\u^{\ell}_{0}, \u^{\ell}_{1} ) 	+ \cG^{\uparrow, \ell}_{1}(\u^{\ell}_{n^{\ell}}, \u^{\ell}_{n^{\ell}+1} )  	+  \cN^{\ell}_{1}  \mathbf{f}^{\ell}  	= \u^{\ell}_{1},
	\end{equation}
	to which we apply the extrapolator, obtaining
	\begin{equation}
		\cE^{\uparrow}_{\ell-1,\ell} \cG^{\downarrow, \ell}_{1}(\u^{\ell}_{0}, \u^{\ell}_{1} ) 	+ \cE^{\uparrow}_{\ell-1,\ell}  \cG^{\uparrow, \ell}_{1}(\u^{\ell}_{n^{\ell}}, \u^{\ell}_{n^{\ell}+1} )  	+  \cE^{\uparrow}_{\ell-1,\ell}  \cN^{\ell}_{1}  \mathbf{f}^{\ell}  	=  \cE^{\uparrow}_{\ell-1,\ell} \u^{\ell}_{1}.
	\end{equation}
	We compute each component of the last equation to show that it is equivalent to Eq. \ref{eq:lemma_jump}.
	Indeed, we use the definition of the extrapolator (Def. \ref{def:extrapolator}) and Lemma \ref{lemma:extrapolator_def} to show that,
	\begin{equation}
		\cE^{\uparrow}_{\ell-1,\ell}  \cG^{\uparrow, \ell}_{1}(\u^{\ell}_{n^{\ell}}, \u^{\ell}_{n^{\ell}+1} ) = \cG^{\uparrow, \ell}_{0}(\u^{\ell}_{n^{\ell}}, \u^{\ell}_{n^{\ell}+1} ), \text{ and} \qquad   \cE^{\uparrow}_{\ell-1,\ell}  \cN^{\ell}_{1}  \mathbf{f}^{\ell}  = \cN^{\ell}_{0}  \mathbf{f}^{\ell}.
	\end{equation}
	We compute the last term left using the matrix form of $\cG^{\ell}$ (Eq. \ref{eq:matrix_form_cG}) which results in
	 \begin{equation}
	\cE^{\uparrow}_{\ell-1,\ell} \cG^{\downarrow, \ell}_{1}(\u^{\ell}_{0}, \u^{\ell}_{1} ) = \frac{\cE^{\uparrow}_{\ell-1,\ell}}{h} 	\left [ \begin{array}{cc}
																																	\mathbf{G}^{\ell} (z_1,z_1)  & - \mathbf{G}^{\ell}(z_1,z_0)
																																\end{array}
																	  													\right]
																	  													\left ( \begin{array}{c}
																	  																\u^{\ell}_{0} \\
																	  																\u^{\ell}_{1}
																	  												   			\end{array}
																	  													\right).
	 \end{equation}
	Moreover, by direct application of  Lemma \ref{lemma:extrapolator_def} we have that
	\begin{equation}
	\cE^{\uparrow}_{\ell-1,\ell}  \mathbf{G}^{\ell} (z_1,z_1) = \mathbf{G}^{\ell} (z_0,z_1), \qquad
	\cE^{\uparrow}_{\ell-1,\ell} \mathbf{G}^{\ell}(z_1,z_0)  = \mathbf{G}^{\ell}(z_0,z_0)  - h\cE^{\uparrow}_{\ell-1,\ell};
	\end{equation}
	thus
	\begin{equation} \label{eq:jump_condition_in_u}
		\cE^{\uparrow}_{\ell-1,\ell} \cG^{\downarrow, \ell}_{1}(\u^{\ell}_{0}, \u^{\ell}_{1} )  = \cG^{\downarrow, \ell}_{0}(\u^{\ell}_{0}, \u^{\ell}_{1} ) - \cE^{\uparrow}_{\ell-1,\ell}\u^{\ell}_{1}.
	\end{equation}

	 Putting everything together we have that
	\begin{equation}
		\cG^{\downarrow, \ell}_{0}(\u^{\ell}_{0}, \u^{\ell}_{1} ) + \cG^{\uparrow, \ell}_{0}(\u^{\ell}_{n^{\ell}}, \u^{\ell}_{n^{\ell}+1} )  + \cE^{\uparrow}_{\ell-1,\ell} \u^{\ell}_{1}	+  \cN^{\ell}_{0}  \mathbf{f}^{\ell}  	=  \cE^{\uparrow}_{\ell-1,\ell} \u^{\ell}_{1},
	\end{equation}
	or
	\begin{equation}
		\cG^{\downarrow, \ell}_{0}(\u^{\ell}_{0}, \u^{\ell}_{1} ) + \cG^{\uparrow, \ell}_{0}(\u^{\ell}_{n^{\ell}}, \u^{\ell}_{n^{\ell}+1} ) +  \cN^{\ell}_{0}  \mathbf{f}^{\ell}  	= 0 ,
	\end{equation}
	which concludes the proof.
\end{proof}

\subsection*{Proof of Lemma \ref{lemma:extrapolator} }

\begin{proof}
    The proof is a direct application of the nullity theorem \cite{Strang:the_interplay_of_ranks_of_submatrices}. If $\u^{\uparrow}$ and $\v^{\uparrow}$ are in the kernel of $\cA^{\uparrow}_{j,j+1}$, then the proof is reduced to showing that $\mathbf{G}^{\ell+1}(z_1,z_1) $ is invertible. Without loss of generality we can reorder the entries of the matrix $\mathbf{G}^{\ell+1}$ such that $\mathbf{G}^{\ell+1}(z_1,z_1)$ is a square diagonal block located at the left upper corner of $\mathbf{G}^{\ell+1}$. Then the nullity of $\mathbf{G}^{\ell+1}(z_1,z_1)$ is equal to the nullity of the complementary block of the inverse, but the inverse is just the Helmholtz matrix reordered with some entries out. Such block is trivially full rank, i.e. nullity equals to zero. Then the nullity of $\mathbf{G}^{\ell+1}(z_1,z_1)$ is zero; therefore,  $\mathbf{G}^{\ell+1}(z_1,z_1)$ is an invertible matrix.
\end{proof}

\subsection*{Proof of Lemma \ref{lemma:extrapolator_def} }

\begin{proof}
	We note that by the definition of the local Green's functions in Eq. \ref{eq:definition_local_helmholtz_problem}, we have that
	\begin{equation}
		 \mathbf{G}^{\ell}_{i,j,i',j'} = \frac{1}{h^2} \left ( \mathbf{H}^{\ell} \right )^{-1}_{i,j,i',j'}.
	\end{equation}
	Then by Def. \ref{def:layer_2_layer_Green_fucntion} and Eq. \ref{eq:def_H_appendix} we have that
	\begin{equation} \label{eq:equivalence_G_H_inv}
		\mathbf{G}^{\ell}_{j,k} = \frac{1}{h} \left ( \mathbf{H}^{\ell} \right )^{-1}_{j,k},
	\end{equation}
	where $\mathbf{G}^{\ell}_{j,k}$ is the layer to layer Green's function. Using the definition of the extrapolator (Def. \ref{def:extrapolator}) and Proposition \ref{proposition:extrapolation_matrix}, in particular Eq. \ref{eq:extrapolation_matrix_form_down} applied to $ j = 0 $ we obtain
	\begin{equation}
		\mathbf{G}^{\ell}(z_0,z_k) = \cE^{\uparrow}_{\ell-1,\ell} \mathbf{G}^{\ell}(z_1,z_k), \qquad \text{for } 0 < k,
	\end{equation}
	and Eq. \ref{eq:extrapolation_matrix_form_up} applied to $j = n^{\ell}+1$
	\begin{equation}
		\mathbf{G}^{\ell}(z_{n^{\ell}+1},z_k) = \cE^{\downarrow}_{\ell,\ell+1} \mathbf{G}^{\ell}(z_{n^{\ell}},z_k), \qquad \text{for }  k < n^{\ell}+1.
	\end{equation}

	We can divide Eq. \ref{eq:jump_condition_proposition_down} by $h$, and we use the definitions of the extrapolator (Def. \ref{def:extrapolator}) and the Green's functions (Eq. \ref{eq:equivalence_G_H_inv}) to obtain
	\begin{equation} \label{eq:jump_proof_appendix}
		\mathbf{G}^{\ell+1}(z_0,z_0) + \frac{\mathbf{C}^{-1}_2 \cE^{\uparrow}_{\ell,\ell+1}}{h}  =    \cE^{\uparrow}_{\ell,\ell+1}   \mathbf{G}^{\ell+1}(z_1,z_0).
	\end{equation}
	However, following the notation of Eq. \ref{eq:def_H_appendix}, we note that if $k$ is such that it corresponds to a $z_k$ that is in the physical domain, then $\mathbf{C}_k$ is just $-I/h^2$. This observation is independent of the formulation, and it is due to the particular ordering of the unknowns that Eq. \ref{eq:def_H_appendix} assumes. Then we can further simplify Eq. \ref{eq:jump_proof_appendix} and obtain
	\begin{equation}
		\mathbf{G}^{\ell+1}(z_0,z_0) - h\cE^{\uparrow}_{\ell,\ell+1}  =    \cE^{\uparrow}_{\ell,\ell+1}   \mathbf{G}^{\ell+1}(z_1,z_0).
	\end{equation}
	We can follow the same reasoning to obtain from Eq. \ref{eq:jump_condition_proposition_up} that
	\begin{equation}
		\mathbf{G}^{\ell}(z_{n^{\ell}+1},z_{n^{\ell}+1}) - h \cE^{\downarrow}_{\ell,\ell+1}  = \cE^{\downarrow}_{\ell,\ell+1}  \mathbf{G}^{\ell}(z_{n^{\ell}},z_{n^{\ell}+1}),
	\end{equation}
	which concludes the proof.

\end{proof}

\subsection*{Proof of Lemma \ref{lemma:annihilator_relations_volume} }

\begin{proof}
We have that
\begin{equation}
	 \cG^{\downarrow, \ell+1}_{1}(\u^{\uparrow}_0, \u^{\uparrow}_1 ) = 0
\end{equation}
is, by definition, equivalent to
\begin{equation} \label{eq:equivalent }
	 \mathbf{G}^{\ell+1}(z_1,z_1) \u^{\uparrow}_0 = \mathbf{G}^{\ell+1}(z_1,z_0)  \u^{\uparrow}.
\end{equation}
The proof is by induction, at each stage we use the extrapolator to shift the evaluation index. We left multiply Eq. \ref{eq:equivalent } by $ \left [ \mathbf{G}^{\ell+1}(z_1,z_1) \right ]^{-1} \mathbf{G}^{\ell+1}(z_1,z_2) $ and follow Remark \ref{remark:extrapolator}, to obtain
\begin{equation}
	\mathbf{G}^{\ell+1}(z_2,z_1) \u^{\uparrow}_0 = \mathbf{G}^{\ell+1}(z_2,z_0)  \u^{\uparrow},
\end{equation}
which can be left multiplied by the matrix $ \left [ \mathbf{G}^{\ell+1}(z_2,z_2) \right ]^{-1} \mathbf{G}^{\ell+1}(z_2,z_3) $ to obtain
\begin{equation}
	\mathbf{G}^{\ell+1}(z_3,z_1) \u^{\uparrow}_0 = \mathbf{G}^{\ell+1}(z_3,z_0)  \u^{\uparrow}.
\end{equation}
Then by induction we obtain the result.
\end{proof}

\subsection*{Proof of Lemma \ref{lemma:extrapolator_M_0} }
\begin{proof}
We give the proof for the case when $j = 0$ -- for $j = n^{\ell}+1$ the proof is analogous.

Given that $\underline{\underline{\u}}$ is solution of the system in Def. \ref{def:polarized_out_going}, we have

\begin{equation}
	\cG^{\uparrow, \ell}_{1}(\u^{\ell, \uparrow}_{n^{\ell}}, \u^{\ell, \uparrow}_{n^{\ell}+1} ) + \cG^{\downarrow, \ell}_{1}(\u^{\ell,\downarrow}_{0}, \u^{\ell, \downarrow}_{1} ) 	+ \cG^{\uparrow, \ell}_{1}(\u^{\ell, \downarrow}_{n^{\ell}}, \u^{\ell, \downarrow}_{n^{\ell}+1} ) 	+  \cN^{\ell}_{1}  \mathbf{f}^{\ell}  = \u^{\ell, \uparrow}_{1} + \u^{\ell, \downarrow}_{1},
\end{equation}
which can be left-multiplied by the extrapolator $\cE^{\uparrow}_{\ell, \ell+1}$. Using Lemma \ref{lemma:extrapolator_def} we have
\begin{equation}
	\cG^{\uparrow, \ell}_{0}(\u^{\ell, \uparrow}_{n^{\ell}}, \u^{\ell, \uparrow}_{n^{\ell}+1} ) +  \cE^{\uparrow}_{\ell, \ell+1} \cG^{\downarrow, \ell}_{1}(\u^{\ell,\downarrow}_{0}, \u^{\ell, \downarrow}_{1} ) 	+ \cG^{\uparrow, \ell}_{0}(\u^{\ell, \downarrow}_{n^{\ell}}, \u^{\ell, \downarrow}_{n^{\ell}+1} ) 	+  \cN^{\ell}_{1}  \mathbf{f}^{\ell}  = \cE^{\uparrow}_{\ell, \ell+1} \u^{\ell, \uparrow}_{1} + \cE^{\uparrow}_{\ell, \ell+1} \u^{\ell, \downarrow}_{1},
\end{equation}
and following the same computation performed in Lemma \ref{lemma:jump_condition}  (Eq. \ref{eq:jump_condition_in_u}) we have
\begin{equation}
	\cG^{\uparrow, \ell}_{0}(\u^{\ell, \uparrow}_{n^{\ell}}, \u^{\ell, \uparrow}_{n^{\ell}+1} ) +   \cG^{\downarrow, \ell}_{0}(\u^{\ell,\downarrow}_{0}, \u^{\ell, \downarrow}_{1} ) + \cE^{\uparrow}_{\ell, \ell+1} \u^{\ell,\downarrow}_{1}	+ \cG^{\uparrow, \ell}_{0}(\u^{\ell, \downarrow}_{n^{\ell}}, \u^{\ell, \downarrow}_{n^{\ell}+1} ) 	+  \cN^{\ell}_{0}  \mathbf{f}^{\ell}  = \cE^{\uparrow}_{\ell, \ell+1} \u^{\ell, \uparrow}_{1} + \cE^{\uparrow}_{\ell, \ell+1} \u^{\ell, \downarrow}_{1}.
\end{equation}
Finally, from the fact that $\cE^{\uparrow}_{\ell, \ell+1} \u^{\ell, \uparrow}_{1} =  \u^{\ell, \uparrow}_{0}$ we obtain that
\begin{equation}
	\cG^{\uparrow, \ell}_{0}(\u^{\ell, \uparrow}_{n^{\ell}}, \u^{\ell, \uparrow}_{n^{\ell}+1} ) +   \cG^{\downarrow, \ell}_{0}(\u^{\ell,\downarrow}_{0}, \u^{\ell, \downarrow}_{1} ) + \cG^{\uparrow, \ell}_{0}(\u^{\ell, \downarrow}_{n^{\ell}}, \u^{\ell, \downarrow}_{n^{\ell}+1} ) 	+  \cN^{\ell}_{0}  \mathbf{f}^{\ell}  = \u^{\ell, \uparrow}_{0}.
\end{equation}
\end{proof}

\subsection*{Proof of Proposition \ref{lemma:equivalence_formulations} }
\begin{proof}
  The sufficient condition is given by Lemma \ref{lemma:extrapolator_M_0}, so we focus on the necessary condition.
The proof can be reduced to showing that
\begin{equation}
	\cG^{\uparrow, \ell}_{0}(\u^{\ell, \uparrow}_{n^{\ell}}, \u^{\ell, \uparrow}_{n^{\ell}+1} ) +   \cG^{\downarrow, \ell}_{0}(\u^{\ell,\downarrow}_{0}, \u^{\ell, \downarrow}_{1} ) + \cG^{\uparrow, \ell}_{0}(\u^{\ell, \downarrow}_{n^{\ell}}, \u^{\ell, \downarrow}_{n^{\ell}+1} ) 	+  \cN^{\ell}_{0}  \mathbf{f}^{\ell}  = \u^{\ell, \uparrow}_{0},
\end{equation}
implies $\cE^{\uparrow}_{\ell, \ell+1} \u^{\ell, \uparrow}_{1} =  \u^{\ell, \uparrow}_{0}$.
Indeed, we have
\begin{equation}
	\cG^{\uparrow, \ell}_{0}(\u^{\ell, \uparrow}_{n^{\ell}}, \u^{\ell, \uparrow}_{n^{\ell}+1} ) +   \cG^{\downarrow, \ell}_{0}(\u^{\ell,\downarrow}_{0}, \u^{\ell, \downarrow}_{1} )  + \cE^{\uparrow}_{\ell, \ell+1} \u^{\ell,\downarrow}_{1} + \cG^{\uparrow, \ell}_{0}(\u^{\ell, \downarrow}_{n^{\ell}}, \u^{\ell, \downarrow}_{n^{\ell}+1} ) 	+  \cN^{\ell}_{0}  \mathbf{f}^{\ell}  = \u^{\ell, \uparrow}_{0}  + \cE^{\uparrow}_{\ell, \ell+1} \u^{\ell,\downarrow}_{1}.
\end{equation}
Given that the extrapolator is invertible (Proposition \ref{proposition:jump_condition} and Remark \ref{remark:invertibility}) we can multiply the equation above on the left by $\left( \cE^{\uparrow}_{\ell, \ell+1}\right)^{-1}$. Using the same computations performed in Lemma \ref{lemma:jump_condition} we have that
\begin{equation}
	\cG^{\uparrow, \ell}_{1}(\u^{\ell, \uparrow}_{n^{\ell}}, \u^{\ell, \uparrow}_{n^{\ell}+1} ) + \cG^{\downarrow, \ell}_{1}(\u^{\ell,\downarrow}_{0}, \u^{\ell, \downarrow}_{1} ) 	+ \cG^{\uparrow, \ell}_{1}(\u^{\ell, \downarrow}_{n^{\ell}}, \u^{\ell, \downarrow}_{n^{\ell}+1} ) 	+  \cN^{\ell}_{1}  \mathbf{f}^{\ell}  = \left( \cE^{\uparrow}_{\ell, \ell+1}\right)^{-1} \u^{\ell, \uparrow}_{0}  + \u^{\ell, \downarrow}_{1}.
\end{equation}
Moreover, by hypothesis $\underline{\underline{\u}}$ satisfies
\begin{equation}
	\cG^{\uparrow, \ell}_{1}(\u^{\ell, \uparrow}_{n^{\ell}}, \u^{\ell, \uparrow}_{n^{\ell}+1} ) + \cG^{\downarrow, \ell}_{1}(\u^{\ell,\downarrow}_{0}, \u^{\ell, \downarrow}_{1} ) 	+ \cG^{\uparrow, \ell}_{1}(\u^{\ell, \downarrow}_{n^{\ell}}, \u^{\ell, \downarrow}_{n^{\ell}+1} ) 	+  \cN^{\ell}_{1}  \mathbf{f}^{\ell}  = \u^{\ell, \uparrow}_{1} + \u^{\ell, \downarrow}_{1},
\end{equation}
which simplifies to
\begin{equation}
	\u^{\ell, \uparrow}_{1} = \left( \cE^{\uparrow}_{\ell, \ell+1}\right)^{-1} \u^{\ell, \uparrow}_{0}.
\end{equation}

Finally, using the fact that the extrapolator is invertible, we obtain the desired result. The proof for $i = n^{\ell}$ is analogous.

\end{proof}

\subsection*{Proof of Theorem \ref{thm:equivalence_integral_formulation} }
\begin{proof}
	Once again we start by proving the statement in 1D.
	One side of the equivalence is already given by Lemma \ref{lemma:discrete_GRF}.
	To show the other side, we need to show that the concatenated solution satisfy the solution at every point.

	Let $\underline{\u}$ be the solution to the discrete integral equation. We can then reconstruct the local solution at each subdomain by
	\begin{equation} \label{eq:discrete_GRF_1D}
		u^{\ell}_{k} =	\cG^{\uparrow, \ell}_{k}(u^{\ell}_{n^{\ell}}, u^{\ell}_{n^{\ell}+1} )  + \cG^{\downarrow, \ell}_{k}( u^{\ell}_{0}, u^{\ell}_{1} ) + \cN^{\ell}_{k}  f^{\ell}.
	\end{equation}
	We have
	\[ \left( H u^{\ell} \right)_i = f^{\ell}_i,  \]
	for $1<l<N^{\ell}$.
	To conclude we need to prove that the difference equation is satisfied at $i=1$ and at $i=n^{\ell}$ with the information coming from the neighboring sub-domains.
	We remark that Eq. \ref{eq:discrete_GRF_1D} is equivalent to solving
	\begin{equation} \label{eq:discrete_local_PDE}
		\left (H u^{\ell} \right )_i = f^{\ell}_i - \delta_{0,l}(\partial^{+}_x u^{\ell}_0 )+ (\partial^{+}_x \delta_{0,l} )u^{\ell}_0  + \delta_{n^{\ell}+1,i}(\partial^{+}_x u^{\ell}_{n^{\ell}+1}) - (\partial^{-}_x \delta_{N^{\ell}+1,i}) u^{\ell}_{n^{\ell}+1}
	\end{equation}
	where
	\begin{equation}
	 	\delta_{i,j} = \left \{ \begin{array}{rl}
	 							\frac{1}{h} & \text{ if } i = j  \\
	 							0 			& \text{ if } i \neq j
	 							\end{array} \right .
	 \end{equation}
	 and the up and down-wind derivatives were defined earlier.
	 Using the fact that $u^{\ell}_i = 0$ if $i = 0$ or $i = n^{\ell}$ (Eq. \ref{eq:lemma_jump}), and the equivalence between the Green's representation formula and the problem stated in Eq. \ref{eq:discrete_local_PDE}, we can apply $H$ to $u^{\ell}$ and evaluate it at $i = 1$ obtaining,
	 \begin{equation}
		\left (H u^{\ell} \right )_1 =  \frac{ 2u^{\ell}_1 - u^{\ell}_2 }{h^2} - m^{\ell}_1\omega^2 u^{\ell}_1 =  f^{\ell}_1  +  \frac{\delta_{1,l}}{h} u^{\ell}_0 = f^{\ell}_1  +  \frac{u^{\ell}_0}{h^2}.
	\end{equation}
	In other words,
	\begin{equation}
		 \frac{ -u^{\ell}_0 + 2u^{\ell}_1 - u^{\ell}_2 }{h^2} - m^{\ell}_1\omega^2 u^{\ell}_1  =  f^{\ell}_1,
	\end{equation}
	and by construction $u^{\ell}_0 = u^{\ell-1}_{n^{\ell-1}}$.
	This procedure can be replicated for $i=n^{\ell}$, yielding
	\begin{equation}
		 \frac{ -u^{\ell}_{n^{\ell}-1} + 2u^{\ell}_{n^{\ell}} - u^{\ell}_{n^{\ell}+1} }{h^2} - m^{\ell}_{n^{\ell}}\omega^2 u^{\ell}_{n^{\ell}} =  f^{\ell}_{n^{\ell}},
	\end{equation}
	in which we have $u^{\ell}_{n^{\ell}+1} = u^{\ell+1}_1$.
	This means that the concatenated solution satisfies the equation at the interfaces; therefore, it satisfies the difference equation everywhere. By construction it satisfies the boundary conditions; therefore, by uniqueness it is the solution to the difference equation.

	In 2D, we need to be especially careful with the PML.
	Using the same proof method as before, we have that
	\begin{equation} \label{eq:grf_local_appendix}
		\u^{\ell}_{k} =	\cG^{\uparrow, \ell}_{k}(\u^{\ell}_{n^{\ell}}, \u^{\ell}_{n^{\ell}+1} )  + \cG^{\downarrow, \ell}_{k}( \u^{\ell}_{0}, \u^{\ell}_{1} ) + \cN^{\ell}_{k}  \mathbf{f}^{\ell},
	\end{equation}
	satisfies the equation
	\[ \left( \mathbf{H} \u^{\ell} \right)_{i,j} = \mathbf{f}^{\ell}_{i,j}, \qquad \text{for } 1 < j < n^{\ell} \, \text{and} -n_{\text{pml}} + 1  < i < n^{\ell} +  n_{\text{pml}}  \]
	where $i$ and $j$ are local indices.
	Following the same reasoning as in the 1D case we have
	\begin{equation} \label{eq:append_global_pde}
	\left( \mathbf{H} \u^{\ell} \right)_{i,1} = \mathbf{f}^{\ell}_{i,1} - \frac{1}{h^2} \u^{\ell}_{i,0},  \qquad \text{for }  \text{and} -n_{\text{pml}} + 1  < i < n^{\ell} +  n_{\text{pml}}.
	\end{equation}
	To prove Eq. \ref{eq:append_global_pde} we use the fact that $\u^{\ell}_k$ is defined by Eq. \ref{eq:grf_local_appendix}, and by Lemma \ref{lemma:jump_condition}, $\u^{\ell}_k = 0 $ for  $k=0$ and $k = n^{\ell}+1$. Then, if we apply the global finite differences operator, $\mathbf{H}$, to the local $\u^{\ell}_k$, and to evaluate it at $k = 1$, we obtain
	\begin{align} \notag
	 	\left( \mathbf{H} \u^{\ell} \right)_{i,1} =	&  -\alpha_x(\x_{i,1})\frac{ \alpha_x(\x_{i+1/2,1})(\u_{i+1,1}^{\ell} - \u_{i,1}^{\ell}) - \alpha_x(\x_{i-1/2,1})( \u_{i,1}^{\ell} - \u_{i-1,1}^{\ell} )   }{h^2}\\
	 											&   +  \frac{1}{h^2} \left ( 2\u_{i,1}^{\ell} -\u_{i,2}^{\ell} \right) - \omega^2 m(\x_{i,1}) \\
	 										 =	& \, 	\mathbf{f}^{\ell}_{i,1} + \frac{1}{h^2} \u^{\ell}_{i,0}.
	 \end{align}
	 It is clear that the right-hand side has a similar form as in the 1D case. The concatenated solution satisfies the discretized PDE at the interfaces. Moreover, we can observe that by construction $\u^{\ell}$ satisfies the homogeneous Dirichlet boundary conditions, because the Green's functions satisfy the same boundary conditions. Furthermore, the traces at the interface also satisfy the zero boundary conditions at the endpoints. Then, the concatenated solution satisfy the finite difference equation inside the domain and the boundary conditions; therefore, by uniqueness it is solution to the discrete PDE.
\end{proof}

\bibliography{GRF_integral_formulations.bib}

\begin{thebibliography}{10}

\bibitem{Sivaram_Darve:HODLR}
S.~{A}mbikasaran and E.~{D}arve.
\newblock An $\mathcal{O}(n \log n)$ fast direct solver for partial
  hierarchically semi-separable matrices.
\newblock {\em Journal of Scientific Computing}, 57(3):477--501, December 2013.

\bibitem{Amestoy:MUMPS}
P.~R. Amestoy, I.~S. Duff, J.~Koster, and J.-Y. L'Excellent.
\newblock A fully asynchronous multifrontal solver using distributed dynamic
  scheduling.
\newblock {\em SIAM Journal on Matrix Analysis and Applications}, 23(1):15--41,
  2001.

\bibitem{babuska_melenk:partition_of_unity_method}
I.~Babuska and J.~M. Melenk.
\newblock The partition of unity method.
\newblock {\em International Journal for Numerical Methods in Engineering},
  40(4):727--758, 1997.

\bibitem{BallardDemmel:minimizing_communication_in_numerical_linear_algebra}
G.~Ballard, J.~Demmel, O.~Holtz, and O.~Schwartz.
\newblock Minimizing communication in numerical linear algebra.
\newblock {\em SIAM Journal on Matrix Analysis and Applications},
  32(3):866--901, 2011.

\bibitem{Bebendorf:2008}
M.~Bebendorf.
\newblock {\em Hierarchical Matrices: A Means to Efficiently Solve Elliptic
  Boundary Value Problems}, volume~63 of {\em Lecture Notes in Computational
  Science and Engineering (LNCSE)}.
\newblock Springer-Verlag, 2008.
\newblock ISBN 978-3-540-77146-3.

\bibitem{RosalieLaurent:compressed_PML}
R.~B\'elanger-Rioux and L.~Demanet.
\newblock Compressed absorbing boundary conditions via matrix probing.
\newblock {\em ArXiv e-prints}, [math.NA] 1401.4421, 2014.

\bibitem{BenamouDespres:domain_decomposition}
J.-D. Benamou and B.~Despr\'es.
\newblock A domain decomposition method for the {H}elmholtz equation and
  related optimal control problems.
\newblock {\em Journal of Computational Physics}, 136(1):68--82, 1997.

\bibitem{Berenger:PML}
J.-P. B\'erenger.
\newblock A perfectly matched layer for the absorption of electromagnetic
  waves.
\newblock {\em Journal of Computational Physics}, 114(2):185--200, 1994.

\bibitem{Berkhout:Seismic_Migration_imaging_of_acoustic_energy_by_wavefield_extrapolation}
A.~J. Berkhout.
\newblock {\em Seismic Migration: Imaging of Acoustic Energy by Wave Field
  Extrapolation}.
\newblock Elsevier, 1980.

\bibitem{Bevilacqua:parallel_solution_of_block_tridiagonal_linear_systems}
R.~Bevilacqua, B.~Codenotti, and F.~Romani.
\newblock Parallel solution of block tridiagonal linear systems.
\newblock {\em Linear Algebra and its Applications}, 104(0):39--57, 1988.

\bibitem{Beylkin:wave_propagation_using_bases_for_bandlimited_functions}
G.~Beylkin and K.~Sandberg.
\newblock Wave propagation using bases for bandlimited functions.
\newblock {\em Wave Motion}, 41(3):263--291, 2005.
\newblock Special Issue in Honor of the 75th Birthday of A.T.de Hoop.

\bibitem{BP_model}
F.~Billette and S.~Brandsberg-Dahl.
\newblock {\em The 2004 {BP} velocity benchmark.}
\newblock EAGE, 2005.

\bibitem{Grote_Schenk:algebraic_multilever_preconditioner_Helmholtz_equation}
M.~Bollh\"ofer, M.~Grote, and O.~Schenk.
\newblock Algebraic multilevel preconditioner for the {H}elmholtz equation in
  heterogeneous media.
\newblock {\em SIAM Journal on Scientific Computing}, 31(5):3781--3805, 2009.

\bibitem{Brandt_Livshits:multi_ray_multigrid_standing_wave_equations}
A.~Brandt and I.~Livshits.
\newblock Wave-ray multigrid method for standing wave equations.
\newblock {\em Electronic Transactions on Numerical Analysis}, 6:162--181,
  1997.

\bibitem{Bruno_Turc:Electromagnetic_integral_equations_requiring_small_numbers_of_krylov_subspace_iterations}
O.~Bruno, T.~Elling, R.~Paffenroth, and C.~Turc.
\newblock Electromagnetic integral equations requiring small numbers of
  {K}rylov-subspace iterations.
\newblock {\em Journal of Computational Physics}, 228(17):6169--6183, September
  2009.

\bibitem{Calandra_Grattonn:an_improved_two_grid_preconditioner_for_the_solution_of_3d_Helmholtz}
H.~Calandra, S.~Gratton, X.~Pinel, and X.~Vasseur.
\newblock An improved two-grid preconditioner for the solution of
  three-dimensional {H}elmholtz problems in heterogeneous media.
\newblock {\em Numerical Linear Algebra with Applications}, 20(4):663--688,
  2013.

\bibitem{Candes_Demanet_Ying:A_Fast_Butterfly_Algorithm_for_the_Computation_of_Fourier_Integral_Operators}
E.~Cand\`es, L.~Demanet, and L.~Ying.
\newblock A fast butterfly algorithm for the computation of fourier integral
  operators.
\newblock {\em Multiscale Modeling \& Simulation}, 7(4):1727--1750, 2009.

\bibitem{Cessenat:Application_dune_nouvelle_formulation_variationelle_aux_equations_dondes_harmonique}
O.~Cessenat.
\newblock {\em Application d'une nouvelle formulation variationnelle aux
  \'equations d'ondes harmoniques}.
\newblock PhD thesis, Universit\'e Paris IX Dauphine, Paris, France, 1996.

\bibitem{Cessenat_Despres:application_of_an_ultra_weak_variational_formulation_of_elliptic_pdes_to_the_2_d_helmholtz_problem}
O.~Cessenat and B.~Despr\'es.
\newblock Application of an ultra weak variational formulation of elliptic pdes
  to the two-dimensional {H}elmholtz problem.
\newblock {\em SIAM Journal on Numerical Analysis}, 35(1):255--299, 1998.

\bibitem{Cessenat_Despres:Using_Plane_Waves_as_Base_Functions_for_Solving_Time_Harmonic_Equations_with_the_Ultra_Weak_Variational_Formulation}
O.~Cessenat and B.~Despr\'es.
\newblock Using plane waves as base functions for solving time harmonic
  equations with the ultra weak variational formulation.
\newblock {\em Journal of Computational Acoustics}, 11(02):227--238, 2003.

\bibitem{Chandler-Wilde:Boundary_value_problems_for_the_Helmholtz_equation_in_a_half-plane}
S.~N. Chandler-Wilde.
\newblock Boundary value problems for the {H}elmholtz equation in a half-plane.
\newblock In {\em Proceedings of the 3rd International Conference on
  Mathematical and Numerical Aspects of Wave Propagation}, page 188–197,
  1995.

\bibitem{Chandler_Graham_Langdon_Spence:Numerical_asymptotic_boundary_integral_methods_in_high-frequency_acoustic_scattering}
S.~N. Chandler-Wilde, I.~G. Graham, S.~Langdon, and E.~A. Spence.
\newblock Numerical-asymptotic boundary integral methods in high-frequency
  acoustic scattering.
\newblock {\em Acta Numerica}, 21:89--305, 5 2012.

\bibitem{Chen:A_dispersion_minimizing_finite_difference_scheme_and_preconditioned_solver_for_the_3D_Helmholtz_equation}
Z.~Chen, D.~Cheng, and T.~Wu.
\newblock A dispersion minimizing finite difference scheme and preconditioned
  solver for the {3D} {H}elmholtz equation.
\newblock {\em Journal of Computational Physics}, 231(24):8152--8175, 2012.

\bibitem{Chen_Xiang:a_source_transfer_ddm_for_helmholtz_equations_in_unbounded_domain}
Z.~Chen and X.~Xiang.
\newblock A source transfer domain decomposition method for {H}elmholtz
  equations in unbounded domain.
\newblock {\em SIAM Journal on Numerical Analysis}, 51(4):2331--2356, 2013.

\bibitem{Cheng_Xiang:A_Source_Transfer_Domain_Decomposition_Method_For_Helmholtz_Equations_in_Unbounded_Domain_Part_II_Extensions}
Z.~Chen and X.~Xiang.
\newblock A source transfer domain decomposition method for {H}elmholtz
  equations in unbounded domain part {II}: Extensions.
\newblock {\em Numerical Mathematics: Theory, Methods and Applications},
  6:538--555, 2013.

\bibitem{Nataf:A_coarse_space_for_heterogeneous_Helmholtz_problems_based_on_the_Dirichlet-to-Neumann_operator}
L.~Conen, V.~Dolean, Krause R., and F.~Nataf.
\newblock A coarse space for heterogeneous helmholtz problems based on the
  dirichlet-to-neumann operator.
\newblock {\em Journal of Computational and Applied Mathematics}, 271(0):83 --
  99, 2014.

\bibitem{Cools_Reps_Vanroose:A_new_level-dependent_coarse_grid_correction_scheme_for_indefinite_Helmholtz_problems}
S.~Cools, B.~Reps, and W.~Vanroose.
\newblock A new level-dependent coarse grid correction scheme for indefinite
  {H}elmholtz problems.
\newblock {\em Numerical Linear Algebra with Applications}, 21(4):513--533,
  2014.

\bibitem{Cools_Vanroose:Local_Fourier_analysis_of_the_complex_shifted_Laplacian_preconditioner_for_Helmholtz_problems}
S.~Cools and W.~Vanroose.
\newblock Local {F}ourier analysis of the complex shifted {L}aplacian
  preconditioner for {H}elmholtz problems.
\newblock {\em Numerical Linear Algebra with Applications}, 20(4):575--597,
  2013.

\bibitem{Cools_Vanroose:Generalization_of_the_complex_shifted_Laplacian:_on_the_class_of_expansion_preconditioners_for_Helmholtz_problems}
S.~Cools and W.~Vanroose.
\newblock Generalization of the complex shifted {L}aplacian: on the class of
  expansion preconditioners for {H}elmholtz problems.
\newblock {\em ArXiv e-prints}, 2015.

\bibitem{Davis:UMFPACK}
T.~A. Davis.
\newblock Algorithm 832: {UMFPACK} v4.3---an unsymmetric-pattern multifrontal
  method.
\newblock {\em ACM Transactions on Mathematical Software}, 30(2):196--199, June
  2004.

\bibitem{Wang:H_multifrontal}
M.~V. de~Hoop, S.~Wang, and J.~Xia.
\newblock On 3{D} modeling of seismic wave propagation via a structured
  parallel multifrontal direct {H}elmholtz solver.
\newblock {\em Geophysical Prospecting}, 59(5):857--873, 2011.

\bibitem{deHoop:generalization_of_the_phase_screen_approximation_for_the_scattering_of_acoustic_waves}
M.V. de~Hoop, J.~H. Le~{R}ousseau, and R-S. Wu.
\newblock Generalization of the phase-screen approximation for the scattering
  of acoustic waves.
\newblock {\em Wave Motion}, 31(1):43--70, 2000.

\bibitem{DemanetYing:FIO}
L.~Demanet and L.~Ying.
\newblock Fast wave computation via {F}ourier integral operators.
\newblock {\em Mathematics of Computation}, 81(279):1455--1486, 2012.

\bibitem{DemmelGrigori:communication_avoiding_rank_revealing_qr_factorization_with_column_pivoting}
J.~Demmel, L.~Grigori, M.~Gu, and H.~Xiang.
\newblock Communication avoiding rank revealing qr factorization with column
  pivoting.
\newblock Technical Report UCB/EECS-2013-46, EECS Department, University of
  California, Berkeley, May 2013.

\bibitem{Despres:domain_decomposition_hemholtz}
B.~Despr\'es.
\newblock D\'ecomposition de domaine et probl\`eme de {H}elmholtz.
\newblock {\em Comptes rendus de l'Acad\'emie des sciences. S\'erie 1,
  Math\'ematique}, 311:313--316, 1990.

\bibitem{Duff_Reid:The_Multifrontal_Solution_of_Indefinite_Sparse_Symmetric_Linear}
I.~S. Duff and J.~K. Reid.
\newblock The multifrontal solution of indefinite sparse symmetric linear.
\newblock {\em ACM Trans. Math. Softw.}, 9(3):302--325, September 1983.

\bibitem{Elman_Ernst:a_multigrid_method_discrete_helmholtz_equation}
H.~Elman, O.~Ernst, and D.~O'Leary.
\newblock A multigrid method enhanced by {K}rylov subspace iteration for
  discrete {H}elmholtz equations.
\newblock {\em SIAM Journal on Scientific Computing}, 23(4):1291--1315, 2001.

\bibitem{EngquistYing:Sweeping_H}
B.~Engquist and L.~Ying.
\newblock Sweeping preconditioner for the {H}elmholtz equation: Hierarchical
  matrix representation.
\newblock {\em Communications on Pure and Applied Mathematics}, 64(5):697--735,
  2011.

\bibitem{EngquistYing:Sweeping_PML}
B.~Engquist and L.~Ying.
\newblock Sweeping preconditioner for the {H}elmholtz equation: moving
  perfectly matched layers.
\newblock {\em Multiscale Modeling \& Simulation}, 9(2):686--710, 2011.

\bibitem{Erlangga:Helmholtz}
Y.~A. Erlangga.
\newblock Advances in iterative methods and preconditioners for the {H}elmholtz
  equation.
\newblock {\em Archives of Computational Methods in Engineering}, 15(1):37--66,
  2008.

\bibitem{Erlangga:shifted_laplacian}
Y.~A. Erlangga, C.~W. Oosterlee, and C.~Vuik.
\newblock A novel multigrid based preconditioner for heterogeneous {H}elmholtz
  problems.
\newblock {\em SIAM Journal on Scientific Computing}, 27(4):1471--1492, 2006.

\bibitem{Gander:why_is_difficult_to_solve_helmholtz_problems_with_classical_iterative_methods}
O.~G. Ernst and M.~J. Gander.
\newblock Why it is difficult to solve {H}elmholtz problems with classical
  iterative methods.
\newblock In Ivan~G. Graham, Thomas~Y. Hou, Omar Lakkis, and Robert Scheichl,
  editors, {\em Numerical Analysis of Multiscale Problems}, volume~83 of {\em
  Lecture Notes in Computational Science and Engineering}, pages 325--363.
  Springer Berlin Heidelberg, 2012.

\bibitem{Farhat:The_discontinuous_enrichment_method}
C.~Farhat, I.~Harari, and L.~P. Franca.
\newblock The discontinuous enrichment method.
\newblock {\em Computer Methods in Applied Mechanics and Engineering},
  190(48):6455--6479, 2001.

\bibitem{FishmandeHoop:exact_constructions_of_square_root_helmholtz_operator_symbols_the_focusing_quadratic_profile}
L.~Fishman, M.~V. de~Hoop, and M.~J.~N. van Stralen.
\newblock Exact constructions of square-root {H}elmholtz operator symbols: The
  focusing quadratic profile.
\newblock {\em Journal of Mathematical Physics}, 41(7):4881--4938, 2000.

\bibitem{GanderNataf:ailu_for_hemholtz_problems_a_new_preconditioner_based_on_an_analytic_factorization}
M.~Gander and F.~Nataf.
\newblock {AILU} for {H}elmholtz problems: a new preconditioner based on an
  analytic factorization.
\newblock {\em Comptes Rendus de l'Acad\'emie des Sciences-Series
  I-Mathematics}, 331(3):261--266, 2000.

\bibitem{Gander_Kwok:optimal_interface_conditiones_for_an_arbitrary_decomposition_into_subdomains}
M.~J. Gander and F.~Kwok.
\newblock Optimal interface conditions for an arbitrary decomposition into
  subdomains.
\newblock In Yunqing Huang, Ralf Kornhuber, Olof Widlund, and Jinchao Xu,
  editors, {\em Domain Decomposition Methods in Science and Engineering XIX},
  volume~78 of {\em Lecture Notes in Computational Science and Engineering},
  pages 101--108. Springer Berlin Heidelberg, 2011.

\bibitem{Gander_Nataf:AILU_for_helmholtz_problems_a_new_preconditioner_based_on_the_analytic_parabolic_factorization}
M.~J. Gander and F.~Nataf.
\newblock {AILU} for {H}elmholtz problems: A new preconditioner based on the
  analytic parabolic factorization.
\newblock {\em Journal of Computational Acoustics}, 09(04):1499--1506, 2001.

\bibitem{Gander_Graham_Spence:largest_shift_for_complex_shitfted_Laplacian}
M.J. Gander, I.G. Graham, and E.A. Spence.
\newblock Applying {GMRES} to the {H}elmholtz equation with shifted {L}aplacian
  preconditioning: what is the largest shift for which wavenumber-independent
  convergence is guaranteed?
\newblock {\em Numerische Mathematik}, pages 1--48, 2015.

\bibitem{GeorgeNested_dissection}
A.~George.
\newblock Nested dissection of a regular finite element mesh.
\newblock {\em SIAM Journal on Numerical Analysis}, 10:345--363, 1973.

\bibitem{Gillman_Barnett_Martinsson:A_spectrally_accurate_solution_technique_for_frequency_domain_scattering_problems_with_variable_media}
A.~Gillman, A.H. Barnett, and P.G. Martinsson.
\newblock A spectrally accurate direct solution technique for frequency-domain
  scattering problems with variable media.
\newblock {\em BIT Numerical Mathematics}, pages 1--30, 2014.

\bibitem{Gittelson_Hipmair_Perugia:Trefftz}
C.~J. Gittelson, R.~Hiptmair, and I.~Perugia.
\newblock Plane wave discontinuous {G}alerkin methods: Analysis of the
  h-version.
\newblock {\em ESAIM: Mathematical Modelling and Numerical Analysis},
  43:297--331, 3 2009.

\bibitem{Haber:A_fast_method_for_the_solution_of_the_Helmholtz_equation}
E.~Haber and S.~MacLachlan.
\newblock A fast method for the solution of the {H}elmholtz equation.
\newblock {\em Journal of Computational Physics}, 230(12):4403--4418, 2011.

\bibitem{Hiptmair:multiple_traces_boundary_integral_formulation_for_hemholtz_transmission_problem}
R.~Hiptmair and C.~Jerez-Hanckes.
\newblock Multiple traces boundary integral formulation for {H}elmholtz
  transmission problems.
\newblock {\em Advances in Computational Mathematics}, 37(1):39--91, 2012.

\bibitem{Hoffman:Complexity_Bounds_for_Regular_Finite_Difference_and_Finite_Element_Grids}
A.~Hoffman, M.~Martin, and D.~Rose.
\newblock Complexity bounds for regular finite difference and finite element
  grids.
\newblock {\em SIAM Journal on Numerical Analysis}, 10(2):364--369, 1973.

\bibitem{Xia:multifrontal}
M.~Gu J.~Xia, S.~Chandrasekaran and X.~S. Li.
\newblock Superfast multifrontal method for large structured linear systems of
  equations.
\newblock {\em SIAM Journal on Matrix Analysis and Applications},
  31(3):1382--1411, 2010.

\bibitem{Johnson:PML}
S.~Johnson.
\newblock Notes on perfectly matched layers ({PML}s), March 2010.

\bibitem{Jones_Ma_Rokhlin:A_Fast_Direct_Algorithm_for_the_Solution_of_the_Laplace_Equation_on_Regions_with_Fractal_Boundaries}
P.~Jones, J.~Ma, and V.~Rokhlin.
\newblock A fast direct algorithm for the solution of the {L}aplace equation on
  regions with fractal boundaries.
\newblock {\em Journal of Computational Physics}, 113(1):35--51, 1994.

\bibitem{Kress:Linear_integral_equations}
R.~Kress.
\newblock Linear integral equations.
\newblock In {\em Applied {M}athematical {S}ciences}. Springer, 1999.

\bibitem{Schenk:pardiso1}
A.~Kuzmin, M.~Luisier, and O.~Schenk.
\newblock Fast methods for computing selected elements of the greens function
  in massively parallel nanoelectronic device simulations.
\newblock In F.~Wolf, B.~Mohr, and D.~Mey, editors, {\em Euro-Par 2013 Parallel
  Processing}, volume 8097 of {\em Lecture Notes in Computer Science}, pages
  533--544. Springer Berlin Heidelberg, 2013.

\bibitem{Li_Yang_Martin_Ho_Ying:Butterfly_Factorization}
Y.~Li, H.~Yang, E.~Martin, K.~Ho, and L.~Ying.
\newblock Butterfly factorization.
\newblock {\em ArXiv e-prints}, 2015.

\bibitem{Lin_Ying:Fast_Construction_of_Hierarchical_Matrix_Representation_from_Matrix_vector_Multiplication}
L.~Lin, J.~Lu, and L.~Ying.
\newblock Fast construction of hierarchical matrix representation from
  matrix-vector multiplication.
\newblock {\em Journal of Computational Physics}, 230(10):4071--4087, May 2011.

\bibitem{Liu_Ying:Additive_Sweeping_Preconditioner_for_the_Helmholtz_Equation}
F.~Liu and L.~Ying.
\newblock Additive sweeping preconditioner for the {H}elmholtz equation.
\newblock {\em ArXiv e-prints}, 2015.

\bibitem{Liu_Ying:Recursive_sweeping_preconditioner_for_the_3d_helmholtz_equation}
F.~Liu and L.~Ying.
\newblock Recursive sweeping preconditioner for the 3{D} {H}elmholtz equation.
\newblock {\em ArXiv e-prints}, 2015.

\bibitem{Luo:Fast_Huygens_sweeping_methods_for_Helmholtz_equations_in_inhomogeneous_media_in_the_high_frequency_regime}
S.~Luo, J.~Qian, and R.~Burridge.
\newblock Fast {H}uygens sweeping methods for {H}elmholtz equations in
  inhomogeneous media in the high frequency regime.
\newblock {\em Journal of Computational Physics}, 270(0):378--401, 2014.

\bibitem{Marmousi_2}
G.~Martin, R.~Wiley, and K.~Marfurt.
\newblock An elastic upgrade for {M}armousi.
\newblock {\em The Leading Edge, Society for Exploration Geophysics}, 25, 2006.

\bibitem{MartinssonRohklin:a_fast_direct_solver_for_scattering_problems_involving_elongated_structures}
P.G. Martinsson and V.~Rokhlin.
\newblock A fast direct solver for scattering problems involving elongated
  structures.
\newblock {\em Journal of Computational Physics}, 221(1):288--302, 2007.

\bibitem{MartinssonRokhlin:A_randomized_algorithm_for_the_decomposition_of_matrices}
P.G. Martinsson, V.~Rokhlin, and M~Tygert.
\newblock A randomized algorithm for the decomposition of matrices.
\newblock {\em Applied and Computational Harmonic Analysis}, 30(1):47--68,
  2011.

\bibitem{McLean:Strongly_elliptic_systems_and_boundary_integral_equations}
W.~Mc{L}ean.
\newblock {\em Strongly Elliptic Systems and Boundary Integral Equations}.
\newblock Cambridge University Press, 2000.

\bibitem{Meurant:a_review_on_the_inverse_of_tridiagonal_matrices}
G.~Meurant.
\newblock A review on the inverse of symmetric tridiagonal and block
  tridiagonal matrices.
\newblock {\em SIAM Journal on Matrix Analysis and Applications},
  13(3):707--728, July 1992.

\bibitem{Perugia:trefft}
A.~Moiola, R.~Hiptmair, and I.~Perugia.
\newblock Plane wave approximation of homogeneous {H}elmholtz solutions.
\newblock {\em Zeitschrift f\"ur angewandte Mathematik und Physik},
  62(5):809--837, 2011.

\bibitem{Moiola_Spence:Is_the_Helmholtz_Equation_Really_Sign_Indefinite}
A.~Moiola and E.~Spence.
\newblock Is the {H}elmholtz equation really sign-indefinite?
\newblock {\em SIAM Review}, 56(2):274--312, 2014.

\bibitem{Monk_Wang:A_least-squares_method_for_the_Helmholtz_equation:}
P.~Monk and D.-Q. Wang.
\newblock A least-squares method for the {H}elmholtz equation.
\newblock {\em Computer Methods in Applied Mechanics and Engineering},
  175(1–2):121--136, 1999.

\bibitem{Neil_Rokhlin:butterfly}
M.~O'Neil, F.~Woolfe, and V.~Rokhlin.
\newblock An algorithm for the rapid evaluation of special function transforms.
\newblock {\em Applied and Computational Harmonic Analysis}, 28(2):203 -- 226,
  2010.
\newblock Special Issue on Continuous Wavelet Transform in Memory of Jean
  Morlet, Part I.

\bibitem{Plessix:A_Helmholtz_iterative_solver_for_3D_seismic_imaging_problems}
R.-E. Plessix.
\newblock A {H}elmholtz iterative solver for {3D} seismic-imaging problems.
\newblock {\em Geophysics}, 72(5):SM185--SM194, 2007.

\bibitem{Plessix_Mulder:Separation_of_variable_preconditioner_for_iterativa_Helmholtz_solver}
R.-E. Plessix and W.~A. Mulder.
\newblock Separation-of-variables as a preconditioner for an iterative
  {H}elmholtz solver.
\newblock {\em Applied Numerical Mathematics}, 44(3):385--400, 2003.

\bibitem{Poulson_Demanet:a_parallel_butterfly_algorithm}
J.~Poulson, L.~Demanet, N.~Maxwell, and L.~Ying.
\newblock A parallel butterfly algorithm.
\newblock {\em SIAM Journal on Scientific Computing}, 36(1):C49--C65, 2014.

\bibitem{Poulson_Engquist:a_parallel_sweeping_preconditioner_for_heteregeneous_3d_helmholtz}
J.~Poulson, B.~Engquist, S.~Li, and L.~Ying.
\newblock A parallel sweeping preconditioner for heterogeneous {3D} {H}elmholtz
  equations.
\newblock {\em SIAM Journal on Scientific Computing}, 35(3):C194--C212, 2013.

\bibitem{Qian_Luo_Burridge:Fast_Huygens_sweeping_methods_for_multiarrival_Greens_functions_of_Helmholtz_equations_in_the_high-frequency_regime}
J.~Qian, S.~Luo, and R.~Burridge.
\newblock Fast {H}uygens' sweeping methods for multiarrival {G}reen's functions
  of {H}elmholtz equations in the high-frequency regime.
\newblock {\em Geophysics}, 80(2):T91--T100, 2015.

\bibitem{Riyanti_Plessix:A_parallel_multigrid-based_preconditioner_for_the_3D_heterogeneous_high-frequency_Helmholtz_equation:}
C.~D. Riyanti, A.~Kononov, Y.~A. Erlangga, C.~Vuik, C.~W. Oosterlee, R.-E.
  Plessix, and W.~A. Mulder.
\newblock A parallel multigrid-based preconditioner for the {3D} heterogeneous
  high-frequency {H}elmholtz equation.
\newblock {\em Journal of Computational Physics}, 224(1):431--448, 2007.
\newblock Special Issue Dedicated to Professor Piet Wesseling on the occasion
  of his retirement from Delft University of Technology.

\bibitem{Rokhlin:rapid_solution_of_integral_equations_of_classical_potential_theory}
V.~Rokhlin.
\newblock Rapid solution of integral equations of classical potential theory.
\newblock {\em Journal of Computational Physics}, 60(2):187--207, 1985.

\bibitem{Rouet_Li_Ghysels:A_distributed-memory_package_for_dense_Hierarchically_Semi-Separable_matrix_computations_using_randomization}
F.-H. Rouet, X.~S. Li, P.~Ghysels, and A.~Napov.
\newblock A distributed-memory package for dense {H}ierarchically
  {S}emi-{S}eparable matrix computations using randomization.
\newblock {\em ArXiv e-prints}, 2015.

\bibitem{Saad:iterative_methods_for_sparse_linear_systems}
Y.~Saad.
\newblock {\em Iterative Methods for Sparse Linear Systems}.
\newblock Society for Industrial and Applied Mathematics, second edition, 2003.

\bibitem{Beylkin:oneway}
K.~Sandberg and G.~Beylkin.
\newblock Full-wave-equation depth extrapolation for migration.
\newblock {\em Geophysics}, 74(6):WCA121--CA128, 2009.

\bibitem{Sheikh_Lahaye_Vuik:On_the_convergence_of_shifted_Laplace_preconditioner_combined_with_multilevel_deflation}
A.~H. Sheikh, D.~Lahaye, and C.~Vuik.
\newblock On the convergence of shifted {L}aplace preconditioner combined with
  multilevel deflation.
\newblock {\em Numerical Linear Algebra with Applications}, 20(4):645--662,
  2013.

\bibitem{Stoffa:split_step_fourier_migration}
P.~L. Stoffa, J.~T. Fokkema, R.~M. de~Luna~Freire, and W.~P. Kessinger.
\newblock Split-step fourier migration.
\newblock {\em Geophysics}, 55(4):410--421, 1990.

\bibitem{CStolk_rapidily_converging_domain_decomposition}
C.~Stolk.
\newblock A rapidly converging domain decomposition method for the {H}elmholtz
  equation.
\newblock {\em Journal of Computational Physics}, 241(0):240--252, 2013.

\bibitem{StolkdeHoop:modeling_of_seismic_data_in_the_downward_continuation_approach}
C.~Stolk and M.~de~Hoop.
\newblock Modeling of seismic data in the downward continuation approach.
\newblock {\em SIAM Journal on Applied Mathematics}, 65(4):1388--1406, 2005.

\bibitem{Bhowmik_Stolk:A_multigrid_method_for_the_helmholtz_equation_with_optimized_coarse_grid_corrections}
C.~C. {Stolk}, M.~{Ahmed}, and S.~K. {Bhowmik}.
\newblock {A multigrid method for the {H}elmholtz equation with optimized
  coarse grid corrections}.
\newblock {\em ArXiv e-prints}, [math.NA] 1304.4103, April 2013.

\bibitem{Strang:the_interplay_of_ranks_of_submatrices}
G.~Strang and T.~Nguyen.
\newblock The interplay of ranks of submatrices.
\newblock {\em SIAM Review}, 46(4):637--646, 2004.

\bibitem{Thakur:optimization_of_collective_communication_operations_in_mpich}
R.~Thakur, R.~Rabenseifner, and W.~Gropp.
\newblock Optimization of collective communication operations in mpich.
\newblock {\em International Journal of High Performance Computing
  Applications}, 19(1):49--66, 2005.

\bibitem{Tsuji_engquist_Ying:A_sweeping_preconditioner_for_time-harmonic_Maxwells_equations_with_finite_elements}
P.~Tsuji, B.~Engquist, and L.~Ying.
\newblock A sweeping preconditioner for time-harmonic {M}axwell's equations
  with finite elements.
\newblock {\em Journal of Computational Physics}, 231(9):3770 -- 3783, 2012.

\bibitem{Tsuji_Poulson:sweeping_preconditioners_for_elastic_wave_propagation}
P.~Tsuji, J.~Poulson, B.~Engquist, and L.~Ying.
\newblock Sweeping preconditioners for elastic wave propagation with spectral
  element methods.
\newblock {\em ESAIM: Mathematical Modelling and Numerical Analysis},
  48:433--447, 3 2014.

\bibitem{Tsuji_Ying:A_sweeping_preconditioner_for_Yees_finite_difference_approximation_of_time-harmonic_Maxwells_equations}
P.~Tsuji and L.~Ying.
\newblock A sweeping preconditioner for {Y}ee's finite difference approximation
  of time-harmonic {M}axwell's equations.
\newblock {\em Frontiers of Mathematics in China}, 7(2):347--363, 2012.

\bibitem{Usmani:inversion_of_jacobi_tridiagonal_matrix}
R.~A. Usmani.
\newblock Inversion of {J}acobi's tridiagonal matrix.
\newblock {\em Computers \& Mathematics with Applications}, 27(8):59--66, 1994.

\bibitem{Vion_Rosalie_Demanet:A_DDM_double_sweep_preconditioner_for_the_Helmholtz_equation_with_matrix_probing_of_the_DtN_map}
A.~Vion, R.~B\'elanger-Rioux, L.~Demanet, and C.~Geuzaine.
\newblock A {DDM} double sweep preconditioner for the {H}elmholtz equation with
  matrix probing of the dtn map.
\newblock In {\em Mathematical and Numerical Aspects of Wave Propagation WAVES
  2013}, 2013.

\bibitem{GeuzaineVion:double_sweep}
A.~Vion and C.~Geuzaine.
\newblock Double sweep preconditioner for optimized {S}chwarz methods applied
  to the {H}elmholtz problem.
\newblock {\em Journal of Computational Physics}, 266(0):171--190, 2014.

\bibitem{Wang_Li_Sia_Situ_Hoop:Efficient_Scalable_Algorithms_for_Solving_Dense_Linear_Systems_with_Hierarchically_Semiseparable_Structures}
S.~Wang, X.~S. Li, Xia J., Y.~Situ, and M.~V. de~Hoop.
\newblock Efficient scalable algorithms for solving dense linear systems with
  hierarchically semiseparable structures.
\newblock {\em SIAM Journal on Scientific Computing}, 35(6):C519--C544, 2013.

\bibitem{Zepeda_Hewett_Demanet:Preconditioning_the_2D_Helmholtz_equation_with_polarized_traces}
L.~Zepeda-N\'u{\~n}ez, R.~J. Hewett, and L.~Demanet.
\newblock Preconditioning the {2D} {H}elmholtz equation with polarized traces.
\newblock In {\em SEG Technical Program Expanded Abstracts 2014}, pages
  3465--3470, 2014.

\end{thebibliography}

\end{document}